\documentclass[12pt,twoside]{article}
\usepackage{amsmath,amssymb}
\usepackage{amsthm}
\usepackage{upref}
\usepackage{citeref}
\numberwithin{equation}{section}

\theoremstyle{plain}
\newtheorem{theorem}{Theorem}[section]
\newtheorem{proposition}[theorem]{Proposition}
\newtheorem{definition}[theorem]{Definition}
\newtheorem{corollary}[theorem]{Corollary}

\theoremstyle{definition}
\newtheorem{remark}[theorem]{Remark}
\newtheorem{problem}[theorem]{Problem}

\begin{document}

\newcommand{\eq}{equation}
\newcommand{\real}{\ensuremath{\mathbb R}}
\newcommand{\comp}{\ensuremath{\mathbb C}}
\newcommand{\rn}{\ensuremath{{\mathbb R}^n}}
\newcommand{\tn}{\ensuremath{{\mathbb T}^n}}
\newcommand{\rnp}{\ensuremath{\real^n_+}}
\newcommand{\rnn}{\ensuremath{\real^n_-}}
\newcommand{\Rn}{\ensuremath{{\mathbb R}^{n-1}}}
\newcommand{\no}{\ensuremath{\nat_0}}
\newcommand{\ganz}{\ensuremath{\mathbb Z}}
\newcommand{\zn}{\ensuremath{{\mathbb Z}^n}}
\newcommand{\zom}{\ensuremath{{\mathbb Z}_{\Om}}}
\newcommand{\zOm}{\ensuremath{{\mathbb Z}^{\Om}}}
\newcommand{\As}{\ensuremath{A^s_{p,q}}}
\newcommand{\Bs}{\ensuremath{B^s_{p,q}}}
\newcommand{\Fs}{\ensuremath{F^s_{p,q}}}
\newcommand{\Fsr}{\ensuremath{F^{s,\rloc}_{p,q}}}
\newcommand{\nat}{\ensuremath{\mathbb N}}
\newcommand{\Om}{\ensuremath{\Omega}}
\newcommand{\di}{\ensuremath{{\mathrm d}}}
\newcommand{\sn}{\ensuremath{{\mathbb S}^{n-1}}}
\newcommand{\Ac}{\ensuremath{\mathcal A}}
\newcommand{\Acs}{\ensuremath{\Ac^s_{p,q}}}
\newcommand{\Bc}{\ensuremath{\mathcal B}}
\newcommand{\Cc}{\ensuremath{\mathcal C}}
\newcommand{\cc}{{\scriptsize $\Cc$}${}^s (\rn)$}
\newcommand{\ccd}{{\scriptsize $\Cc$}${}^s (\rn, \delta)$}
\newcommand{\Fc}{\ensuremath{\mathcal F}}
\newcommand{\Lc}{\ensuremath{\mathcal L}}
\newcommand{\Mc}{\ensuremath{\mathcal M}}
\newcommand{\Ec}{\ensuremath{\mathcal E}}
\newcommand{\Pc}{\ensuremath{\mathcal P}}
\newcommand{\Efr}{\ensuremath{\mathfrak E}}
\newcommand{\Mfr}{\ensuremath{\mathfrak M}}
\newcommand{\Mbf}{\ensuremath{\mathbf M}}
\newcommand{\Dbb}{\ensuremath{\mathbb D}}
\newcommand{\Lbb}{\ensuremath{\mathbb L}}
\newcommand{\Pbb}{\ensuremath{\mathbb P}}
\newcommand{\Qbb}{\ensuremath{\mathbb Q}}
\newcommand{\Rbb}{\ensuremath{\mathbb R}}
\newcommand{\vp}{\ensuremath{\varphi}}
\newcommand{\hra}{\ensuremath{\hookrightarrow}}
\newcommand{\supp}{\ensuremath{\mathrm{supp \,}}}
\newcommand{\ssupp}{\ensuremath{\mathrm{sing \ supp\,}}}
\newcommand{\dist}{\ensuremath{\mathrm{dist \,}}}
\newcommand{\unif}{\ensuremath{\mathrm{unif}}}
\newcommand{\ve}{\ensuremath{\varepsilon}}
\newcommand{\vk}{\ensuremath{\varkappa}}
\newcommand{\vr}{\ensuremath{\varrho}}
\newcommand{\pa}{\ensuremath{\partial}}
\newcommand{\oa}{\ensuremath{\overline{a}}}
\newcommand{\ob}{\ensuremath{\overline{b}}}
\newcommand{\of}{\ensuremath{\overline{f}}}
\newcommand{\LA}{\ensuremath{L^r\!\As}}
\newcommand{\LcA}{\ensuremath{\Lc^{r}\!A^s_{p,q}}}
\newcommand{\LcdA}{\ensuremath{\Lc^{r}\!A^{s+d}_{p,q}}}
\newcommand{\LcB}{\ensuremath{\Lc^{r}\!B^s_{p,q}}}
\newcommand{\LcF}{\ensuremath{\Lc^{r}\!F^s_{p,q}}}
\newcommand{\Lb}{\ensuremath{L^r b^s_{p,q}}}
\newcommand{\Lf}{\ensuremath{L^r\!f^s_{p,q}}}
\newcommand{\La}{\ensuremath{L^r a^s_{p,q}}}
\newcommand{\Lob}{\ensuremath{L^r \ob{}^s_{p,q}}}
\newcommand{\Lof}{\ensuremath{L^r \of{}^s_{p,q}}}
\newcommand{\Loa}{\ensuremath{L^r\, \oa{}^s_{p,q}}}
\newcommand{\Lcoa}{\ensuremath{\Lc^{r}\oa{}^s_{p,q}}}
\newcommand{\Lcob}{\ensuremath{\Lc^{r}\ob{}^s_{p,q}}}
\newcommand{\Lcof}{\ensuremath{\Lc^{r}\of{}^s_{p,q}}}
\newcommand{\Lca}{\ensuremath{\Lc^{r}\!a^s_{p,q}}}
\newcommand{\Lcb}{\ensuremath{\Lc^{r}\!b^s_{p,q}}}
\newcommand{\Lcf}{\ensuremath{\Lc^{r}\!f^s_{p,q}}}
\newcommand{\id}{\ensuremath{\mathrm{id}}}
\newcommand{\tr}{\ensuremath{\mathrm{tr\,}}}
\newcommand{\trd}{\ensuremath{\mathrm{tr}_d}}
\newcommand{\trL}{\ensuremath{\mathrm{tr}_L}}
\newcommand{\ext}{\ensuremath{\mathrm{ext}}}
\newcommand{\re}{\ensuremath{\mathrm{re\,}}}
\newcommand{\Rea}{\ensuremath{\mathrm{Re\,}}}
\newcommand{\Ima}{\ensuremath{\mathrm{Im\,}}}
\newcommand{\loc}{\ensuremath{\mathrm{loc}}}
\newcommand{\rloc}{\ensuremath{\mathrm{rloc}}}
\newcommand{\osc}{\ensuremath{\mathrm{osc}}}
\newcommand{\pr}{\pageref}
\newcommand{\wh}{\ensuremath{\widehat}}
\newcommand{\wt}{\ensuremath{\widetilde}}
\newcommand{\ol}{\ensuremath{\overline}}
\newcommand{\os}{\ensuremath{\overset}}
\newcommand{\Li}{\ensuremath{\overset{\circ}{L}}}
\newcommand{\Ai}{\ensuremath{\os{\, \circ}{A}}}
\newcommand{\Ci}{\ensuremath{\os{\circ}{\Cc}}}
\newcommand{\dom}{\ensuremath{\mathrm{dom \,}}}
\newcommand{\SA}{\ensuremath{S^r_{p,q} A}}
\newcommand{\SB}{\ensuremath{S^r_{p,q} B}}
\newcommand{\SF}{\ensuremath{S^r_{p,q} F}}
\newcommand{\Hc}{\ensuremath{\mathcal H}}
\newcommand{\Lci}{\ensuremath{\overset{\circ}{\Lc}}}
\newcommand{\bmo}{\ensuremath{\mathrm{bmo}}}
\newcommand{\BMO}{\ensuremath{\mathrm{BMO}}}
\newcommand{\cm}{\\[0.1cm]}
\newcommand{\Aa}{\ensuremath{\os{\, \ast}{A}}}
\newcommand{\Ba}{\ensuremath{\os{\, \ast}{B}}}
\newcommand{\Fa}{\ensuremath{\os{\, \ast}{F}}}
\newcommand{\Aas}{\ensuremath{\Aa{}^s_{p,q}}}
\newcommand{\Bas}{\ensuremath{\Ba{}^s_{p,q}}}
\newcommand{\Fas}{\ensuremath{\Fa{}^s_{p,q}}}
\newcommand{\Ca}{\ensuremath{\os{\, \ast}{{\mathcal C}}}}
\newcommand{\Cas}{\ensuremath{\Ca{}^s}}
\newcommand{\Car}{\ensuremath{\Ca{}^r}}

\begin{center}
{\Large Quarkonial analysis}
\\[1cm]
{Hans Triebel}
\\[0.2cm]
Institut f\"{u}r Mathematik\\
Friedrich--Schiller--Universit\"{a}t Jena\\
07737 Jena, Germany
\\[0.1cm]
email: hans.triebel@uni-jena.de
\\[5cm]
\end{center}

\begin{abstract}
The spaces $\As (\rn)$ with $A \in \{B,F \}$, $s\in \real$ and $0<p,q \le \infty$ are usually introduced in terms of 
Fourier--analytical decompositions. Related characterizations based on atoms and wavelets are known nowadays in a rather final way. 
Quarks atomize the atoms into constructive building blocks. It is the main aim of these notes to raise quarkonial  decompositions to
the same level as related representations of the spaces $\As (\rn)$
in terms of atoms or wavelets. This will be complemented by some applications. 
In addition we deal also with quarks in domains and their relations to so--called refined localization spaces.
\\[2cm]
\end{abstract}



\newpage

\tableofcontents

\newpage

\section{Preliminaries}   \label{S1}
\subsection{Introduction}   \label{S1.1}
These notes may be considered as a complement to \cite{T20}, although we do not care very much about the spaces $F^s_{\infty,q}
(\rn)$, $0<q \le \infty$, treated there in detail. But we repeat  basic notation and related assertions in order to make this text to
some extent independently readable.

The classical Sobolev spaces $W^k_p (\rn)$, $k\in \nat$, $1<p<\infty$, and the classical Besov spaces $\Bs (\rn)$, $s>0$ and $1\le
p,q \le \infty$, including the H\"{o}lder--Zygmund spaces $\Cc^s (\rn) = B^s_{\infty,\infty} (\rn)$, $s>0$, as introduced  in the
1950s, are based on derivatives and differences in the context of distributions. But there are some forerunners and modifications
(especially for spaces on domains) going back to the 1930s. Details and (historical) references may be found in \cite[Sections 4.3.1 --
4.3.3, pp.\,132--135]{T20}.  Beginning with the 1960s this theory has been extended to the space
$\As (\rn)$, $A \in \{B,F \}$, $s\in \real$ and $0<p,q \le \infty$, now relying on Fourier--analytical decompositions and distinguished
building blocks:

\begin{itemize}
\item[1.] The Fourier--analytical definition of the spaces $\As (\rn)$, $A\in \{B,F \}$, $s\in \real$ and $0<p,q \le \infty$ ($p<\infty
$ for $F$--spaces) goes back to the decade between the mid 1960s and the mid 1970s. This had been complemented by corresponding definitions
for the spaces $F^s_{\infty,q} (\rn)$, $s\in \real$, $0<q<\infty$, around 1990. 
\item[2.] Atoms (and molecules) as building blocks in function spaces have a long history. The versions as used up to our time go back
to the early 1990s. 
\item[3.] Wavelets proved to be a powerful instrument in the theory of the above  function spaces. They originated around 1990. But the
final version for the full scale of the spaces $\As (\rn)$ as described below is more recent. 
\end{itemize}

Fourier--analytical decompositions (including the so--called
paramultiplication), atoms (and molecules), as well as wavelets as building blocks of the
spaces $\As (\rn)$ have their advantages and disadvantages. If one steps from $\As (\rn)$ to corresponding spaces $\As (\Om)$ on 
domains  $\Om$ in $\rn$ or to $\As (\Gamma)$ on rough sets $\Gamma$ including fractals $\Gamma$ on \rn, then Fourier--analytical
decompositions into holomorphic building blocks are not well adapted. Atoms are qualitatively defined. This is an advantage at many
occasions and will be used in what follows several times. But usually one cannot expect frame characterizations with (constructive)
optimal coefficients. Wavelets with compact supports are rather rigid and not universal (they are necessarily of limited smoothness). They
are quite often bases on related separable spaces $\As (\rn)$ and produce isomorphic maps onto suitable sequence spaces. But they are
not so well adapted for the study of corresponding spaces on domains and rough sets, including fractals, in \rn. We tried to combine 
the advantages of the above building blocks. Quite obviously, one must be ready to pay a prize. These attempts started in \cite{T97}.
There we decomposed harmonic and caloric  extensions of some spaces $\As(\rn)$ to $\real^{n+1}_+$. The resulting building blocks are
called {\em quarks}. We returned to this topic in \cite[Chapter 2]{T01} where we defined $\As (\rn)$ with $0<p,q \le \infty$ ($p<
\infty$ for $F$--spaces), and
\begin{\eq}   \label{1.1}
s > n \cdot
\begin{cases}
\max ( \frac{1}{p}, 1 ) -1 &\text{for $B$--spaces}, \\
\max (\frac{1}{p}, \frac{1}{q},1) -1 &\text{for $F$--spaces},
\end{cases}
\end{\eq}
in terms of quarkonial decompositions and proved afterwards that these spaces coincide with corresponding spaces introduced 
Fourier--analytically. But the offered extension of this theory to all spaces $\As (\rn)$ with $s\in \real$ and $0<p,q \le \infty$
($p<\infty$ for $F$--spaces) is not really satisfactory. In \cite[Section 3.2]{T06} we dealt with a refined and modified version of
this theory restricted to $B^s_p (\rn) = B^s_{p,p} (\rn)$ (sometimes even with $p>1$) as an instrument to study the local and pointwise
regularity of functions in a fractal setting. A description and a discussion of at least some of these assertions have been given in
\cite[Section 4.3.6, pp.\,138--142]{T20}. But now new (or overlooked) tools are available. They pave the way to raise
\begin{itemize}
\item[4.] quarkonial representations of the spaces $\As (\rn)$
\end{itemize}
to the same level as the above--described building blocks: Fourier--analytical decompositions, atomic representations 
 and wavelet expansions. This is 
the main topic of these notes. In addition to quarkonial representations for the spaces $\As (\rn)$ we deal also with corresponding
assertions for their weighted generalizations $\As (\rn,w)$, where $w$ is a so--called {\em admissible weight}. This
is not only done for its own sake but it paves the way
for the description of universal  expansions of elements
in $S(\rn)$ and $S'(\rn)$. Furthermore we deal with some applications offering
 new proofs of already known assertions about
pointwise multipliers, multiplication algebras and the so--called positivity property of some function spaces.    
This illustrates that the (now)
\begin{itemize}
\item[] 4 building blocks (Fourier--analytical decompositions, atoms, wavelets and quarks)
\end{itemize}
have their advantages and disadvantages. They complement in a decisive way each other and
the classical tools such as derivatives, differences, caloric expansions and interpolation to mention a few. 

Finally we deal with
quarkonial expansions of some function spaces on domains in \rn, complementing  again related atomic and wavelet representations.

\subsection{Definitions}   \label{S1.2}
We use standard notation. Let $\nat$ be the collection of all natural numbers and $\no = \nat \cup \{0 \}$. Let $\rn$ be Euclidean $n$-space, where $n \in \nat$. Put $\real = \real^1$, whereas $\comp$ is the complex plane. Let $S(\rn)$ be the Schwartz space of all complex-valued rapidly decreasing
infinitely differentiable functions on $\rn$ and let $S' (\rn)$ be the  dual space of all
tempered distributions on \rn. Let $L_p (\rn)$ with $0<p \le \infty$ be the standard complex quasi-Banach space with respect to the Lebesgue measure in \rn, quasi-normed by
\begin{\eq}   \label{1.2}
\| f \, | L_p (\rn) \| = \Big( \int_{\rn} |f(x)|^p \, \di x \Big)^{1/p}
\end{\eq}
with the natural modification if $p= \infty$. Similarly, $L_p (M)$ where $M$ is a Lebesgue--measurable subset of \rn.
As usual $\ganz$ is the collection of all integers. Let $\zn$ where $n\in \nat$ be the lattice of all points $m = (m_1, \ldots, m_n) \in \rn$ with $m_j \in \ganz$. Let $\nat^n_0 = \{ m\in \zn: \ m_j \in \no \}$. 
If $\vp \in S(\rn)$ then
\begin{\eq}   \label{1.3}
\wh{\vp} (\xi) = (F \vp)(\xi) = (2 \pi)^{-n/2} \int_{\rn} e^{-ix \xi} \, \vp (x) \, \di  x, \qquad \xi \in \rn,
\end{\eq}
denotes the Fourier transform of \vp. As usual, $F^{-1} \vp$ and $\vp^\vee$ stand for the inverse Fourier transform, which is given by the right-hand side of 
\eqref{1.3} with $i$ in place $-i$. Note that $x \xi = \sum^n_{j=1} x_j \xi_j$ 
is  the scalar product in \rn. Both $F$ and $F^{-1}$ are extended to $S' (\rn)$ in the standard way. Let $\vp_0 \in S(\rn)$ with
\begin{\eq}   \label{1.4}
\vp_0 (x) =1 \ \text{if $|x| \le 1$} \qquad \text{and} \qquad \vp_0 (x) =0 \ \text{if $|x| \ge 3/2$},
\end{\eq}
and let
\begin{\eq}   \label{1.5}
\vp_k (x) = \vp_0 \big( 2^{-k} x) - \vp_0 \big( 2^{-k+1} x \big), \qquad x \in \rn, \quad k \in \nat.
\end{\eq}
Since
\begin{\eq}    \label{1.6}
\sum^\infty_{j=0} \vp_j (x) = 1 \qquad \text{for} \quad x\in \rn,
\end{\eq}
the $\vp_j$ form a dyadic resolution of unity. The entire analytic functions $(\vp_j \wh{f} )^\vee (x)$ make sense pointwise in $\rn$ for any $f \in S' (\rn)$. Let
\begin{\eq}   \label{1.7}
Q_{J,M} = 2^{-J} M + 2^{-J} (0,1)^n, \qquad J \in \ganz, \quad M \in \zn,
\end{\eq}
be the usual dyadic cubes in \rn, $n\in \nat$, with sides of length $2^{-J}$ parallel to the axes of coordinates and with $2^{-J}M$
as lower left corner. If $Q$ is a cube in $\rn$ and $d>0$ then $dQ$ is the cube in $\rn$ concentric with $Q$ whose  side--length
is $d$ times the side--length of $Q$. Let $|\Om|$ be the Lebesgue measure of the Lebesgue measurable  set $\Om$ in \rn. Let $a^+ =
\max (a,0)$ for $a\in \real$.

\begin{definition}   \label{D1.1}
Let $\vp = \{ \vp_j \}^\infty_{j=0}$ be the above dyadic resolution  of unity.
\\[0.1cm]
{\upshape (i)} Let
\begin{\eq}   \label{1.8}
0<p \le \infty, \qquad 0<q \le \infty, \qquad s \in \real.
\end{\eq}
Then $\Bs (\rn)$ is the collection of all $f \in S' (\rn)$ such that
\begin{\eq}   \label{1.9}
\| f \, | \Bs (\rn) \|_{\vp} = \Big( \sum^\infty_{j=0} 2^{jsq} \big\| (\vp_j \widehat{f})^\vee \, | L_p (\rn) \big\|^q \Big)^{1/q} 
\end{\eq}
is finite $($with the usual modification if $q= \infty)$. 
\\[0.1cm]
{\upshape (ii)} Let
\begin{\eq}   \label{1.10}
0<p<\infty, \qquad 0<q \le \infty, \qquad s \in \real.
\end{\eq}
Then $F^s_{p,q} (\rn)$ is the collection of all $f \in S'(\rn)$ such that
\begin{\eq}   \label{1.11}
\| f \, | F^s_{p,q} (\rn) \|_{\vp} = \Big\| \Big( \sum^\infty_{j=0} 2^{jsq} \big| (\vp_j \wh{f})^\vee (\cdot) \big|^q \Big)^{1/q} \big| L_p (\rn) \Big\|
\end{\eq}
is finite $($with the usual modification if $q=\infty)$.
\\[0.1cm]
{\upshape (iii)} Let $0<q < \infty$ and $s\in \real$. Then $F^s_{\infty,q} (\rn)$ is the collection of all $f\in S'(\rn)$ such that
\begin{\eq}  \label{1.12}
\| f \, | F^s_{\infty,q} (\rn) \|_{\vp} = \sup_{J\in \ganz, M\in \zn} 2^{Jn/q} \Big(\int_{Q_{J,M}} \sum_{j \ge J^+} 2^{jsq} \big|
(\vp_j \wh{f} )^\vee (x) \big|^q \, \di x \Big)^{1/q}
\end{\eq}
is finite. Let $F^s_{\infty, \infty} (\rn) = B^s_{\infty, \infty} (\rn)$.
\end{definition}

\begin{remark}    \label{R1.2}
We take for granted that the reader is familiar with the theory of the spaces $\As (\rn)$, $A\in \{B,F \}$, $s\in \real$ and $0 <p,q
\le \infty$. In particular they are independent of the chosen resolution of unity $\vp$ (equivalent quasi--norms). This justifies our
omission of the subscript $\vp$ in \eqref{1.9}, \eqref{1.11} and \eqref{1.12}. Discussions and references may be found in 
\cite[Section 1.1.1, pp.\,1--5]{T20}. The spaces $F^s_{\infty, q} (\rn)$  will play only a marginal role (incorporated in related
assertions if no additional efforts are needed).  
On the other hand it will be 
occasionally useful to specify some properties to the following  special spaces. Let
\begin{\eq}   \label{1.13}
B^s_p (\rn) = B^s_{p,p} (\rn) = F^s_{p,p} (\rn), \quad 0<p \le \infty, \quad \text{and} \quad s\in \real,
\end{\eq}
including the {\em H\"{o}lder--Zygmund spaces}
\begin{\eq}   \label{1.14}
\Cc^s (\rn) = B^s_\infty (\rn) = B^s_{\infty,\infty} (\rn), \qquad s \in \real.
\end{\eq}
Furthermore,
\begin{\eq}   \label{1.15}
H^s_p (\rn) = F^s_{p,2} (\rn), \qquad 1<p<\infty \quad \text{and} \quad s\in \real
\end{\eq}
are the {\em Sobolev spaces}, including the {\em classical Sobolev spaces}
\begin{\eq}   \label{1.16}
W^k_p (\rn) = H^k_p (\rn), \qquad 1<p<\infty, \quad k\in \no,
\end{\eq}
which can be equivalently normed by
\begin{\eq}   \label{1.17}
\| f \, | W^k_p(\rn) \| = \sum_{|\alpha| \le k} \| D^\alpha f \, | L_p (\rn) \|, \qquad 1<p<\infty, \quad k\in \no,
\end{\eq}
where as usual
\begin{\eq}   \label{1.18}
D^\alpha = \prod^n_{j=1} \pa^{\alpha_j}_j, \qquad \alpha = (\alpha_1, \ldots, \alpha_n), \quad \alpha \in \nat^n_0, \quad |\alpha|=
\sum^n_{j=1} \alpha_j,
\end{\eq}
based on $\pa^m_j = \pa^m/ \pa x^m_j$, $m\in \no$ with $\pa^1_j = \pa_j$ and $\pa^0 f = f$.
\end{remark}

\subsection{Atoms}    \label{S1.3}
Next we describe the well--known representations of elements belonging to the above spaces $\As (\rn)$, adapted to our later needs. 
Let $\nat^n_0$, where $n\in \nat$, be again the set of all multi--indices,
\begin{\eq}    \label{1.19}
\alpha = (\alpha_1, \ldots, \alpha_n) \qquad \text{with $\alpha_j \in \no$ and $|\alpha| = \sum^n_{j=1} \alpha_j$},
\end{\eq}
underlying \eqref{1.18}, and
\begin{\eq}   \label{1.20}
x^\beta = \prod^n_{j=1} x^{\beta_j}_j \qquad \text{for $x= (x_1, \ldots, x_n) \in \rn$ and $\beta = (\beta_1, \ldots, \beta_n) \in 
\nat^n_0$}.
\end{\eq}
Let $Q_{j,m}$ be as in \eqref{1.7} and let $d\,Q_{j,m}$ be a cube concentric with $Q_{j,m}$ and with side--length $d\,2^{-J}$, $d>0$.

\begin{definition}   \label{D1.3}
Let $K \in \no$, $L\in \no$, $C>0$ and $d>1$. Then
\begin{\eq}   \label{1.21}
\{ a_{j,m}: \ j\in \no, \ m\in \zn \}
\end{\eq}
is called an atomic system if the atoms $a_{j,m}: \rn \mapsto \comp$ are $L_\infty$--normalized functions such that
\begin{\eq}   \label{1.22}
\supp a_{j,m} \subset d\,Q_{j,m}, \qquad j\in \no, \quad m\in \zn,
\end{\eq}
\begin{\eq}    \label{1.23}
\int_{\rn} x^\beta a_{j,m} (x) \ \di x =0, \qquad |\beta| <L, \quad j\in \nat, \quad m\in \zn,
\end{\eq}
and
\begin{\eq}   \label{1.24}
| D^\alpha a_{j,m}  (x) | \le C\,2^{j|\alpha|}, \qquad |\alpha| \le K, \quad j\in \no, \quad m\in \zn.
\end{\eq}
\end{definition}

\begin{remark}    \label{R1.4}
This is an adapted version of \cite[Definition 1.5, p.\,4]{T08} with a different normalization. No cancellation \eqref{1.23} for $a_{0,
m}$ is required. If $K=0$ then \eqref{1.24} means that $a_{j,m} \in L_\infty (\rn)$ and $|a_{j,m} (x)| \le C$. If $K\in \nat$ then
\eqref{1.24} means that $a_{j,m}$ has classical derivatives for all $\alpha \in \nat^n_0$ with $|\alpha| \le K$. If $L=0$ then 
\eqref{1.23} is empty (no cancellations are required). The condition $d>1$ ensures that for any fixed $j\in \no$ some controlled
over-laps of the atoms $a_{j,m}$, $m\in \zn$, are admitted. Usually one assumes $C=1$ in \eqref{1.24}. But the above modification
will be convenient  for us. For fixed $K\in \no$, $L\in \no$, $C>0$ and $d>1$ we call
\begin{\eq}    \label{1.25}
\{ a_{j,m}: \ j\in \no, \ m\in \zn \} \qquad \text{an $(K,L,C,d)$--atomic system}.
\end{\eq}
\end{remark}

Atomic representations 
\begin{\eq}   \label{1.26}  
f= \sum_{j\in \no} \sum_{m \in \zn}  \lambda_{j,m} \, a_{j,m}
\end{\eq}
of elements $f$ belonging to some spaces $\As (\rn)$ as introduced in Definition \ref{D1.1} rely on suitable atomic systems according 
to \eqref{1.25} and  the following sequence spaces for the coefficients $\lambda_{j,m} \in \comp$. Let $\chi_{j,m}$ be the 
characteristic function of $Q_{j,m}$ according to \eqref{1.7}. 

\begin{definition}    \label{D1.5}
Let $s\in \real$ and $0<p,q \le \infty$ $(p<\infty$ for $f$--spaces$)$. Then $b^s_{p,q} (\rn)$ is the collection of all sequences
\begin{\eq}   \label{1.27}
\lambda = \{ \lambda_{j,m} \in \comp: \ j \in \no, \ m \in \zn \}
\end{\eq}
such that
\begin{\eq}   \label{1.28}
\| \lambda \, | b^s_{p,q} (\rn) \| =  \bigg( \sum^\infty_{j=0} 2^{j(s- \frac{n}{p})q} \Big( \sum_{m\in \zn}
|\lambda_{j,m} |^p \Big)^{q/p} \bigg)^{1/q}
\end{\eq}
is finite and $f^s_{p,q} (\rn)$ is the collection of all sequences according to \eqref{1.27} such that
\begin{\eq}   \label{1.29}
\| \lambda \, | f^s_{p,q} (\rn) \| = 
\bigg\| \Big( \sum^\infty_{j=0} \sum_{m \in \zn} 2^{jsq} \big| 
\lambda_{j,m} \, \chi_{j,m} (\cdot) \big|^q \Big)^{1/q} \, | L_p (\rn) \bigg\|
\end{\eq}
is finite $($with the usual modifications if $\max (p,q) = \infty.)$
\end{definition}

\begin{remark}   \label{R1.6}
Quite obviously, $a^s_{p,q}(\rn)$ with $a\in \{b,f \}$ are quasi--Banach spaces (Banach spaces if $p \ge 1$, $q \ge 1$). Furthermore
$b^s_{p,p} (\rn) = f^s_{p,p} (\rn)$, $p<\infty$.
\end{remark}

Let
\begin{\eq}   \label{1.30}
\sigma^n_p = n \Big( \max \big( \frac{1}{p}, 1 \big) -1 \Big), \qquad \sigma^n_{p,q} = n \Big( \max \big( \frac{1}{p}, \frac{1}{q}, 1
\big) - 1 \Big),
\end{\eq}
where $n\in \nat$ and $0<p,q \le \infty$. These are the right--hand sides of \eqref{1.1}.

\begin{theorem}   \label{T1.7}
Let $n\in \nat$.
\cm
{\em (i)} Let $0<p \le \infty$, $0<q \le \infty$, $s\in \real$. Let $K\in \no$, $L\in \no$, $C\in \real$, $d\in \real$ with
\begin{\eq}   \label{1.31}
K > s, \qquad L > \sigma^n_p -s,
\end{\eq}
$d>1$ and $C>0$ be fixed. Then $f \in S'(\rn)$ belongs to $\Bs (\rn)$ if, and only if, it can be represented as
\begin{\eq}   \label{1.32}
f = \sum^\infty_{j=0} \sum_{m\in \zn}  \lambda_{j,m} a_{j,m}
\end{\eq}
where $\{a_{j,m} \}$ is an $(K,L,C,d)$--atomic system according to \eqref{1.25} and $\lambda = \{\lambda_{j,m} \} \in b^s_{p,q} (\rn)$.
Furthermore, 
\begin{\eq} \label{1.33}
\| f \, | \Bs (\rn) \| \sim \inf \|\lambda \, | b^s_{p,q} (\rn) \|
\end{\eq}
are equivalent quasi--norms where the infimum is taken over all admissible representations \eqref{1.32}.
\cm
{\em (ii)} Let $0<p < \infty$, $0<q \le \infty$, $s\in \real$. Let $K\in \no$, $L\in \no$, $C\in \real$, $d\in \real$ with
\begin{\eq}   \label{1.34}
K > s, \qquad L > \sigma^n_{p,q} -s,
\end{\eq}
$d>1$ and $C>0$ be fixed. Then $f \in S'(\rn)$ belongs to $\Fs (\rn)$ if, and only if, it can be represented by \eqref{1.32} where $\{
a_{j,m} \}$ is an $(K,L,C,d)$--atomic system and $\lambda = \{\lambda_{j,m} \} \in f^s_{p,q} (\rn)$. Furthermore
\begin{\eq} \label{1.35}
\| f \, | \Fs (\rn) \| \sim \inf \|\lambda \, | f^s_{p,q} (\rn) \|
\end{\eq}
are equivalent quasi--norms where the infimum is taken over all admissible representations \eqref{1.32}.
\end{theorem}

\begin{remark}   \label{R1.8}
This coincides essentially with \cite[Theorem 1.7, p.\,5]{T08} with a reference to \cite[Section 1.5.1, p.\,12--15]{T06}, based on a
different, but immaterial, normalization of the underlying atoms $a_{j,m}$ and the related sequence spaces $a^s_{p,q} (\rn)$, $a\in 
\{b,f \}$. This ensures that the series in \eqref{1.32} converges in any case unconditionally in $S'(\rn)$. As indicated in in 
\cite[p.\,138]{T20} a corresponding representation  for the spaces $F^s_{\infty,q} (\rn)$, $0<q<\infty$, can be based on the 
identification of these spaces according to \cite[(1.71), p.\,13]{T20} with suitable hybrid spaces and the corresponding atomic
representations as described in \cite[Theorem 3.33, p.\,67]{T14}. But this will not be used in the sequel.
\end{remark}

\begin{remark}   \label{R1.9}
It is well known that one can relax \eqref{1.22} by the decay
\begin{\eq}   \label{1.36}
|b_{j,m} (x)| \le C \,\big( 1+ |2^j x-m | \big)^{-N}, \qquad x\in \rn, \quad j\in \no, \quad m\in \zn,
\end{\eq}
for related {\em molecules}. In particular the above theorem remains valid for the spaces $\As (\rn)$ with $s<0$ if one replaces there
the atoms $a_{j,m} \in L_\infty (\rn)$ by the molecules $b_{j,m} \in L_\infty (\rn)$, satisfying
\begin{\eq}  \label{1.37}
\int_{\rn} x^\beta \, b_{j,m} (x) \, \di x =0, \qquad |\beta| <L, \quad j \in \nat, \quad m \in \zn,
\end{\eq}
with
\begin{\eq}   \label{1.38}
 L>
\begin{cases}
\sigma^{n}_p -s &\text{if $A=B$}, \\
\sigma^{n}_{p,q} -s &\text{if $A=F$},
\end{cases}
\end{\eq}
as in \eqref{1.23} and \eqref{1.31}, \eqref{1.34} and $N > L+n$ in \eqref{1.36}. We refer the reader to \cite[pp.\,17, 24]{Tor91} and
the respective assertions in in \cite{FrJ90} where one finds molecular representations for all $s\in \real$. We described in
\cite[p.\,85]{T13} how corresponding assertions for the spaces $\As (\rn)$ with $s \ge 0$ look like. But this will not be needed in
the sequel.
\end{remark}

\subsection{Wavelets}   \label{S1.4}
As indicated in the Introduction, Section \ref{S1.1}, it is our main aim  to raise quarkonial representations in the function spaces 
$\As (\rn)$ to the same level as Fourier--analytical decompositions, atomic representations and wavelet expansions. This will be mainly
based on Fourier--analytical and atomic arguments. Wavelets will play no decisive role. But by our intention it is quite clear that a
detailed description of wavelet expansions is desirable and to discuss what are the advantages and disadvantages of the diverse 
building blocks. We follow essentially \cite[Section 1.2.1, p.\,7--10]{T20} which in turn is based on \cite[Section 3.2.3, 
pp.\,51--54]{T14}. There one finds further explanations and, in particular, corresponding references.

As usual, $C^{u} (\real)$ with $u\in
\nat$ collects all bounded complex-valued continuous functions on $\real$ having continuous bounded derivatives up to order $u$ inclusively. Let
\begin{\eq}   \label{1.39}
\psi_F \in C^{u} (\real), \qquad \psi_M \in C^{u} (\real), \qquad u \in \nat,
\end{\eq}
be {\em real} compactly supported Daubechies wavelets with
\begin{\eq}   \label{1.40}
\int_{\real} \psi_M (x) \, x^v \, \di x =0 \qquad \text{for all $v\in \no$ with $v<u$.}
\end{\eq}
Recall that $\psi_F$ is called the {\em scaling function} (father wavelet) and $\psi_M$ the associated {\em wavelet} 
(mother wavelet). We extend these wavelets from $\real$ to $\rn$ by the usual multiresolution procedure. Let $n\in \nat$ and let
\begin{\eq}   \label{1.41}
G = (G_1, \ldots, G_n) \in G^0 = \{F,M \}^n
\end{\eq}
which means that $G_r$ is either $F$ or $M$. Furthermore let
\begin{\eq}   \label{1.42}
G= (G_1, \ldots, G_n) \in G^* = G^j = \{F, M \}^{n*}, \qquad j \in \nat,
\end{\eq}
which means that $G_r$ is either $F$ or $M$, where $*$ indicates that at least one of the components of $G$ must be an $M$. Hence $G^0$ has $2^n$ elements, whereas $G^j$ with $j\in \nat$ and $G^*$ have $2^n -1$ elements. Let
\begin{\eq}   \label{1.43}
\psi^j_{G,m} (x) = \prod^n_{l=1} \psi_{G_l} \big(2^j x_l -m_l \big), \qquad G\in G^j, \quad m \in \zn, \quad x\in \rn,
\end{\eq}
where (now) $j \in \no$. We always assume that $\psi_F$ and $\psi_M$ in \eqref{1.39} have $L_2$--norm 1. Then 
\begin{\eq}   \label{1.44}
 \big\{ 2^{jn/2} \psi^j_{G,m}: \ j \in \no, \ G\in G^j, \ m \in \zn \big\}
\end{\eq}
is an {\em orthonormal basis} in $L_2 (\rn)$ (for any $u\in \nat$) and
\begin{\eq}   \label{1.45}
f = \sum^\infty_{j=0} \sum_{G \in G^j} \sum_{m \in \zn} \lambda^{j,G}_m \, \psi^j_{G,m}
\end{\eq}
with
\begin{\eq}   \label{1.46}
\lambda^{j,G}_m = \lambda^{j,G}_m (f) = 2^{jn} \int_{\rn} f(x) \, \psi^j_{G,m} (x) \, \di x = 2^{jn} \big(f, \psi^j_{G,m} \big)
\end{\eq}
is the corresponding expansion. Let again $\chi_{j,m}$ be the characteristic function of the cube $Q_{j,m} = 2^{-j} m + 2^{-j} (0,1)^n$
where $j \in \no$ and $m\in \zn$. We adapt the sequence spaces $a^s_{p,q} (\rn)$, $a \in \{b,f \}$, as introduced in Definition
\ref{D1.5} to the above wavelets. This requires an additional summation over $G\in G^j$.

\begin{definition}   \label{D1.10}
Let $s\in \real$ and $0<p,q \le \infty$ $(p<\infty$ for $f$--spaces$)$. Then $b^s_{p,q} (\rn)_W$ is the collection of all sequences
\begin{\eq}   \label{1.47}
\lambda = \big\{ \lambda^{j,G}_m \in \comp: \ j \in \no, \ G \in G^j, \ m \in \zn \big\}
\end{\eq}
such that
\begin{\eq}   \label{1.48}
\| \lambda \, | b^s_{p,q} (\rn)_W \| = \bigg( \sum^\infty_{j=0} 2^{j(s- \frac{n}{p})q} \sum_{G \in G^j} \Big( \sum_{m \in \zn} |\lambda^{j,G}_m|^p \Big)^{q/p} \bigg)^{1/q}
\end{\eq}
is finite and $f^s_{p,q} (\rn)_W$ is the collection of all sequences according to \eqref{1.47} such that
\begin{\eq}   \label{1.49}
\| \lambda \, | f^s_{p,q} (\rn)_W \| = \Big\| \Big( \sum_{\substack{j\in \no, G\in G^j,\\ m\in \zn}}
2^{jsq} \big| \lambda^{j,G}_m \, \chi_{j,m} (\cdot) \big|^q \Big)^{1/q} \big| L_p (\rn)
\Big\|
\end{\eq}
is finite $($with the usual modifications if $\max(p,q) =\infty)$.
\end{definition}

\begin{remark}   \label{R1.11}
Quite obviously, $a^s_{p,q} (\rn)_W$ with $a\in \{b,f \}$ are quasi--Banach spaces (Banach spaces if $p \ge 1$, $q\ge 1$).
\end{remark}

Let again $\sigma^n_p$ and $\sigma^n_{p,q}$ be as in \eqref{1.30}.

\begin{theorem}   \label{T1.12}
{\upshape (i)} Let $0<p \le \infty$, $0<q \le \infty$, $s\in \real$ and
\begin{\eq}   \label{1.50}
u > \max (s, \sigma^{n}_p -s).
\end{\eq}
Let $f \in S' (\rn)$. Then $f \in \Bs (\rn)$ if, and only if, it can be represented as
\begin{\eq}   \label{1.51}
f= \sum_{\substack{j\in \no, G\in G^j, \\ m\in \zn}}
 \lambda^{j,G}_m \, \psi^j_{G,m}, \qquad \lambda \in b^s_{p,q} (\rn)_W,
\end{\eq}
the unconditional convergence being in $S' (\rn)$. The representation \eqref{1.51} is unique,
\begin{\eq}  \label{1.52}
\lambda^{j,G}_m = \lambda^{j,G}_m (f) =  2^{jn} \big( f, \psi^j_{G,m} \big)
\end{\eq}
and
\begin{\eq}   \label{1.53}
I: \quad f \mapsto \big\{ \lambda^{j,G}_m (f) \big\}
\end{\eq}
is an isomorphic map of $\Bs (\rn)$ onto $b^s_{p,q} (\rn)_W$.
\\[0.1cm]
{\upshape (ii)} Let $0<p < \infty$, $0<q\le \infty$, $s\in \real$ and
\begin{\eq}   \label{1.54}
u > \max (s, \sigma^{n}_{p,q} -s).
\end{\eq}
Let $f \in S' (\rn)$. Then $f \in \Fs (\rn)$ if, and only if, it can be represented as
\begin{\eq}   \label{1.55}
f = \sum_{\substack{j\in \no, G\in G^j, \\ m\in \zn}} \lambda^{j,G}_m \, \psi^j_{G,m}, \qquad \lambda \in f^s_{p,q} (\rn)_W,
\end{\eq}
the unconditional convergence being in $S' (\rn)$. The representation \eqref{1.55}
 is unique with \eqref{1.52}. Furthermore $I$ in \eqref{1.53} is an isomorphic map of
$\Fs (\rn)$ onto $f^s_{p,q} (\rn)_W$.
\end{theorem}

\begin{remark}    \label{R1.13}
These assertions are covered by \cite[Proposition 1.11, pp.\,9--10]{T20} where we dealt also with the spaces $F^s_{\infty,q} (\rn)$,
$0<q<\infty$ (not needed here). Otherwise we refer the reader to \cite[Section 3.2.3, pp.\,51--54]{T14} where one finds further
explanations and references. As already mentioned, wavelet expansions for elements belonging to $\As (\rn)$ will play only a marginal
role in what follows. We included the above material mainly for sake of completeness in the presentation of the most distinguished 
building blocks in the recent theory of function spaces.
\end{remark}

\subsection{Building blocks}   \label{S1.5}
So far we dealt in Section \ref{S1.3} with atomic representations  and in Section \ref{S1.4} with wavelet expansions  in the spaces 
$\As (\rn)$ as introduced in Definition \ref{D1.1} in terms of Fourier--analytical decompositions (mostly with $p<\infty$ for 
$F$--spaces). It is quite clear that the qualitative {\em atoms} and the constructive  rigid {\em wavelets} are rather different. As
already indicated in the Introduction, Section \ref{S1.1}, it is the main aim of these notes to complement these two types of
distinguished {\em building blocks} by a third type of building blocks, called {\em quarks}. This will be done in the following 
sections. The underlying constructions justify to refer to atoms, wavelets and quarks as {\em building blocks} from which all the
spaces $\As (\rn)$ are made from. This suggests to ask for a reformulation of Definition \ref{D1.1} in terms of {\em 
Fourier--analytical building blocks}. This has already been done in \cite[Section 2.5.2, pp.\,79--80]{T83} with the following outcome.

Let $s\in \real$ and $0<p,q \le \infty$. Then $\ell^s_q \big( L_p (\rn) \big)$ is the quasi--Banach space of all sequences $b= \{b_j
(x) \}^\infty_{j=0} \subset L_p (\rn)$ such that
\begin{\eq}   \label{1.56}
\big\| b \, | \ell^s_q \big( L_p (\rn) \big) \big\| = \Big( \sum^\infty_{j=0} 2^{jsq} \| b_j \, | L_p (\rn) \|^q \Big)^{1/q}
\end{\eq}
is finite and $L_p \big( \rn, \ell^s_q \big)$ is the quasi--Banach space of all sequences $b= \{b_j (x) \}^\infty_{j=0} \subset L_p (\rn)$ such that
\begin{\eq}   \label{1.57}
\| b \, | L_p \big( \rn, \ell^s_q \big) \| = \Big\| \Big( \sum^\infty_{j=0} 2^{jsq} |b_j (\cdot)|^q \Big)^{1/q} | L_p (\rn) \Big\|
\end{\eq}
is finite (with the usual modification if $q=\infty$). Let $\mathfrak{B}_p (\rn)$, $0<p \le \infty$ be the collection of all
\begin{\eq}   \label{1.58}
b = \{b_j \}^\infty_{j=0}, \qquad b_j \in S' (\rn) \cap L_p (\rn)
\end{\eq}
such that
\begin{\eq}   \label{1.59}
\supp b_j \subset 
\begin{cases}
\{y\in \rn: \ |y| \le 2 \} &\text{if $j=0$}, \\
\{ y \in \rn: \ 2^{j-1} \le |y| \le 2^{j+1} \} &\text{if $j\in \nat$}.
\end{cases}
\end{\eq}

\begin{proposition}    \label{P1.14}
Let $n\in \nat$.
\cm
{\em (i)} Let $0<p \le \infty$, $0<q \le \infty$, $s\in \real$. Then $f\in S' (\rn)$ belongs to $\Bs (\rn)$ if, and only if, it can be
represented as
\begin{\eq}   \label{1.60}
f = \sum^\infty_{j=0} b_j, \qquad b = \{b_j \}^\infty_{j=0} \in \ell^s_q \big( L_p (\rn) \big) \cap \mathfrak{B}_p (\rn),
\end{\eq}
the unconditional convergence being in $S'(\rn)$. Furthermore,
\begin{\eq}   \label{1.61}
\| f \, | \Bs (\rn) \| \sim \inf \big\| b \, | \ell^s_q \big( L_p (\rn) \big) \big\|
\end{\eq}
where the infimum is taken over all representations \eqref{1.60}.
\cm
{\em (ii)} Let $0<p<\infty$, $0<q  \le \infty$, $s\in \real$. Then $f\in S'(\rn)$ belongs to $\Fs (\rn)$ if, and only if, it can be
represented as
\begin{\eq}   \label{1.62}
f = \sum^\infty_{j=0} b_j, \qquad b = \{b_j \}^\infty_{j=0} \in L_p (\rn, \ell^s_q) \cap \mathfrak{B}_p (\rn),
\end{\eq}
the unconditional convergence being in $S'(\rn)$. Furthermore,
\begin{\eq}   \label{1.63}
\| f \, | \Fs (\rn) \| \sim \inf \big\| b \, | L_p (\rn, \ell^s_q ) \big\|
\end{\eq}
where the infimum is taken over all representations \eqref{1.62}.
\end{proposition}

\begin{remark}    \label{R1.15}
Definition \ref{D1.1} shows that $b_j = (\vp_j \wh{f} )^\vee$ is the optimal choice. Otherwise the short proof in \cite[p.\,79]{T83}
is essentially a Fourier multiplier theorem. We will not need Proposition \ref{P1.14} in the sequel. But it may justify to call $b_j$
in \eqref{1.58}, \eqref{1.59} {\em Fourier--analytical building blocks} on equal footing with atoms, wavelets, and quarks.
\end{remark}

\section{Quarkonial representations}   \label{S2}
\subsection{Quarkonial systems}   \label{S2.1}
So far we described in Section \ref{S1} the distinguished building blocks playing a decisive role in the recent theory of function 
spaces. These are the Fourier--analytical building blocks in the Sections \ref{S1.2} and \ref{S1.5}, atoms (and molecules) in Section
\ref{S1.3}, and wavelet in Section \ref{S1.4}. This will now be complemented by quarks as a fourth type of building blocks. They admit
universal representations of tempered distributions on $\rn$ in the context of the spaces $\As (\rn)$ and their weighted 
generalizations. First we collect what is already known. We follow essentially \cite[Section 4.3.6, pp.\,138--142]{T20} and the
references given there, in particular to \cite[Sections 1.6, 3.2]{T06}.

Let $k$ be a non--negative $C^\infty$ function in $\rn$, $n\in \nat$, with
\begin{\eq}   \label{2.1}
\supp k \subset \big\{ y\in \rn: \ |y| < 2^J, \ y_j >0 \big\}
\end{\eq}
for some $J\in \nat$ such that
\begin{\eq}   \label{2.2}
\sum_{m \in \zn} k(x-m) =1 \qquad \text{where} \quad x\in \rn.
\end{\eq}
Let again $x^\beta$ 
with $x \in \rn$ and $\beta \in \nat^n_0$ be as in \eqref{1.20}. Let $ax = (ax_1, \ldots, ax_n)$ for 
$a\in \real$ and $x= (x_1, \ldots, x_n) \in \rn$.  Then
\begin{\eq}   \label{2.3}
k^\beta (x) = \big( 2^{-J} x \big)^\beta k(x) \ge 0, \qquad x \in \rn, \quad \beta \in \nat^n_0,
\end{\eq}
and
\begin{\eq}   \label{2.4}
k^\beta_{j,m} (x) = k^\beta \big( 2^j x -m \big), \qquad j\in \no, \quad m\in \zn, \quad \beta \in \nat^n_0, \quad x\in \rn.
\end{\eq}
These elementary building blocks, called {\em quarks}, are $L_\infty$-normalized atoms as recalled in Definition \ref{D1.3} without
moment conditions. The additional assumption $y_j >0$ in \eqref{2.1} resulting in \eqref{2.3} is immaterial for what follows in the
next sections. But it is useful later on for some applications, for example the positivity property as considered in Section 
\ref{S3.3}. 
We need a second system which is {\em totally independent} of the above  quarks $k^\beta_{j,m}$. 
For $\beta = (\beta_1, \ldots, \beta_n) \in \nat^n_0$ let again
$|\beta| = \sum^n_{j=1} \beta_j$ and $\beta! = \prod^n_{j=1} \beta_j !$ (with $0! =1$). Let $x^\beta$ be as above. 
Let $\omega$ be a real $C^\infty$ function in $\rn$ with
\begin{\eq}   \label{2.5}
\supp \omega \subset (- \pi, \pi)^n \qquad \text{and} \quad \omega (x)=1 \quad \text{if} \quad |x| \le 2.
\end{\eq}
Let for some $J\in \nat$ as in \eqref{2.1},
\begin{\eq}   \label{2.6}
\omega^\beta (x) = \frac{i^{|\beta|} 2^{J|\beta|}}{(2\pi)^n \beta!} \, x^\beta \, \omega (x), \qquad x\in \rn, \quad \beta \in
\nat^n_0,
\end{\eq}
and
\begin{\eq}   \label{2.7}
\Om^\beta (x) = \sum_{m \in \zn} (\omega^\beta)^\vee (m) \, e^{-imx}, \qquad x\in \rn.
\end{\eq}
Let $\vp_0$ be a non--negative  $C^\infty$ function in $\rn$ with
\begin{\eq}   \label{2.8}
\vp_0 (x) =1 \ \text{if} \ |x| \le 1 \quad \text{and} \quad \vp_0 (x) =0 \ \text{if} \ |x| \ge 3/2
\end{\eq}
as in \eqref{1.4} and let $\vp^0 (x) = \vp_0 (x) - \vp_0 (2x)$, $x\in \rn$. Then 
\begin{align}   \label{2.9}
(\Phi^\beta_F )^\vee (\xi) &= \vp_0 (\xi) \, \Om^\beta (\xi), &&\text{$\beta \in \nat^n_0$, \quad $\xi \in \rn,$}&&   \\ \label{2.10}
(\Phi^\beta_M )^\vee (\xi) &= \vp^0 (\xi) \, \Om^\beta (\xi), &&\text{$\beta \in \nat^n_0$, \quad $\xi \in \rn,$}&& 
\end{align}
and
\begin{\eq}   \label{2.11}
\Phi^\beta_{j,m} (x) =
\begin{cases}
\Phi^\beta_F (x-m) &\text{if $j=0$}, \\
\Phi^\beta_M (2^j x -m) &\text{if $j\in \nat$},
\end{cases}
\end{\eq}
where $x\in \rn$, $\beta \in \nat^n_0$ and $m\in \zn$. Then $\Phi^\beta_{j,m} \in S(\rn)$ are analytic functions. One may assume that
$\omega (x) = \prod^n_{j=1} \wt{\omega} (x_j)$ is the product of related real even one--dimensional functions. Then $(\omega^\beta)^\vee(m) = (-1)^{|\beta|} (\omega^\beta)^\vee (-m)$ 
are real, $m\in \zn$. Inserted in \eqref{2.7} it follows that $\Om^\beta (x) = (-1)^{|\beta|} \Om^\beta (-x)$ and that $i^{|\beta|}
\Om^\beta (x)$ is real. Assuming in addition that $\vp_0$ is the product of related one--dimensional even real functions then one obtains from \eqref{2.9}--\eqref{2.11} that the
functions $\Phi^\beta_{j,m}$ are also real. In other words, we may assume that
\begin{\eq} \label{2.12}
\Phi^\beta_{j,m} \in S(\rn) \quad \text{are real and analytic}, \quad \beta \in \nat^n_0, \ j \in \no, \ m\in \zn.
\end{\eq}
We collect some basic properties of the two quarkonial systems
\begin{\eq}   \label{2.13}
\Big\{ k^\beta_{j,m}: \ \beta \in \nat^n_0, \ j\in \no, \ m\in \zn \Big\}
\end{\eq}
and
\begin{\eq}   \label{2.14}
\Big\{ \Phi^\beta_{j,m}: \ \beta \in \nat^n_0, \ j\in \no, \ m\in \zn \Big\}.
\end{\eq} 
This is the quarkonial counterpart of the atomic systems in \eqref{1.25}, the wavelet systems in \eqref{1.44} and the 
Fourier--analytical system \eqref{1.58}, \eqref{1.59}.

Let $k$ be as in \eqref{2.1}. Then there is a number $\ve >0$ such that
\begin{\eq}   \label{2.15}
\supp k \subset \big\{ y \in \rn: \ |y| <2^{J - \ve}, \ y_j >0 \big\}.
\end{\eq}

\begin{proposition}  \label{P2.1}
{\em (i)} Let $- \infty < \vk < \ve$ and $K\in \nat$. Then there is a number $C>0$ such that
\begin{\eq}   \label{2.16}
\big| 2^{\vk |\beta|} D^\alpha k^\beta_{j,m} (x) \big| \le C \, 2^{j|\alpha|}, \qquad x\in \rn,
\end{\eq}
for all $\beta \in \nat^n_0$, $j\in \no$, $m\in \zn$ and all $\alpha \in \nat^n_0$ with $|\alpha| \le K$.
\\[0.1cm]
{\em (ii)} Let $\vk \in \real$, $K \in \nat$ and $N \in \nat$. Then there is a number $C>0$ such that
\begin{\eq}   \label{2.17}
\big| 2^{\vk |\beta|} \, D^\alpha \Phi^\beta_{j,m} (x) \big| \le C \, 2^{j |\alpha|} \big( 1 + |2^j x -m|\big)^{-N},
\quad x \in \rn,
\end{\eq}
for all $\beta \in \nat^n_0$, $j\in \no$, $m\in \zn$ and all $\alpha \in \nat^n_0$ with $|\alpha| \le K$.
\end{proposition}

\begin{proof}
{\em Step 1.} Let $\delta >0$ and let $k^\beta (x)$ be as in \eqref{2.3}. Then \eqref{2.16} follows from
\begin{\eq}   \label{2.18}
| D^\alpha k^\beta (x) | \le c \, (1+|\beta|)^K \le c_{\delta} \, 2^{\delta |\beta|}, \qquad \beta \in \nat^n_0, \quad \alpha \in \nat^n_0, \quad |\alpha| \le K,
\end{\eq}
and \eqref{2.4}.
\cm
{\em Step 2.} Part (ii) follows from
\begin{\eq}   \label{2.19}
\big| \Phi^\beta_{j,m} (x) \big| \le c_{\vk} \, 2^{-\vk \, |\beta|} \big(1 + |2^j x-m| \big)^{-N}, \qquad x\in \rn,
\end{\eq}
for any $\vk \in \real$, some $c_{\vk} >0$ and all $\beta \in \nat^n_0$. This assertion is covered by \cite[(3.132),
 p.\,171]{T06}. It relies on \eqref{2.5}--\eqref{2.11}, Stirling's formula applied to $\beta!$ and the crucial observation that for
any $a>0$ there are constants $C>0$ and $c_a >0$ such that
\begin{\eq}   \label{2.20}
\big| D^\beta \omega^\vee (x) \big| \le c_a \, 2^{C |\beta|} \, \big( 1 + |x|^2 \big)^{-a}, \qquad x\in \rn, \quad \beta \in \nat^n_0,
\end{\eq}
\cite[(3.115), p.\,167]{T06}. A short direct proof of this estimate may be found in \cite[Remark 3.22, pp.\,168--169]{T06}. An earlier
version may be found in \cite[pp.\,18--19]{T01} based on Cauchy's representation theorem for holomorphic  functions of several 
variables.
\end{proof}

We return to the two quarkonial systems introduced in \eqref{2.13} and \eqref{2.14}. It follows part (i) of the above proposition that 
\begin{\eq}   \label{2.21}
\Big\{ 2^{\vk |\beta|} \, k^\beta_{j,m} : \ \beta \in \nat^n_0, \ j\in \no, \ m\in \zn \Big\}, \qquad \vk < \ve,
\end{\eq}
is a collection of $(K,0,C,d)$--atomic systems according to \eqref{1.25}, based on Definition \ref{D1.3} for a suitably chosen (but
immaterial) $d>1$, uniformly in $\beta \in \nat^n_0$. Let
\begin{\eq}   \label{2.21a}
\Big\{ 2^{\vk |\beta|} \, \Phi^\beta_{j,m} : \ \beta \in \nat^n_0, \ j\in \no, \ m\in \zn \Big\}, \qquad \vk \in \real,
\end{\eq}
be the generalization of the quarkonial system \eqref{2.14}. One has by \eqref{2.10} and \eqref{2.11} that
\begin{\eq}   \label{2.22}
\int_{\rn} x^\gamma \Phi^\beta_{j,m} (x) \, \di x =0 \quad \text{for all $j\in \nat$, $\beta \in \nat^n_0$, $\gamma \in \nat^n_0$, $m
\in \zn$}.
\end{\eq}
Then it follows from part (ii) of Proposition \ref{P2.1} and Remark \ref{R1.9} that \eqref{2.21a} is a collection of admitted molecular
systems for the spaces $\As (\rn)$ with $s<0$, uniformly in $\beta \in \nat^n_0$.

\subsection{Representations: positive smoothness}    \label{S2.2}
We ask for quarkonial counterparts of the atomic representations in the spaces $\As (\rn)$ as described in Theorem \ref{T1.7} and of
the wavelet decompositions according to Theorem \ref{T1.12}. For this purpose one has first to introduce  related sequence spaces as in
the Definitions \ref{D1.5} and \ref{D1.10}.

Let again $\chi_{j,m}$ be the characteristic function of the cube $Q_{j,m} = 2^{-j}m + 2^{-j} (0,1)^n$ where $j\in \no$ and $m\in \zn$.
We adapt the sequence spaces $a^s_{p,q} (\rn)$, $a\in \{b,f \}$, as introduced in Definition \ref{D1.5} to the quarkonial systems
\eqref{2.21} and \eqref{2.21a} incorporating the additional parameter $\beta \in \nat^n_0$.

\begin{definition}   \label{D2.2}
Let $s\in \real$ and $0<p,q \le \infty$ $(p<\infty$ for $f$--spaces$)$. Let $\vk \in \real$. Then $b^s_{p,q} (\rn)^{\vk}$ is the
collection of all sequences
\begin{\eq}   \label{2.23}
\lambda = \{ \lambda^\beta: \ \beta \in \nat^n_0 \}, \quad \lambda^\beta = \{ \lambda^\beta_{j,m} \in \comp: \ j\in \no, \ m \in \zn 
\}
\end{\eq}
such that
\begin{\eq}   \label{2.24}
\| \lambda \, | b^s_{p,q} (\rn)^{\vk} \| = \sup_{\beta \in \nat^n_0} 
2^{\vk |\beta|} \bigg( \sum^\infty_{j=0} 2^{j(s- \frac{n}{p})q} \Big( \sum_{m\in \zn}
|\lambda^\beta_{j,m} |^p \Big)^{q/p} \bigg)^{1/q}
\end{\eq}
is finite and $f^s_{p,q} (\rn)^{\vk}$ is the collection of all sequences according to \eqref{2.23} such that
\begin{\eq}   \label{2.25}
\| \lambda \, | f^s_{p,q} (\rn)^{\vk} \| = \sup_{\beta \in \nat^n_0} 2^{\vk |\beta|}
\bigg\| \Big( \sum^\infty_{j=0} \sum_{m \in \zn} 2^{jsq} \big| 
\lambda^\beta_{j,m} \, \chi_{j,m} (\cdot) \big|^q \Big)^{1/q} \, | L_p (\rn) \bigg\|
\end{\eq}
is finite $($with the usual modifications if $\max (p,q) = \infty).$
\end{definition}

\begin{remark}   \label{R2.3}
This is the counterparts of the Definitions \ref{D1.5} and \ref{D1.10}. Quite obviously, $a^s_{p,q} (\rn)^{\vk}$ with $a \in \{b,f \}$ 
are quasi--Banach spaces (Banach spaces if $p \ge 1$, $q \ge 1$).
\end{remark}

We wish to extend Theorem \ref{T1.7} from the qualitative atomic systems as introduced in \eqref{1.25} to the more constructive 
quarkonial  systems according to \eqref{2.21}, \eqref{2.21a}. Let $k^\beta_{j,m}$ and $\Phi^\beta_{j,m}$ be as there. Let $\ve >0$
be as in \eqref{2.15} and Proposition \ref{P2.1}. Let again
\begin{\eq}   \label{2.26}
\sigma^n_p = n \Big( \max \big( \frac{1}{p}, 1 \big) -1 \Big), \qquad \sigma^n_{p,q} = n \Big( \max \big( \frac{1}{p}, \frac{1}{q}, 1
\big) - 1 \Big),
\end{\eq}
where $n\in \nat$ and $0<p,q \le \infty$. Let $\As (\rn)$ be the spaces as introduced in Definition \ref{D1.1} and Remark \ref{R1.2}.

\begin{theorem} \label{T2.4}
Let $0<p,q \le \infty$ $(p<\infty$ for $F$--spaces$)$,
\begin{\eq}   \label{2.27}
s> \sigma^{n}_p \ \text{for $B$--spaces}, \qquad s>\sigma^{n}_{p,q} \ \text{for $F$--spaces},
\end{\eq}
and $\vk >-\ve$. 
Then $f\in S'(\rn)$ belongs to $\As (\rn)$ if, and only if, it can be represented as
\begin{\eq}   \label{2.28}
f = \sum_{\substack{\beta \in \nat^n_0, j\in \no, \\ m \in \zn}} \lambda^\beta_{j,m} \, k^\beta_{j,m}, \qquad \lambda \in a^s_{p,q}(\rn)^{\vk},
\end{\eq}
unconditional convergence being in $S'(\rn)$. Furthermore,
\begin{\eq}   \label{2.29}
\| f \, | \As (\rn) \| \sim \inf \| \lambda \, | a^s_{p,q} (\rn)^{\vk} \|
\end{\eq}
where the infimum is taken over all admissible representations \eqref{2.28}. Let
\begin{\eq}   \label{2.30}
\lambda^\beta_{j,m} (f) = 2^{jn} \, \big(f, \Phi^\beta_{j,m} \big), \qquad f \in S' (\rn),
\end{\eq}
and
\begin{\eq}   \label{2.31}
\lambda (f) = \{ \lambda^\beta (f): \ \beta \in \nat^n_0 \}, \quad \lambda^\beta (f) = \{ \lambda^\beta_{j,m} (f): \ j \in \no, \ m\in
\zn \}.
\end{\eq}
Then $f \in \As (\rn)$ can be represented as
\begin{\eq}   \label{2.32}
f = \sum_{\substack{\beta \in \nat^n_0, j\in \no, \\ m \in \zn}} \lambda^\beta_{j,m}(f) \, k^\beta_{j,m}, 
\end{\eq}
unconditional convergence being in $S'(\rn)$, and
\begin{\eq}   \label{2.33}
\| f \, | \As (\rn) \| \sim \| \lambda (f) \, | a^s_{p,q} (\rn)^{\vk} \|
\end{\eq}
$($equivalent quasi--norms$)$.
\end{theorem}

\begin{proof}
{\em Step 1.} We proved in \cite[Theorem 3.21, pp.\,165--169]{T06} the above theorem for the spaces $B^s_p
(\rn) = B^s_{p,p} (\rn)$, $0<p \le \infty$, $s > \sigma^{n}_p$, including the representability \eqref{2.32} with \eqref{2.30}. The
usual embeddings
ensure the representation \eqref{2.32} in all spaces covered by \eqref{2.27}. In other words, it remains to extend this assertion
to all these spaces.
\cm
{\em Step 2.} We rely on atomic representations as described in Theorem \ref{T1.7}. By \eqref{2.27} we may choose $L=0$, which means
that no moment conditions are required. Then it follows from Definition \ref{D1.3} and Proposition \ref{P2.1}(i) that
\begin{\eq}   \label{2.34}
\big\{ 2^{-\vr \, |\beta|} \, k^\beta_{j,m}: \ j\in \no, \ m \in \zn \big\}, \qquad \beta \in \nat^n_0,
\end{\eq}
with $-\vr < \ve$ are admitted atoms, uniformly in $\beta \in \nat^n_0$. Recall that $\As (\rn)$ is a $u$--Banach space where $u= 
\min (1,p,q)$. Let $\vk > \vr > -\ve$. Then it follows from \eqref{2.28}, \eqref{2.23} and Theorem \ref{T1.7} that
\begin{\eq}   \label{2.35}
\| f \, | \As (\rn) \|^{u} \le c \sum_{\beta \in \nat^n_0} 2^{\vr |\beta|u} \, \| \lambda^\beta \, | a^s_{p,q} (\rn) \|^{u} \le c'\,
\|\lambda \, | a^s_{p,q} (\rn)^{\vk} \|^{u}.
\end{\eq}
\cm
{\em Step 3.} By \eqref{2.21a} and \eqref{2.22} one could interpret \eqref{2.30} as the molecular version of the local means for the spaces
$\As (\rn)$ according to \cite[Definition 1.9, Theorem 1.15, pp.\,6--7]{T08}. But this is not stated there explicitly. For this reason
we decompose the kernels $2^{jn} \Phi^\beta_{j,m}$ into atomic pieces. Let $\psi$ be a 
compactly supported $C^\infty$ function in $\rn$ such that
\begin{\eq}  \label{2.36}
1 = \sum_{k\in \zn} \psi(x-k), \qquad x\in \rn.
\end{\eq}
Let $j\in \nat$. The terms with $j=0$ can be incorporated afterwards. Let $(-\Delta)^L = \big( - \sum^n_{l=1} \pa^2_l \big)^L$, $L\in
\nat$, and let $\Phi^\beta_{j,m}$ be given by \eqref{2.10}, \eqref{2.11}. Then the summation in
\begin{\eq}   \label{2.37}
\begin{aligned}
 \Phi^\beta_{j,m} (x) &= \big( F F^{-1} \Phi^\beta_{j,m} \big)(x) \\
&= \sum_{k\in \zn} (-\Delta)^L \Big[ \psi (2^j x-m-k) F \Big( |\xi|^{-2L} \big( F^{-1} \Phi^\beta_{j,m} \big)(\xi) \Big) (x)
\Big]
\end{aligned}
\end{\eq}
will be justified by what follows. One has by \eqref{2.10}, \eqref{2.11},
\begin{\eq}   \label{2.38}
\begin{aligned}
\big( F^{-1} \Phi^\beta_{j,m} \big)(\xi) &= c \, \int_{\rn} e^{ix \xi} \Phi^\beta_M (2^j x -m ) \, \di x \\
&= c\, 2^{-jn} \int_{\rn} e^{i 2^{-j} x \xi} \Phi^\beta_M (x-m) \, \di x \\
&= c \, 2^{-jn} \, e^{i 2^{-j} m \xi} \int_{\rn} e^{i 2^{-j} x \xi} \, \Phi^\beta_M (x) \, \di x \\
&= 2^{-jn} e^{i 2^{-j} m \xi} \big( \Phi^\beta_M \big)^\vee ( 2^{-j} \xi) \\
&= 2^{-jn} e^{i 2^{-j} m \xi} \vp^0 (2^{-j} \xi) \, \Om^\beta (2^{-j} \xi ).
\end{aligned}
\end{\eq}
Using the harmless modification $\vp^L (\xi) = |\xi|^{-2L} \vp^0 (\xi)$ of $\vp^0 (\xi)$ one obtains that
\begin{\eq}   \label{2.39}
\begin{aligned}
F \Big( |\xi|^{-2L} \big( F^{-1} \Phi^\beta_{j,m} \big) (\xi) \Big) (x) & = 2^{-2Lj - jn} F \Big( e^{i 2^{-j} m \xi} \vp^L
(2^{-j} \xi) \,\Om^\beta (2^{-j} \xi) \Big)(x) \\
&= 2^{-2Lj -jn} F \Big( \vp^L (2^{-j} \xi)\,\Om^\beta (2^{-j} \xi) \Big) (x- 2^{-j} m) \\
&= 2^{-2Lj} F \Big( \vp^L (\xi)\,\Om^\beta (\xi) \Big) (2^j x - m)
\end{aligned}
\end{\eq}
belongs to $S(\rn)$. Inserted in \eqref{2.37} one obtains
\begin{\eq}   \label{2.40}
\Phi^\beta_{j,m}(x) = \sum_{k \in \zn} a^{L,\beta,k}_{j,m} (x)
\end{\eq}
with
\begin{\eq}   \label{2.41}
a^{L,\beta,k}_{j,m} (x) = \Big[ (-\Delta)^L \Big( \psi( \cdot -k) \, F \big( \vp^L (\xi)\,\Om^\beta (\xi) \big) (\cdot) \Big)
\Big] (2^j x-m).
\end{\eq}
The arguments resulting in \eqref{2.12}, Proposition \ref{P2.1}(ii) and \eqref{2.19} do not change if one replaces $\vp^0$ in \eqref{2.10}
by $\vp^L$. In particular one has the same exponential decay with respect to $\beta \in \nat^n_0$. Furthermore $\psi(\cdot -k)$ applied
to a function belonging to $S(\rn)$ decays as $(1 + |k|^2)^{-D/2} = \langle k \rangle^{-D}$ for any $D>0$. This shows that
\begin{\eq}   \label{2.42}
\langle k \rangle^D 2^{\vk |\beta|} a^{L,\beta,k}_{j,m} \qquad \text{are $L_\infty$--normalized}
\end{\eq}
classical atoms according to Definition \ref{D1.3}
for any $D>0$ and any $\vk >0$, located at $2^{-j} (m+k)$ having moment conditions  up to order $2L-1$. Multiplied with
$2^{jn}$ they may serve as kernels of local means to which \cite[Theorem 1.15, p.\,7]{T08} can be applied based on the atomic sequence
spaces $a^s_{p,q} (\rn)$ in Definition \ref{D1.5}. They are $u$--Banach space with $u= \min (1,p,q)$. Then it follows from 
\eqref{2.30}, \eqref{2.31} and \eqref{2.40}, \eqref{2.42} that
\begin{\eq}   \label{2.43}
\| \lambda^\beta (f) \, | a^s_{p,q} (\rn) \| \le c \, 2^{-\vk \, |\beta|} \| f \, | \As (\rn) \|, \qquad f \in \As (\rn),
\end{\eq}
uniformly in $\beta \in \nat^n_0$ where $\vk \in \real$ is at our disposal. This shows in the same way as in \eqref{2.35} that
\begin{\eq}   \label{2.44}
\| \lambda (f) \, | a^s_{p,q} (\rn)^{\vk} \| \le c \, \|f \, | \As (\rn) \|, \qquad \vk \in \real.
\end{\eq}
Now \eqref{2.35}, the already ensured representability \eqref{2.32} and \eqref{2.44} prove the theorem.
\end{proof}

\begin{remark}   \label{R2.5}
The restriction $\vk > - \ve$ in the above theorem
comes from Proposition \ref{P2.1} (with $-\vk$ in place of $\vk$)
 and \eqref{2.34}. Of interest later on is the case $\vk =0$ what 
simplifies \eqref{2.24}, \eqref{2.25} somewhat. In any case one can replace $\ell_\infty$ with respect to $\beta \in \nat^n_0$ in
\eqref{2.24}, \eqref{2.25} by $\ell_r$ with $r>0$. This will be of some service for us later on.
\end{remark}

Of special interest for us are the H\"{o}lder--Zygmund spaces
\begin{\eq}   \label{2.45}
\Cc^s (\rn) = B^s_{\infty, \infty} (\rn), \qquad s \in \real.
\end{\eq}
according to \eqref{1.14}. The related specifications in the Theorems \ref{T1.7} and \ref{T1.12} show how corresponding atomic and
wavelet representations look like. It will be of interest for us later on to fix the quarkonial counterpart as far as covered by
Theorem \ref{T2.4}. By \eqref{2.24} the space
$\text{\cc} = b^s_{\infty, \infty} (\rn)^0$ collects all sequences \eqref{2.23},
\begin{\eq}   \label{2.46}
\lambda = \big\{ \lambda^\beta_{j,m} \in \comp : \ \beta \in \nat^n_0, \ j \in \no,  \ m \in \zn \big\}
\end{\eq}
such that
\begin{\eq}  \label{2.47}
\| \lambda \, | \text{\cc} \| = \sup_{\substack{\beta \in \nat^n_0, j\in \no, \\ m \in \zn}} 2^{js} \, \big| \lambda^\beta_{j,m} \big|
\end{\eq}
is finite. We specify Theorem \ref{T2.4}. In particular the quarkonial building blocks $k^\beta_{j,m}$ and $\Phi^\beta_{j,m}$ have the
same meaning as there, \eqref{2.4}, \eqref{2.11}.

\begin{corollary}   \label{C2.6}
Let $s>0$. Then $f\in S'(\rn)$ belongs to $\Cc^s (\rn)$ if, and only if, it can be represented as
\begin{\eq}  \label{2.47a}
f = \sum_{\substack{\beta \in \nat^n_0, j\in \no, \\ m \in \zn}} \lambda^\beta_{j,m} \, k^\beta_{j,m}, \qquad \lambda \in \text{\cc},
\end{\eq}
unconditional convergence being in $S'(\rn)$. Furthermore,
\begin{\eq}   \label{2.48}
\| f \, | \Cc^s (\rn) \| \sim \inf \| \lambda \, | \text{\cc} \|
\end{\eq}
where the infimum is taken over all admissible representations  \eqref{2.47a}. Let
\begin{\eq}   \label{2.49}
\lambda^\beta_{j,m} (f) = 2^{jn} \big( f , \Phi^\beta_{j,m} \big), \qquad f \in S'(\rn),
\end{\eq}
and
\begin{\eq}   \label{2.50}
\lambda (f) = \big\{ \lambda^\beta_{j,m} (f): \ \beta \in \nat^n_0, \  j\in \no, \ m\in \zn \big\}.
\end{\eq}
Then $f\in \Cc^s (\rn)$ can be represented as
\begin{\eq}   \label{2.51}
f = \sum_{\substack{\beta \in \nat^n_0, j\in \no, \\ m \in \zn}} \lambda^\beta_{j,m} (f) \, k^\beta_{j,m}, 
\end{\eq}
unconditional convergence being in $S'(\rn)$, and
\begin{\eq}   \label{2.52}
\| f \, | \Cc^s (\rn) \| \sim \| \lambda (f) \, |  \text{\cc} \|
\end{\eq}
$($equivalent norms$)$.
\end{corollary}

\begin{proof}
This is a special case of Theorem \ref{T2.4} based on \eqref{2.45}.
\end{proof}

So far the arguments ensure that
the series in \eqref{2.28} and \eqref{2.32} converge unconditionally in $S'(\rn)$. But this can be improved if both $p<\infty$ 
and $q<\infty$. Let again $\ve >0$ be as in \eqref{2.15} and Proposition \ref{P2.1}. 

\begin{corollary}  \label{C2.7}
Let $0<p,q <\infty$ and let $s$ be as in \eqref{2.27}. Let $\vk > - \ve$. Then the series in \eqref{2.28} and \eqref{2.32} converge in $\As (\rn)$ unconditionally to $f$.
\end{corollary}

\begin{proof}
The assertion follows by standard arguments, Lebesgue's bounded convergence theorem for $L_p (\rn)$, $0<p<\infty$, in the case of the
$F$--spaces, and that $\vk > -\ve$ is at our disposal (or the replacement of $\ell_\infty$ by $\ell_r$, $0<r<\infty$, as indicated in
Remark \ref{R2.5}).
\end{proof}

\begin{remark}   \label{R2.8}
In other words, \eqref{2.32} is a {\em frame representation} in $\As (\rn)$ if $0<p,q<\infty$ and $s$ as in \eqref{2.27}. One should be
aware that the two systems \eqref{2.13} and \eqref{2.14} (with their indicated individual properties) are totally independent to each
other. This follows from the arguments in
\cite[(3.110), (3.111), p.\,167]{T06}. In particular, \eqref{2.32}, \eqref{2.51} are frame representations
and \eqref{2.28}, \eqref{2.47a} are not uniquely determined, in contrast, for example, to wavelet expansions according to Theorem
\ref{T1.12}. 
On the other hand, it is
quite clear that $\beta =0$ is the dominant term. This follows from the strong decay of the elements both in \eqref{2.13} and 
\eqref{2.14} with respect to $\beta$. This has already been used above in related estimates and will be also later on of great service
for us.
\end{remark}

\subsection{Representations: negative smoothness}    \label{S2.3}
The representability \eqref{2.32} for the spaces $\As (\rn)$ covered by Theorem \ref{T2.4}  has been taken over from \cite{T06}.
By \cite[Theorem 3.26, p.\,172]{T06} there is a related counterpart for some spaces $\As (\rn)$ with $s<0$. But we wish to have a fresh
look at this topic adapted to our later needs. We rely on the duality
\begin{\eq}   \label{2.53}
\Cc^s (\rn) = B^s_{\infty, \infty} (\rn) = B^{-s}_{1,1} (\rn)', \qquad s \in \real,
\end{\eq}
in the framework of the dual pairing  $\big( S(\rn), S'(\rn) \big)$, \cite[Theorem 2.11.2, p.\,178]{T83}. Then
\begin{\eq}   \label{2.54}
f_l \to f, \qquad l \to \infty, \quad \text{in the weak*--topology of $\Cc^s (\rn)$}
\end{\eq}
for $\{ f_l \} \subset \Cc^s (\rn)$ and $f\in \Cc^s (\rn)$ means
\begin{\eq}  \label{2.55}
( f_l, g ) \to (f,g) \qquad \text{for any $g\in B^{-s}_{1,1} (\rn)$}.
\end{\eq}
Let $\{ k^\beta_{j,m} \}$ and $\{ \Phi^\beta_{j,m} \}$ be the quarkonial systems as introduced in \eqref{2.13} and \eqref{2.14}.

\begin{proposition}   \label{P2.9}
Let $s<0$. Then $f \in \Cc^s (\rn)$ can be represented in the weak*--topology of $\Cc^s(\rn)$ as
\begin{\eq}   \label{2.56}
f = \sum_{\substack{\beta \in \nat^n_0, j\in \no, \\ m \in \zn}} 2^{jn}\,\big( f, k^\beta_{j,m} \big) \Phi^\beta_{j,m}.
\end{\eq}
\end{proposition}

\begin{proof}
According to Corollary \ref{C2.7} one can expand $g\in B^{-s}_{1,1} (\rn)$ by the unconditionally converging series
\begin{\eq}   \label{2.57}
g = \sum_{\substack{\beta \in \nat^n_0, j\in \no, \\ m \in \zn}} 2^{jn}\,\big( g, \Phi^\beta_{j,m} \big) k^\beta_{j,m}.
\end{\eq}
Let $f\in \Cc^s (\rn)$. Then one has
\begin{\eq}   \label{2.58}
(f,g) = \Big( f, \sum_{\beta,j,m} 2^{jn} \big(g, \Phi^\beta_{j,m} ) k^\beta_{j,m} \Big) 
= \Big( \sum_{\beta,j,m} 2^{jn} \big( f, k^\beta_{j,m} \big) \Phi^\beta_{j,m}, g \Big)
\end{\eq}
for any $g \in B^{-s}_{1,1} (\rn)$. This proves the proposition.
\end{proof}

In particular, it follows from well--known embeddings that any 
\begin{\eq}   \label{2.59}
f \in \As (\rn) \quad \text{with} \quad  A \in \{B,F \}, \quad s\in \real \quad \text{and} \quad 0<p,q \le \infty
\end{\eq}
can be represented as
\begin{\eq}   \label{2.60}
f = \sum_{\substack{\beta \in \nat^n_0, j\in \no, \\ m \in \zn}} 2^{jn}\,\big( f, k^\beta_{j,m} \big) \Phi^\beta_{j,m},
\end{\eq}
unconditional convergence being in $S'(\rn)$. There are two questions.
\begin{itemize}
\item Can the expansion \eqref{2.60} be used to characterize the spaces $\As (\rn)$ to which $f$ belongs?
\item Can \eqref{2.60} be extended to all $f\in S'(\rn)$?
\end{itemize}

For the second question we have an affirmative answer. This will be done in Section \ref{S2.6}. For this purpose we extend in the
Sections \ref{S2.4} and \ref{S2.5}
the theory developed in the Sections \ref{S2.2}, \ref{S2.3} from $\As (\rn)$ to $\As (\rn,w)$ where $w$ is a so--called
admissible weight. As far as the first question is concerned we remarked already at the end of Section \ref{S2.1}
that the functions $\Phi^\beta_{j,m}$ in \eqref{2.21a} are universal 
$L_\infty$--normalized molecules in the spaces $\As (\rn)$ with $s<0$, whereas $2^{jn} k^\beta_{j,m}$ can be interpreted as kernels 
of related local means according to \cite[Theorem 1.15, p.\,7]{T08}.
This will be detailed in what follows. A direct extension of this theory to spaces $\As (\rn)$ with $s>0$ is
at least doubtful. But we describe later on in  in Corollary \ref{C2.13} a substitute. 

Let
\begin{\eq}   \label{2.61}
a^s_{p,q} (\rn)^{\vk} \quad \text{with $a\in \{b,f \}$, $s\in \real$ and $0<p,q \le \infty$},
\end{\eq}
($p < \infty$ for $f$--spaces) be the sequence spaces according to Definition \ref{D2.2}. Let $\ve >0$ be as in \eqref{2.15} and Proposition \ref{P2.1}.

\begin{theorem}   \label{T2.10}
Let  $0<p,q \le \infty$ $(p<\infty$ for $F$--spaces$)$, $s<0$ and $\vk < \ve$. Then $f\in S'(\rn)$ belongs to $\As (\rn)$ if, and only if, it can be represented as
\begin{\eq}   \label{2.62}
f = \sum_{\substack{\beta \in \nat^n_0, j\in \no, \\ m \in \zn}} \lambda^\beta_{j,m} \, \Phi^\beta_{j,m}, \qquad \lambda \in a^s_{p,q} (\rn)^{\vk},
\end{\eq}
unconditional convergence being in $S'(\rn)$. Furthermore,
\begin{\eq}   \label{2.63}
\| f \, | \As (\rn) \| \sim \inf \| \lambda \, | a^s_{p,q} (\rn)^{\vk} \|
\end{\eq}
where the infimum is taken over all  admissible representations \eqref{2.62}. Let
\begin{\eq}   \label{2.64}
\lambda^\beta_{j,m} (f) = 2^{jn} \, \big(f, k^\beta_{j,m} \big), \qquad f \in S' (\rn),
\end{\eq}
and
\begin{\eq}   \label{2.65}
\lambda (f) = \{ \lambda^\beta (f): \ \beta \in \nat^n_0 \}, \quad \lambda^\beta (f) = \{ \lambda^\beta_{j,m} (f): \ j \in \no, \ m\in
\zn \}.
\end{\eq}
Then $f \in \As (\rn)$ can be represented as
\begin{\eq}   \label{2.66}
f = \sum_{\substack{\beta \in \nat^n_0, j\in \no, \\ m \in \zn}} \lambda^\beta_{j,m}(f) \, \Phi^\beta_{j,m}, 
\end{\eq}
unconditional convergence being in $S'(\rn)$, and
\begin{\eq}   \label{2.67}
\| f \, | \As (\rn) \| \sim \| \lambda (f) \, | a^s_{p,q} (\rn)^{\vk} \|
\end{\eq}
$($equivalent quasi--norms$)$.
\end{theorem}

\begin{proof}
{\em Step 1.} Let $f$ be given by \eqref{2.62}. We argue similarly as in Step 2 of the proof of Theorem \ref{T2.4}, but change the 
roles played by $\Phi^\beta_{j,m}$ and $k^\beta_{j,m}$. In particular according to the end of Section \ref{S2.1},
\begin{\eq}   \label{2.68}
\Big\{ 2^{-\vr |\beta|} \Phi^\beta_{j,m}: \ j\in \no, \ m\in \zn \Big\}, \qquad \beta \in \nat^n_0,
\end{\eq}
with $\vr < \vk$ are $L_\infty$--normalized universal molecules in all spaces $\As (\rn)$ with $s<0$. Likewise one could break them as
in \eqref{2.40}, \eqref{2.41} into related atoms having sufficiently many moment conditions. Otherwise one can argue similarly as in
\eqref{2.35} ensuring $f\in \As (\rn)$  as a consequence of
\begin{\eq}   \label{2.69}
\| f \, | \As (\rn) \|^{u} \le c \, \sum_{\beta \in \nat^n_0} 2^{\vr |\beta| u} \, \| \lambda^\beta \, | a^s_{p,q} (\rn) \|^{u} 
\le c' \, \| \lambda \, | a^s_{p,q} (\rn)^{\vk} \|^{u}.
\end{\eq}
{\em Step 2.} As mentioned in \eqref{2.59}, \eqref{2.60} any $f\in \As (\rn)$ can be represented as \eqref{2.66} with \eqref{2.64}. 
For fixed $\beta \in \nat^n_0$ we interpret the coefficients $\lambda^\beta_{j,m} (f)$ of $\lambda^\beta (f)$ as local means according
to \cite[Theorem 1.15, p.\,7]{T08} for the above spaces $\As (\rn)$ with $s<0$. 
No moment conditions are needed. Using $\vk < \ve$ one has by \eqref{2.16} that
\begin{\eq}   \label{2.70}
2^{\vk |\beta|} \, \| \lambda^\beta (f) \, | a^s_{p,q} (\rn) \| \le c \, \|f \, | \As (\rn) \|, \qquad \beta \in \nat^n_0,
\end{\eq}
uniformly in $\beta$. Then both \eqref{2.63} and \eqref{2.67} follow from \eqref{2.69}, \eqref{2.70}.
\end{proof}

It is quite clear that $\beta =0$ is the main term in \eqref{2.61} and \eqref{2.67}, based on Definition \ref{D2.2}. However in
the representation \eqref{2.62} one needs all $\beta \in \nat^n_0$. But one may ask whether the equivalence \eqref{2.67} remains valid
if one replaces the right--hand side by its main term, which means
\begin{\eq}   \label{2.71}
\| f \, | \As (\rn) \| \sim \| \lambda^0 (f) \, | a^s_{p,q} (\rn) \|
\end{\eq}
with $\lambda^0 (f)$ as in \eqref{2.65} and
$a^s_{p,q} (\rn)$ as in Definition \ref{D1.5}, Remark \ref{R1.6}. This can be done under some mild additional restrictions for the function $k$ as introduced in \eqref{2.1}, \eqref{2.2}. Let
\begin{\eq}   \label{2.72}
k(t,f)(x) = t^{-n} \int_{\rn} k \Big( \frac{y-x}{t} \Big) \, f(y) \, \di y, \qquad x\in \rn, \quad t>0,
\end{\eq}
be the continuous counterpart of the above local means \eqref{2.64}
as considered in \cite[Section 1.4]{T06}. In particular one has by \eqref{2.64}, \eqref{2.65} and \eqref{2.4} that
\begin{\eq}   \label{2.73}
\begin{aligned}
k(2^{-j}, f) (2^{-j}m) &= 2^{jn} \int_{\rn} k(2^jy -m) \, f(y) \, \di y \\
&= 2^{jn} \, (f, k^0_{j,m} ) \\
&= \lambda^0_{j,m} (f), \qquad j\in \no, \quad m\in \zn.
\end{aligned}
\end{\eq}
Let $a^s_{p,q} (\rn)$ be as in Definition \ref{D1.5}, Remark \ref{R1.6}.

\begin{proposition}  \label{P2.11}
Let $s<0$ and $0<p,q \le \infty$ $(p<\infty$ for $F$--spaces$)$. Let
\begin{\eq}  \label{2.74}
\lambda^0 (f) = \{ \lambda^0_{j,m} (f): \ j\in \no, \ m\in \zn \}
\end{\eq}
be based on \eqref{2.73} with $k$ as above and $|\nabla k(x)| \le \delta$ for some $\delta >0$. There is a number $\delta_0 >0$ such
that for all $0<\delta \le \delta_0$,
\begin{\eq}   \label{2.75}
\| f \, | \As (\rn) \| \sim \| \lambda^0 (f) \, | a^s_{p,q} (\rn) \|
\end{\eq}
$($equivalent quasi--norms$)$.
\end{proposition}

\begin{remark}   \label{R2.12}
This assertion remains valid without the additional
assumption $|\nabla k(x)| \le \delta$ if one replaces $k(2^{-j}, f)(2^{-j}m)$ in \eqref{2.73}
by $k(2^{-j}, f)(r 2^{-j}m)$ according to \eqref{2.72} with $0<r \le r_0$ and some sufficiently small $r_0$. 
Both versions are essentially covered  by \cite[Theorem 4, pp.\,16--17]{Win95}.
Basically the proofs rely on the characterization of the spaces $\As (\rn)$ with $s<0$ by the above local means and related maximal
functions as described in \cite[Corollary 1.12, p.\,11]{T06}. An explicit formulation, again with a reference to \cite{Win95}, for the
spaces $B^s_p (\rn) = B^s_{p,p} (\rn)$, $0<p \le \infty$, $s<0$, may be found in \cite[Remark 3.28, p.\,174]{T06}. In other words, if 
one wishes to replace \eqref{2.67} by \eqref{2.75} then one needs some mild  extra conditions for the basic function $k$.
\end{remark}

Although the representation \eqref{2.59}, \eqref{2.60} applies to all spaces $\As (\rn)$ one cannot expect that Theorem \ref{T2.10} can be extended to spaces $\As (\rn)$ with $s \ge 0$. According to \cite[Theorem 1.15, p.\,7]{T08} (and the references given there) a
characterization of spaces $\As (\rn)$ with $s \ge 0$ in terms of local means requires moment conditions for the related kernels. But
the functions $k^\beta_{j,m}$ in \eqref{2.3}, \eqref{2.4} are non--negative. However some rescue comes from the following observation: Let $f \in\As (\rn)$ where
\begin{\eq}   \label{2.76}
A \in \{B,F \}, \qquad s \in \real \quad \text{and} \quad 0<p,q \le \infty.
\end{\eq}
Let $N \in \no$. Then
\begin{\eq}   \label{2.77}
\| f \, | \As (\rn) \| \sim \sup_{0 \le |\alpha| \le N} \| D^\alpha f \, | A^{s-N}_{p,q} (\rn) \|
\end{\eq}
(equivalent quasi--norms). If $p<\infty$ for the $F$--spaces then this is a very classical assertion is covered by 
\cite[Theorem 2.3.8, pp.\,58--59]{T83}. But it remains also valid for
$F^s_{\infty,q} (\rn)$, \cite[Theorem 1.24, p.\,17]{T20} (however these  spaces are not considered here in detail). Theorem \ref{T2.10}
can now be complemented as follows. We use the same notation as there.

\begin{corollary}   \label{C2.13}
Let
\begin{\eq}    \label{2.78}
A \in \{B,F \}, \quad -\infty <s<N \in \no \quad \text{and} \quad 0<p,q \le \infty
\end{\eq}
$(p<\infty$ for $F$--spaces$)$. Let $\vk < \ve$. Then any $f\in \As (\rn)$ can be represented as 
\begin{\eq}   \label{2.79}
f = \sum_{\substack{\beta \in \nat^n_0, j\in \no, \\ m \in \zn}} 2^{jn} \big( f, k^\beta_{j,m} \big) \Phi^\beta_{j,m}, 
\end{\eq}
such that
\begin{\eq}   \label{2.80}
\| f \, | \As (\rn) \| \sim \sup_{0 \le |\alpha| \le N} \| \lambda (D^\alpha f) \, | a^{s-N}_{p,q} (\rn)^{\vk} \|
\end{\eq}
$($equivalent quasi--norms$)$.
\end{corollary}

\begin{proof}
This assertion follows from \eqref{2.77} and Theorem \ref{T2.10} applied to $D^\alpha f \in A^{s-N}_{p,q} (\rn)$.
\end{proof}

\begin{remark}   \label{R2.14}
By \eqref{2.64} one has
\begin{\eq} \label{2.81}
\lambda^\beta_{j,m} (D^\alpha f) = 2^{jn} \big( D^\alpha f, k^\beta_{j,m} \big) = (-1)^{|\alpha|} 2^{jn} \big( f, D^\alpha k^\beta_{j,m}\big).
\end{\eq}
At least $D^\alpha k^\beta_{j,m} \in D(\rn) = C^\infty_0 (\rn)$ with $|\alpha| =N$
are kernels of related local means for $\As (\rn)$
with sufficiently many moment conditions as requested in \cite[Theorem 1.15, p.\,7]{T08}. The
outcome is satisfactory. It clarifies under which circumstances $f\in S'(\rn)$ represented according to \eqref{2.79}
belongs to $\As (\rn)$. Furthermore Theorem \ref{T2.10} and Corollary \ref{C2.13} complement the theory of local means in the
spaces $\As (\rn)$ as described in \cite[Section 1.4, pp.\,9--12]{T06} based on related references. There we preferred continuous
versions with $\ell^s_q \big( L_p (\rn) \big)$ in place of $b^s_{p,q}(\rn)$ and $L_p \big( \rn, \ell^s_q \big)$ in place of $f^s_{p,q} (\rn)$ with $a^s_{p,q} (\rn)$ as in Definition \ref{D1.5}, Remark \ref{R1.6}.
But as already indicated in Remark \ref{R2.12} the characterizations  in
\cite{T06} in terms of maximal functions allow to step from corresponding continuous versions to their related discrete counterparts 
used above. We refer in this context again to \cite[Chapter 2]{Win95}.
\end{remark}

\begin{problem}    \label{P2.15}
One may ask whether the above considerations can be extended to other classes of distinguished spaces. As far as  isotropic spaces
are concerned one may think about diverse types of Morrey smoothness spaces which are in common use. This applies especially to the
spaces
\begin{\eq}   \label{2.82}
\Lambda^{\vr} \As(\rn) \quad \text{and} \quad \Lambda_{\vr} \As (\rn) \qquad \text{with} \quad A\in \{B,F \}
\end{\eq}
and
\begin{\eq}    \label{2.83}
s\in \real, \quad 0<p<\infty, \quad 0<q \le \infty, \qquad -n <\vr <0,
\end{\eq}
using the notation as introduced in \cite{HaT22}. The best possible extension of Theorem \ref{T2.4} from $\As (\rn)$ to the spaces in
\eqref{2.82} is the representation
\begin{\eq}   \label{2.84}
f = \sum_{\substack{\beta \in \nat^n_0, j\in \no, \\ m \in \zn}} \lambda^\beta_{j,m} (f) \, k^\beta_{j,m}, 
\end{\eq}
where $\lambda = \{\lambda^\beta_{j,m} \}$ belongs to suitable sequence spaces,
\begin{\eq}   \label{2.85}
s > \sigma^{|\vr|}_p = |\vr| \Big( \max \big( \frac{1}{p}, 1 \big) -1 \Big) \qquad \text{for $B$--spaces} 
\end{\eq}
(bold) and
\begin{\eq}   \label{2.86}
s > \sigma^{|\vr|}_{p,q} = |\vr| \Big( \max \big( \frac{1}{p}, \frac{1}{q}, 1 \big) - 1 \Big) \quad \text{for $F$--spaces}
\end{\eq}
(very bold). The restriction for $s$ in \eqref{2.27} comes from the needed moment conditions in the atomic representation Theorem
\ref{T1.7}. As far as we know there is no related counterpart for the spaces in \eqref{2.82}, \eqref{2.83} with $\sigma^{|\vr|}_p$
in place of $\sigma^n_p$ and $\sigma^{|\vr|}_{p,q}$ in place of $\sigma^n_{p,q}$. We proved in \cite[Theorem 3.33, pp.\,67--68]{T14}
and \cite[Theorem 1.37, pp.\,28--32]{T13} a weaker version for the spaces $\Lambda^{\vr} \As (\rn)$ by reduction to Theorem \ref{T1.7}
with  the same restriction as in \eqref{1.31} and \eqref{1.34}. In addition one has to check whether the other ingredients in the proof
of Theorem \ref{T2.4} have related counterparts for the spaces in \eqref{2.82}, \eqref{2.83}. Similarly one may ask for a counterpart
of Theorem \ref{T2.10} for the spaces in \eqref{2.82}, \eqref{2.83} with $s<0$. As will be proved in Section \ref{S2.6} any $f\in
S'(\rn)$ can be expanded in $S'(\rn)$ according to \eqref{2.66}. This applies in particular to the spaces in \eqref{2.82}, 
\eqref{2.83}. In other words, one asks for a counterpart of \eqref{2.67}. According to \cite[(1.71), p.\,13]{T20} (and the references
given there) one has
\begin{\eq}   \label{2.87}
\Lambda^0 \Fs (\rn) = F^s_{\infty,q} (\rn), \qquad 0<q < \infty, \quad s\in \real,
\end{\eq}
$0<p<\infty$, for the spaces introduced in Definition \ref{D1.1}(iii). This shows that an extension of the above representation
theory for the spaces $\As (\rn)$ to the spaces in \eqref{2.82}, \eqref{2.83} covers also the fashionable spaces $F^s_{\infty,q} (\rn),
$ $0<q< \infty$, including the spaces $\mathrm{bmo}^s (\rn) = F^s_{\infty,2} (\rn)$ of bounded mean oscillation. Furthermorre as already mentioned in Section \ref{S2.1} 
one may assume that all functions in the systems \eqref{2.13}, \eqref{2.14},
\begin{\eq}   \label{2.87a}
\big\{ k^\beta_{j,m}, \ \Phi^\beta_{j,m}: \ \beta \in \nat^n_0, \ j\in \no, \ m\in \zn \big\}
\end{\eq}
are created from underlying basic functions having
a product structure with respect to $x= (x_1, \ldots, x_n )$. This paves the way for a corresponding theory for the spaces
$S^r_{p,q} A (\rn)$ with dominating mixed smoothness as treated in \cite{ST87} and \cite{T19}. 
\end{problem}

\subsection{Weighted spaces}    \label{S2.4}
As indicated in Section \ref{S1.1}
at the end of the Introduction we wish to extend the theory developed in the preceding sections to  universal quarkonial
expansions in $S(\rn)$ and $S'(\rn)$. We reduce this task to corresponding assertions for some weighted spaces.

Let $0<w(x) \in C^\infty (\rn)$ be a weight function such that
\begin{\eq}   \label{2.88}
|D^\gamma w(x) | \le c_\gamma w(x), \qquad x\in \rn, \quad \gamma \in \nat^n_0,
\end{\eq}
and
\begin{\eq}    \label{2.89}
w(x) \le c \, w(y) \, \big( 1 + |x-y|^2 \big)^{\alpha/2}, \qquad x\in \rn, \quad y\in \rn,
\end{\eq}
for some $c_\gamma >0$, $c>0$ and $\alpha \ge 0$. Let $L_p (\rn, w)$, $0<p \le \infty$, be the weighted $L_p$--space in $\rn$ 
quasi--normed by
\begin{\eq}   \label{2.90}
\| f \, | L_p (\rn, w) \| = \| wf \, | L_p (\rn) \|.
\end{\eq}
Let $0<p,q \le \infty$ ($p<\infty$ for $F$--spaces), $s\in \real$ and let $\vp = \{ \vp_j \}^\infty_{j=0}$ be the usual dyadic resolution of unity in $\rn$ as introduced in\eqref{1.4}--\eqref{1.6}. We replace $L_p (\rn)$ in the parts (i) and (ii) of Definition 
\ref{D1.1} by $L_p (\rn,w)$. Then $\Bs (\rn,w)$ is the collection of all $f\in S'(\rn)$ such that
\begin{\eq}   \label{2.91}
\| f \, | \Bs (\rn,w) \|_{\vp} = \Big( \sum^\infty_{j=0} 2^{jsq} \| (\vp_j \wh{f} )^\vee | L_p (\rn, w) \|^q \Big)^{1/q}
\end{\eq}
is finite and $\Fs (\rn, w)$ is the collection of all $f\in S'(\rn)$ such that
\begin{\eq}  \label{2.92}
\| f \, | \Fs (\rn, w) \|_{\vp} = \Big\| \Big( \sum^\infty_{j=0} 2^{jsq} \big| \big( \vp_j \wh{f} \big)^\vee (\cdot) \big|^q \Big)^{1/q} | L_p (\rn, w) \Big\|
\end{\eq}
is finite (with the usual modifications if $q=\infty$).
It comes out that the spaces $\As (\rn,w)$, $A\in \{B,F \}$, are independent of $\vp = \{ \vp_j \}^\infty_{j=0}$ (equivalent quasi--norms).
Furthermore, $f \mapsto wf$ is an isomorphic map from $\Bs (\rn,w)$ onto $\Bs (\rn)$ and from $\Fs (\rn,w)$ onto $\Fs (\rn)$,
\begin{\eq}   \label{2.93}
\| wf \, | \As (\rn) \| \sim \| f\, | \As (\rn, w) \|.
\end{\eq}
Of special interest are the weights
\begin{\eq}    \label{2.94}
w_\delta (x) = (1+ |x|^2)^{\delta/2}, \qquad \delta \in \real, \quad x\in \rn.
\end{\eq}
In particular, for fixed $A \in \{B,F \}$ and $0<p,q \le \infty$ ($p<\infty$ for $F$--spaces) the corresponding spaces $\As (\rn, 
w_\delta)$ are monotone with respect to $s\in \real$ and $\delta \in \real$. Details, explanations and references to the substantial
history of these spaces may be found in \cite[Chapter 6]{T06}. We ask for weighted counterparts of the Theorems \ref{T2.4} and 
\ref{T2.10}. Let $f \in
\As (\rn,w)$ with $p,q,s$ as in Theorem \ref{T2.4} and $w(x) \ge c >0$ (this applies in particular to $w_\delta$ with $\delta \ge 0
$ which is the  case of interest for us later on).  Then $f$
can be represented as in Theorem \ref{T2.4}. This applies also to $wf \in \As (\rn)$ with the optimal representation
\begin{\eq}   \label{2.95}
wf = \sum_{\substack{\beta \in \nat^n_0, j\in \no, \\ m \in \zn}} \lambda^\beta_{j,m}(f,w) \, k^\beta_{j,m}
\end{\eq}
where
\begin{\eq}   \label{2.96}
\lambda^\beta_{j,m} (f,w) = 2^{jn} \big(f, w \Phi^\beta_{j,m} \big), \qquad f \in S'(\rn),
\end{\eq}
and
\begin{\eq}   \label{2.97}
\| f \, | \As (\rn, w) \| \sim \| \lambda (f,w) \, | a^s_{p,q} (\rn)^{\vk} \|.
\end{\eq}
We used representations  of this type occasionally, for example in \cite[Chapters 6--8]{T01} in connection with function spaces on domains
and manifolds, but also for some representations on \rn. But the outcome is not really satisfactory for representations in \rn. One
would like to stick at the same representations \eqref{2.28}, \eqref{2.30}
as in the unweighted case and to shift the weight to the related sequence spaces as it
has already been done in connection with wavelet decompositions for the spaces $\As (\rn,w)$ in \cite[Section 6.2]{T06}. At least
some (but not all) of the arguments used there can be transferred to the above quarkonial expansions. For this purpose we first 
introduce the weighted counterpart of Definition \ref{D2.2} in the same way as the weighted modification according to
\cite[Definition 6.11, p.\,269]{T06} of Definition \ref{D1.10} in connection with wavelet expansions of the spaces $\As (\rn,w)$.

\begin{definition}    \label{D2.16}
Let $w$ be the above weight function. Let $s\in \real$ and $0< p,q \le \infty$ $(p<\infty$ for $f$--spaces$)$. Let $\vk \in \real$.
Then $b^s_{p,q} (\rn, w)^{\vk}$ is the collection of all sequences
\begin{\eq}   \label{2.98}
\lambda = \{ \lambda^\beta: \ \beta \in \nat^n_0 \}, \quad \lambda^\beta = \{ \lambda^\beta_{j,m} \in \comp: \ j\in \no, \ m \in \zn \}
\end{\eq}
such that
\begin{\eq}   \label{2.99}
\| \lambda \, | b^s_{p,q} (\rn,w)^{\vk} \| = \sup_{\beta \in \nat^n_0} 2^{\vk |\beta|}
\bigg( \sum^\infty_{j=0} 2^{j(s- \frac{n}{p})q} \Big( \sum_{m\in \zn} w(2^{-j} m)^p |\lambda^\beta_{j,m} |^p \Big)^{q/p} \bigg)^{1/q}
\end{\eq}
is finite and $f^s_{p,q} (\rn, w)^{\vk}$ is the collection of all sequences according to \eqref{2.98} such that 
\begin{\eq}   \label{2.100}
\begin{aligned}
&\| \lambda \, | f^s_{p,q} (\rn, w)^{\vk} \| \\
& \quad = \sup_{\beta \in \nat^n_0} 2^{\vk |\beta|} \bigg\| \Big( \sum^\infty_{j=0} \sum_{m \in \zn} 2^{jsq} w(2^{-j}m
)^q \big|  \lambda^\beta_{j,m} \, \chi_{j,m} (\cdot) \big|^q \Big)^{1/q} \, | L_p (\rn) \bigg\|
\end{aligned}
\end{\eq}
is finite $($with the usual modifications if $\max(p,q) =\infty.)$
\end{definition}

\begin{remark}   \label{R2.17}
This is the weighted generalization of Definition \ref{D2.2}. Quite obviously, $a^s_{p,q} (\rn,w)^{\vk}$ with $a \in \{b,f \}$ are
quasi--Banach spaces (Banach spaces if $p \ge 1$, $q \ge 1$).
\end{remark}

Now one can 
extend Theorem \ref{T2.4} to its weighted generalization for the same parameters as there. As mentioned above the representability 
\eqref{2.32} is ensured so far if, in addition, $w(x) \ge c >0$. But in the course of the arguments below one obtains in Proposition
\ref{P2.19} by duality an extension of Proposition \ref{P2.9} to $\Cc^s (\rn,w) = B^s_{\infty,\infty} (\rn,w)$ for $s<0$ and all
admissible weights $w$. Then \eqref{2.32} for $f\in \As (\rn, w)$, where $w$ is an arbitrary admitted weight, and with $s$ as in
\eqref{2.27} can be justified by a second duality argument. We return to this point at the end of this section
in Remark \ref{R2.24} below and take temporarily
the representability \eqref{2.32} for granted for all respective spaces. Let $\sigma^{n}_p$ and $\sigma^{n}_{p,q}$ be as in 
\eqref{2.26} and let $k^\beta_{j,m}$ and $\Phi^\beta_{j,m}$ be the quarkonial building blocks as introduced in \eqref{2.4} and \eqref{2.11}. Let $\ve >0$ be as in \eqref{2.15} and Proposition \ref{P2.1}.

\begin{theorem}   \label{T2.18}
Let $0<w(x) \in C^\infty (\rn)$ be a weight function satisfying \eqref{2.88}, \eqref{2.89}. Let $0<p,q \le \infty$ $(p<\infty$ for $F$--spaces$)$,
\begin{\eq}   \label{2.101}
s > \sigma^{n}_p \ \text{for $B$--spaces}, \qquad s> \sigma^{n}_{p,q} \ \text{for $F$--spaces},
\end{\eq}
and $\vk > -\ve$.
Then $f\in S'(\rn)$ belongs to $\As (\rn,w)$ if, and only if, it can be represented as
\begin{\eq}   \label{2.102}
f = \sum_{\substack{\beta \in \nat^n_0, j\in \no, \\ m \in \zn}} \lambda^\beta_{j,m} \, k^\beta_{j,m}, \qquad \lambda \in a^s_{p,q}
(\rn,w)^{\vk},
\end{\eq}
unconditional convergence being in $S'(\rn)$. Furthermore,
\begin{\eq}   \label{2.103}
\| f \, | \As (\rn,w) \| \sim \inf \| \lambda \, | a^s_{p,q} (\rn, w)^{\vk} \|
\end{\eq}
where the infimum is taken over all admissible representations \eqref{2.102}. Let
\begin{\eq}   \label{2.104}
\lambda^\beta_{j,m} (f) = 2^{jn} \, \big(f, \Phi^\beta_{j,m} \big), \qquad f \in S'(\rn),
\end{\eq}
and
\begin{\eq}   \label{2.105}
\lambda (f) = \{ \lambda^\beta (f): \ \beta \in \nat^n_0 \}, \qquad 
\lambda^\beta (f) = \big\{ \lambda^\beta_{j,m} (f): \ j\in \no, \ m\in \zn \big\}.
\end{\eq}
Then $f \in \As (\rn,w)$ can be represented as
\begin{\eq}   \label{2.106}
f = \sum_{\substack{\beta \in \nat^n_0, j\in \no, \\ m \in \zn}} \lambda^\beta_{j,m} (f)\, k^\beta_{j,m}, 
\end{\eq}
unconditional convergence being in $S'(\rn)$ and
\begin{\eq}   \label{2.107}
\| f \, | \As (\rn,w) \| \sim \|\lambda(f) \, | a^s_{p,q} (\rn, w)^{\vk} \|
\end{\eq}
$($equivalent quasi--norms$)$.
\end{theorem}

\begin{proof}
{\em Step 1.} Let $f \in S'(\rn)$ be represented according to \eqref{2.102}. Similarly as in \eqref{2.34},
\begin{\eq}   \label{2.108}
\Big\{ 2^{-\vr |\beta|} \, k^\beta_{j,m} (x) \, \frac{w(x)}{w(2^{-j} m)}: \ j\in \no, \ m\in \zn \Big\}, \qquad \beta \in \nat^n_0,
\end{\eq}
with $-\vr <\ve$ are admitted atoms, uniformly in $\beta \in \nat^n_0$. Then
\begin{\eq}   \label{2.109}
wf = \sum_{\substack{\beta \in \nat^n_0, j\in \no, \\ m \in \zn}} w(2^{-j} m) \,
\lambda^\beta_{j,m} \cdot k^\beta_{j,m} (x) \frac{w(x)}{w(2^{-j} m)}
\end{\eq}
is an atomic decomposition of $wf \in S'(\rn)$. Now it follows from \eqref{2.93} and the same arguments as in Step 2 of the proof of
Theorem \ref{T2.4} that $f \in \As (\rn, w)$ and
\begin{\eq}   \label{2.110}
\| f \, | \As (\rn, w) \| \le c \, \|\lambda \, | a^s_{p,q} (\rn, w)^{\vk} \|.
\end{\eq}
{\em Step 2.} We prove \eqref{2.107} for the spaces $\Fs (\rn,w)$. This justifies also \eqref{2.103} for these spaces.
Let $\psi$ be a non--negative 
compactly supported $C^\infty$ function in $\rn$ such that
\begin{\eq}   \label{2.111}
\sum_{l \in \zn} \psi_l (x) =1, \quad x\in \rn, \qquad \text{where} \quad \psi_l (x) = \psi (x-l).
\end{\eq}
Then $f\in \Fs (\rn, w)$ if, and only if, $\psi_l f \in \Fs (\rn)$ and
\begin{\eq}   \label{2.112}
\| f \, | \Fs (\rn,w) \|^p \sim \sum_{l\in \zn} w^p (l) \, \| \psi_l f \, | \Fs (\rn) \|^p.
\end{\eq}
This is covered by \cite[Proposition 6.13, pp.\,269--270]{T06}. It is an easy extension of a corresponding assertion for the unweighted
spaces $\Fs (\rn)$ as it may be found in \cite[Theorem 2.35, Remark 2.36, pp.\,43--44]{T20} with a reference to \cite[Theorem 2.4.7, 
pp.\,124--125]{T92}. Let $f\in \Fs (\rn,w)$. Then
\begin{\eq}   \label{2.113}
\lambda^\beta_{j,m} (f) = \sum_{l \in \zn} 2^{jn} \, \big(\psi_l f, \Phi^\beta_{j,m} \big)
\end{\eq}
decomposes \eqref{2.104} into  local means for $\psi_l f \in \Fs (\rn)$ with the molecular kernels $\Phi^\beta_{j,m}$
which  can be broken into the atomic kernels in \eqref{2.40} and \eqref{2.42} with the counterpart
\begin{\eq}    \label{2.114}
\| \lambda^\beta (\psi_l f) \, | f^s_{p,q} (\rn) \| \le c \, 2^{-\vk |\beta|} \| \psi_l f \, | \Fs (\rn) \|, \qquad  l \in \zn,
\end{\eq}
of \eqref{2.43}.  The sequence spaces
$f^s_{p,q} (\rn)$ according to \eqref{1.29} can be decomposed similarly as \eqref{2.112} into an $\ell_p$--sum with respect to the
cubes $k + C (0,1)^n$, $k \in \zn$, $C>1$, related to the supports of $\psi_k$. Clipping together these ingredients first on the atomic
level, extended afterwards via \eqref{2.40}--\eqref{2.42} to the molecular kernels $\Phi^\beta_{j,m}$, where $w(k) \le c\,w(l) \langle
k-l \rangle^{\alpha}$ according to \eqref{2.89} can be incorporated by a suitable chosen $D$ in \eqref{2.42}, one obtains
\begin{\eq}   \label{2.115}
\begin{aligned}
\| \lambda (f) \, | f^s_{p,q} (\rn, w)^{\vk}\|^p & \le c \, \sum_{l\in \zn} w^p (l) \, \| \psi_l f \, |\Fs (\rn) \|^p \\
&\le c' \| wf \, | \Fs (\rn) \|^p \\
&\le c'' \, \| f \, | \Fs (\rn, w ) \|^p.
\end{aligned}
\end{\eq}
This estimate and \eqref{2.110} justify \eqref{2.103} and \eqref{2.107} for the spaces $\Fs (\rn, w)$. These arguments apply also to the spaces 
\begin{\eq}   \label{2.116}
F^s_{\infty, \infty} (\rn, w) = B^s_{\infty, \infty} (\rn, w), \qquad s>0.
\end{\eq}
{\em Step 3.} According to \cite[Theorem 2.4.7, pp.\,124--125]{T92}
there is no counterpart of \eqref{2.112} for the spaces $\Bs (\rn,w)$ with $p\not=q$. We rely on the real interpolation
\begin{\eq}   \label{2.117}
\Bs (\rn,w) = \big( B^{s_0}_{p,p} (\rn,w), B^{s_1}_{p,p} (\rn, w) \big)_{\theta,q}
\end{\eq}
where $\sigma^{n}_p <s_0 <s< s_1 <\infty$, $s= (1-\theta)s_0 + \theta s_1$ and $0<p,q \le \infty$, what follows from a
related assertion for the unweighted spaces $\Bs (\rn)$ and the isomorphic map $f \to wf$ from the weighted spaces onto the related
unweighted spaces. We apply \eqref{2.115} to the spaces $B^s_{p,p} (\rn,w) = F^s_{p,p} (\rn,w)$ with
$0<p\le \infty$, $s> \sigma^{n}_p$. As mentioned in Remark \ref{R2.5} one can replace $\ell_\infty$ with respect to $\beta \in 
\nat^n_0$ in \eqref{2.24}, \eqref{2.25} by $\ell_r$ with $r>0$. This applies also to their weighted counterparts \eqref{2.99}, 
\eqref{2.100}. In particular let
$b^s_p (\rn,w)^{\vk}$ be the modification of $b^s_{p,p} (\rn,w)^{\vk} = f^s_{p,p} (\rn, w)^{\vk}$ in \eqref{2.99}, \eqref{2.100} with an $\ell_p$--quasi--norm in place of $\sup_{\beta \in \nat^n_0}$. Then
\begin{\eq}   \label{2.118}
b^s_p (\rn,w)^{\vk}_q = \big( b^{s_0}_p (\rn,w)^{\vk}, b^{s_1}_p (\rn, w)^{\vk} \big)_{\theta,q}
\end{\eq}
is the related substitute of $b^s_{p,q} (\rn,w)^{\vk}$ in \eqref{2.99} with $\sum_{m\in \zn, \beta \in \nat^n_0}$ in place of $\sum_{m\in\zn}$ and without $\sup_{\beta \in \nat^n_0}$. As far as the interpolation \eqref{2.118} is concerned one may consult 
\cite[p.\,272]{T06} where we reduced related assertions via wavelet isomorphisms  to \eqref{2.117}. 
By Step 2 one has the linear bounded map
\begin{\eq}   \label{2.119}
f \mapsto \big\{ \lambda^\beta_{j,m} (f): \ \beta \in \nat^n_0, \ j \in \no, \ m\in \zn \big\}
\end{\eq}
from $B^{s_l}_{p,p} (\rn, w)$ into $b^{s_l}_p (\rn, w)^{\vk}$ where $l = 0,1$. Then it follows from \eqref{2.117}, \eqref{2.118} and the indicated modification that
\begin{\eq}   \label{2.120}
\| \lambda (f) \, | b^s_{p,q} (\rn, w)^{\vk} \| \le c \, \|f \, | \Bs (\rn, w) \|.
\end{\eq}
Now \eqref{2.110} on the one hand and \eqref{2.115}, \eqref{2.120} on the other hand prove the theorem.
\end{proof}

Next we ask for the weighted counterpart of Theorem \ref{T2.10}. We need two preparations. Let $w(x)$ be an admissible weight 
function satisfying
\eqref{2.88}, \eqref{2.89}. Then $w^{-1} (x)$ is also an admissible weight function and one can define the related spaces $\As (\rn,w)$ and
$\As (\rn, w^{-1})$ quasi--normed according to \eqref{2.91}, \eqref{2.92} (and its counterpart with $w^{-1}$ in place of $w$). It
follows from the isomorphic map $f \mapsto wf$ according to \eqref{2.93} that the duality \eqref{2.53} can be extended in the 
framework of the dual pairing $\big( S(\rn), S'(\rn) \big)$ to
\begin{\eq}  \label{2.121}
\Cc^s (\rn, w) = B^s_{\infty, \infty} (\rn, w) = B^{-s}_{1,1} (\rn, w^{-1})'.
\end{\eq}
Let again $k^\beta_{j,m}$ and $\Phi^\beta_{j,m}$ be the quarkonial building blocks as introduced in \eqref{2.4} and \eqref{2.11}. Then
one can extend Proposition \ref{P2.9} as follows.

\begin{proposition}   \label{P2.19}
Let $w$ be an admissible weight function satisfying \eqref{2.88}, \eqref{2.89} and let $s\in \real$.
Then any $f \in \Cc^s (\rn, w)$ can be represented as
\begin{\eq}   \label{2.122}
f = \sum_{\substack{\beta \in \nat^n_0, j\in \no, \\ m \in \zn}} 2^{jn}\,\big( f, k^\beta_{j,m} \big) \Phi^\beta_{j,m},
\end{\eq}
unconditional convergence being in $S'(\rn)$.
\end{proposition}

\begin{proof} According to Theorem \ref{T2.18} and the preceding comments one can expand $g\in B^{-s}_{1,1} (\rn, w^{-1})$ with $-s >0$
and $w^{-1} (x) \ge c >0$ by the unconditionally converging series
\begin{\eq}   \label{2.123}
g = \sum_{\substack{\beta \in \nat^n_0, j\in \no, \\ m \in \zn}} 2^{jn}\,\big( g, \Phi^\beta_{j,m} \big) k^\beta_{j,m}.
\end{\eq}
Then \eqref{2.122} follows from \eqref{2.58} and \eqref{2.121} for $f\in \Cc^s (\rn,w)$ with $s<0$ and $w(x) \le c$ in the same way as
in the proof of Proposition \ref{P2.9}. In particular, \eqref{2.122} converges in the weak*--topology of $\Cc^s (\rn,w)$ and 
unconditionally in $S'(\rn)$. This applies in particular to $w_\delta (x)$ according to \eqref{2.94} with $\delta \le 0$ and by
monotonicity to all $w_\delta (x)$ with $\delta \in \real$. If $w$ is an arbitrary admissible weight function
then one has by \eqref{2.89} that
\begin{\eq}    \label{2.124} 
w(x) \le c \, w_\alpha (x) \qquad \text{and} \qquad w_{- \alpha} (x) \le c \, w^{-1} (x).
\end{\eq}
This shows, again by monotonicity, that \eqref{2.122} can be extended to any $f\in \Cc^s (\rn, w)$ with $s<0$, converging 
unconditionally at least in $S'(\rn)$. The extension to $s\in \real$ is again a matter of monotonicity with respect to $s$ (and fixed
$w$).
\end{proof}

\begin{remark}   \label{R2.20}
The series \eqref{2.123} converges unconditionally in $B^{-s}_{1,1} (\rn, w^{-1})$
for $-s >0$, and any $w^{-1}$. This follows from
Remark \ref{R2.24} below and the respective assertions in Theorem \ref{T2.18} based on the preceding comments. Then \eqref{2.121}
shows that \eqref{2.122} converges  in the weak*--topology if $s<0$ complementing Proposition \ref{P2.9}. But this cannot be extended
to $s>0$. We dealt with related problems in Corollary \ref{C2.13} and Remark \ref{R2.14}.
\end{remark}

We need a second preparation. The atomic representation Theorem \ref{T1.7} remains valid if one replaces the moment conditions \eqref{1.23} by
\begin{\eq}   \label{2.125}
\Big| \int_{\rn} \Psi (x) \, a_{j,m} (x) \, \di x \Big| \le c\,2^{-j(L+n)} \| \Psi \, | C^L (\rn) \|
\end{\eq}
for some $c>0$ and all functions $\Psi \in C^L (\rn)$, where $C^L (\rn)$ is normed by
\begin{\eq}   \label{2.126}
\| \Psi \, | C^L (\rn) \| = \sup_{x \in \rn, |\alpha| \le L} |D^\alpha \Psi (x) |
\end{\eq}
with the same $L\in \nat$ as in \eqref{1.31} for the $B$--spaces and as in \eqref{1.34} for the $F$--spaces.
This modification has been observed in \cite{Skr98} dealing with spaces of type $\As$ on smooth manifolds. A related
formulation may also be found in \cite[p.\,134]{T08}. This modification is sufficient for our purpose in this Section \ref{S2.4}. 
But later on when dealing is Section \ref{S3.1} with pointwise multipliers for the spaces
\begin{\eq}   \label{2.127}
\As (\rn) \qquad \text{with $A\in \{B,F \}$, \ $s<0$ and $0<p,q \le \infty$},
\end{\eq}
($p<\infty$ for $F$--spaces) we rely on the refinement of \eqref{2.125}
\begin{\eq}   \label{2.128}
\Big| \int_{\rn} \Psi (x) \, a_{j,m} (x) \, \di x \Big| \le C\, 2^{-j(\vr+n)} \| \Psi \, | \Cc^{\vr} (\rn) \|, \quad j\in \no, \quad m\in \zn,
\end{\eq}
for all $\Psi \in \Cc^{\vr} (\rn) = B^{\vr}_{\infty, \infty} (\rn)$ replacing $L$ in \eqref{1.31} and \eqref{1.34}
by \vr. This observation goes back
to \cite{Scha13}. We fix the outcome restricting us to the spaces $\As (\rn)$ in \eqref{2.127} (in particular $s<0$). Let $\vr >0$,
$d>1$ and $C>0$. Then
\begin{\eq}   \label{2.129}
a_{j,m} \in L_\infty (\rn) \quad \text{with} \quad \supp a_{j,m} \subset d \, Q_{j,m}, \quad j\in \no, \quad m\in \zn,
\end{\eq}
as in \eqref{1.22} and \eqref{2.128} for all $\Psi \in \Cc^{\vr} (\rn)$ is called an $\vr$--atom (more precisely an 
$(\vr,C,d)$--atom). Let $a^s_{p,q} (\rn)$ be as in Definition \ref{D1.5} and Remark \ref{R1.6}. Let $\sigma^n_p$ and $\sigma^n_{p,q}$
be as in \eqref{1.30}.

\begin{proposition}   \label{P2.21}
Let $\As (\rn)$ be as in \eqref{2.127} with $p<\infty$ for the $F$--spaces. Let
\begin{\eq}   \label{2.130}
\vr >
\begin{cases}
\sigma^{n}_p -s &\text{if $A=B$}, \\
\sigma^{n}_{p,q} -s &\text{if $A=F$}.
\end{cases}
\end{\eq}
Let $f\in S'(\rn)$. Then $f\in \As (\rn)$ if, and only if, it can be represented as
\begin{\eq}   \label{2.131}
f = \sum^\infty_{j=0} \sum_{m \in \zn} \mu_{j,m} \, a_{j,m}, \qquad \mu = \{ \mu_{j,m} \} \in a^s_{p,q} (\rn)
\end{\eq}
where $a_{j,m}$ are $\vr$--atoms $($more precisely $(\vr, C,d)$--atoms for a fixed $d>1$ and $C>0)$,
unconditional convergence being in $S'(\rn)$. Furthermore,
\begin{\eq}  \label{2.132}
\| f \, | \As (\rn) \| \sim \inf \| \mu \, | a^s_{p,q} (\rn) \|
\end{\eq}
where the infimum is taken over all admissible representations \eqref{2.131} $($equivalent quasi--norms$)$.
\end{proposition}

\begin{remark}   \label{R2.22}
As already said the above proposition modifies the classical atomic representation Theorem \ref{T1.7} by the related observations in \cite{Skr98} and \cite{Scha13}. An extension to molecular
decompositions would be desirable and very useful for what follows, but not (yet) available. For this purpose one has to insert
molecular expansions in corresponding characterizations of $\As (\rn)$ in terms of local means extending the arguments in \cite{Scha13}
from atoms to molecules.
\end{remark}

After these preparations we can now extend Theorem \ref{T2.10} to the related spaces $\As (\rn,w)$, quasi--normed according to 
\eqref{2.91}, \eqref{2.92} and their sequence counterparts according  to Definition \ref{D2.16}. Let
$k^\beta_{j,m}$ and $\Phi^\beta_{j,m}$ be the quarkonial building blocks as introduced in \eqref{2.4} and \eqref{2.11}. Let $\ve >0$
be as in \eqref{2.15} and Proposition \ref{P2.1}.

\begin{theorem}   \label{T2.23}
Let $0<w(x) \in C^\infty (\rn)$ be a weight function satisfying \eqref{2.88}, \eqref{2.89}. Let $0<p,q \le \infty$ $(p<\infty$ for
$F$--spaces$)$, $s<0$ and $\vk <\ve$.
Then $f\in S'(\rn)$ belongs to $\As (\rn,w)$ if, and only if, it can be
represented as
\begin{\eq}   \label{2.133}
f = \sum_{\substack{\beta \in \nat^n_0, j\in \no, \\ m \in \zn}} \lambda^\beta_{j,m} \, \Phi^\beta_{j,m}, \qquad \lambda \in a^s_{p,q} (\rn,w)^{\vk},
\end{\eq}
unconditional convergence being in $S'(\rn)$. Furthermore,
\begin{\eq}   \label{2.134}
\| f \, | \As (\rn,w) \| \sim \inf \| \lambda \, | a^s_{p,q} (\rn,w)^{\vk} \|
\end{\eq}
where the infimum is taken over all admissible representations \eqref{2.133}. Let
\begin{\eq}   \label{2.135}
\lambda^\beta_{j,m} (f) = 2^{jn} \, \big(f, k^\beta_{j,m} \big), \qquad f \in S' (\rn),
\end{\eq}
and
\begin{\eq}   \label{2.136}
\lambda (f) = \{ \lambda^\beta (f): \ \beta \in \nat^n_0 \}, \quad \lambda^\beta (f) = \{ \lambda^\beta_{j,m} (f): \ j \in \no, \ m\in
\zn \}.
\end{\eq}
Then $f \in \As (\rn,w)$ can be represented as
\begin{\eq}   \label{2.137}
f = \sum_{\substack{\beta \in \nat^n_0, j\in \no, \\ m \in \zn}} \lambda^\beta_{j,m}(f) \, \Phi^\beta_{j,m}, 
\end{\eq}
unconditional convergence being in $S'(\rn)$, and
\begin{\eq}   \label{2.138}
\| f \, | \As (\rn,w) \| \sim \| \lambda (f) \, | a^s_{p,q} (\rn,w)^{\vk} \|
\end{\eq}
$($equivalent quasi--norms$)$.
\end{theorem}

\begin{proof}
{\em Step 1.} Let $f$ be given by \eqref{2.133}. Then
\begin{\eq}   \label{2.139}
wf = \sum_{\substack{\beta \in \nat^n_0, j\in \no, \\ m \in \zn}} w(2^{-j} m) \,
\lambda^\beta_{j,m} \cdot \Phi^\beta_{j,m} \frac{w(x)}{w(2^{-j} m)}
\end{\eq}
is the counterpart of \eqref{2.109}. We rely again on the decomposition of the classical molecules $\Phi^\beta_{j,m}$ into classical
atoms having moment conditions up to order $2L-1$ as described in \eqref{2.40}--\eqref{2.42}. We wish to apply Proposition \ref{P2.21}
based on
\begin{\eq}   \label{2.140}
\frac{w(x)}{w(2^{-j}m)} \, \Phi^\beta_{j,m} (x) = \sum_{k \in \zn} \frac{w(x)}{w(2^{-j}m)} \, a^{L,\beta,k}_{j,m} (x)
\end{\eq}
where
\begin{\eq}   \label{2.141}
\langle k \rangle^D 2^{\vk |\beta|} \, a^{L,\beta,k}_{j,m} (x) \, \frac{w(x)}{w(2^{-j}m)} \qquad \text{are $L_\infty$--normalized}
\end{\eq}
building blocks for any $D>0$ and any $\vk >0$. Here $a^{L,\beta,k}_{j,m} (x)$ are the same atoms as in \eqref{2.41}. They are 
localized at $2^{-j} (m+k)$. Then it follows from \eqref{2.89} that $w(x) \le c \, w(2^{-j} m) \langle k \rangle^\alpha$ in 
\eqref{2.141}. This show that the related additional  factor in \eqref{2.141} is well compensated by \eqref{2.42}. Otherwise one has to
justify \eqref{2.125} which means now
\begin{\eq}   \label{2.142}
\Big| \int_{\rn} \Psi (x) \, \frac{w(x)}{w(2^{-j}m)} \, a^{L,\beta,k}_{j,m} (x) \, \di x \Big| \le c\,2^{-j(L+n)} \langle k \rangle^{-D} 2^{-\vk |\beta|} \| \Psi \, | C^L (\rn) \|.
\end{\eq}
But this follows from \eqref{2.41}, \eqref{2.42}
 on the one hand and the Taylor expansion of $\Psi (x) \frac{w(x)}{w(2^{-j}m)}$ at $2^{-j} (m+k)$ with remainder term of order $L$, using in addition \eqref{2.88}, on the other hand. We insert \eqref{2.140} in \eqref{2.139}. For
fixed $\beta \in \nat^n_0$ and $k \in \zn$ one can apply Proposition \ref{P2.21} (with $\vr =L$).
Then it follows from \eqref{2.93} and a related
counterpart of \eqref{2.35} that $f\in \As (\rn,w)$ and
\begin{\eq}   \label{2.143}
\| f \, | \As (\rn, w) \| \le c \, \| \lambda \, | a^s_{p,q} (\rn, w)^{\vk} \|.
\end{\eq}
{\em Step 2.} By Proposition \ref{P2.19}, well--known embeddings for the spaces $\As (\rn)$ and \eqref{2.93} it follows that any
$f\in \As (\rn,w)$ with $s\in \real$ and $0<p,q \le \infty$ ($p<\infty$ for $F$--spaces) can be represented as
\begin{\eq}   \label{2.144}
f = \sum_{\substack{\beta \in \nat^n_0, j\in \no, \\ m \in \zn}} 2^{jn}\,\big( f, k^\beta_{j,m} \big) \Phi^\beta_{j,m}.
\end{\eq}
Similarly as in Step 2 of the proof of Theorem \ref{T2.18} we deal first with the $F$--spaces where now $s<0$, $0<p<\infty$ and $0<q
\le \infty$ (and $s<0$, $p=q=\infty$). Then
\begin{\eq}   \label{2.145}
\lambda^\beta_{j,m} (f) = \sum_{l\in \zn} 2^{jn} \, \big( \psi_l f, k^\beta_{j,m} \big)
\end{\eq}
is the counterpart of \eqref{2.113}. Now it follows by the same arguments as in Step 2 of the proof of Theorem \ref{T2.10} combined
with \eqref{2.115} that
\begin{\eq}  \label{2.146}
\| \lambda (f) \, | f^s_{p,q} (\rn, w)^{\vk} \| \le c \, \| f \, | \Fs (\rn, w) \|.
\end{\eq}
This is the point where $\vk < \ve$ is needed. Then \eqref{2.143} and \eqref{2.146} prove both \eqref{2.134} and \eqref{2.138} for the
$F$--spaces. The extension of this assertion to the $B$--spaces can now be done by interpolation as in Step 3 of the proof of Theorem
\ref{T2.18}.
\end{proof} 

\begin{remark}   \label{R2.24}
We return to the representability 
\begin{\eq}   \label{2.147}
f = \sum_{\substack{\beta \in \nat^n_0, j\in \no, \\ m \in \zn}} \lambda^\beta_{j,m} (f)\, k^\beta_{j,m}, 
\end{\eq}
for the spaces covered by Theorem \ref{T2.18}. Formally this is justified so far only for these spaces if $w(x) \ge c >0$. But this is
sufficient to ensure by duality and embedding the representability  \eqref{2.122} and \eqref{2.137} for all admissible weights $w$
satisfying \eqref{2.88}, \eqref{2.89}. Now we reverse the duality argument. According to \cite[Theorem 2.11.2, p.\,178]{T83} one has
\begin{\eq}   \label{2.148}
B^s_{p,q} (\rn)' = B^{-s}_{p',q'} (\rn), \qquad s\in \real \quad \text{and} \quad 1\le p,q <\infty,
\end{\eq}
where as usual $\frac{1}{p} + \frac{1}{p'} = \frac{1}{q} + \frac{1}{q'} =1$. By the arguments preceding Proposition \ref{P2.19} one
can extend \eqref{2.148} in the framework of the dual pairing $\big( S(\rn), S'(\rn) \big)$ to
\begin{\eq}   \label{2.149}
\Bs (\rn,w)' = B^{-s}_{p',q'} (\rn, w^{-1}), \qquad s\in \real \quad \text{and} \quad 1\le p,q< \infty
\end{\eq}
for any admissible weight function $w$. Let $s<0$. 
Then it follows from Theorem \ref{T2.23} that any $g\in \Bs (\rn, w)$ can be represented as
\begin{\eq}   \label{2.150}
g = \sum_{\substack{\beta \in \nat^n_0, j\in \no, \\ m \in \zn}} 2^{jn}\,\big( g, k^\beta_{j,m} \big) \Phi^\beta_{j,m},
\end{\eq}
converging unconditionally in $\Bs (\rn,w)$. Now one can argue as in \eqref{2.58} where $k^\beta_{j,m}$ and $\Phi^\beta_{j,m}$ change
their roles. This shows that any $f\in B^{-s}_{p',q'} (\rn, w^{-1})$ with $-s >0$ and $1< p', q' \le \infty$ can be represented in the
weak*--topology of $B^{-s}_{p',q'} (\rn, w^{-1})$ as the unconditionally converging series 
\begin{\eq}   \label{2.151}
f = \sum_{\substack{\beta \in \nat^n_0, j\in \no, \\ m \in \zn}} 2^{jn}\,\big( f, \Phi^\beta_{j,m} \big) k^\beta_{j,m}.
\end{\eq}
This ensures by embedding the unconditional expansion \eqref{2.104}--\eqref{2.106} in $S'(\rn)$ for all spaces $\As (\rn,w)$ covered
by Theorem \ref{T2.18} and all admissible weight functions $w$.
\end{remark}

\subsection{Weighted H\"{o}lder--Zygmund spaces}   \label{S2.5}                   
We wish to use the above quarkonial
representations in the spaces $\As (\rn)$ and their weighted generalizations $\As (\rn, w)$ to say something
similar for $S(\rn)$ and, in particular, $S'(\rn)$. First, but not really satisfactory, attempts may be found in \cite[Section 8.7,
pp.\,108--111]{T01}. We try to simplify related assertions as much as possible. This suggests to concentrate on the weighted 
H\"{o}lder--Zygmund spaces
\begin{\eq}   \label{2.152}
\Cc^s (\rn, \delta) = B^s_{\infty, \infty} (\rn, w_\delta), \qquad s\in \real, \quad \delta \in \real,
\end{\eq}
with $w_\delta (x) = (1 + |x|^2)^{\delta/2}$ as in \eqref{2.94}. In specification of \eqref{2.91} the space $\Cc^s 
(\rn, \delta)$ is the collection of all $f\in S'(\rn)$ such that
\begin{\eq}  \label{2.153}
\| f \, | \Cc^s (\rn, \delta) \|_{\vp} = \sup_{j\in \no, x\in \rn} (1 + |x|)^\delta \, 2^{js} \, \big| (\vp_j \wh{f} \, )^\vee (x)
\big|
\end{\eq}
is finite (equivalent norms). If $\delta =0$ then $\Cc^s (\rn,0) = \Cc^s (\rn)$ are the classical H\"{o}lder--Zygmund spaces according
to \eqref{1.14}. One has by \eqref{2.93}
\begin{\eq}   \label{2.154}
\|f\, | \Cc^s(\rn, \delta) \| \sim \| w_\delta f \, | \Cc^s (\rn) \| \sim 
\sup_{j \in \no, x\in \rn} 2^{js} \big| (\vp_j \wh{w_\delta f})^\vee (x) \big|.
\end{\eq}
Let
\begin{\eq}   \label{2.155}
\text{\ccd} = b^s_{\infty,\infty} (\rn, w_\delta)^0, \qquad s\in \real, \quad \delta \in \real,
\end{\eq}
be the specification of \eqref{2.99}, extending \eqref{2.46}, \eqref{2.47} to its weighted generalization. It is the collection of all
sequences
\begin{\eq}  \label{2.156}
\lambda = \big\{ \lambda^\beta_{j,m} \in \comp: \ \beta \in \nat^n_0, \ j \in \no, \ m\in \zn \big\}
\end{\eq} 
such that 
\begin{\eq}   \label{2.157}
\| \lambda \, | \text{\ccd} \| = \sup_{\substack{\beta \in \nat^n_0, j\in \no, \\ m \in \zn}} 2^{js}\, (1+ 2^{-j} |m|)^\delta \,
|\lambda^\beta_{j,m}|
\end{\eq}
is finite. We specify the Theorems \ref{T2.18} and \ref{T2.23} to \eqref{2.152}. This extends also Corollary \ref{C2.6} to the related
weighted spaces and complements Proposition \ref{P2.19}. Let again $k^\beta_{j,m}$ and $\Phi^\beta_{j,m}$ be the quarkonial building
blocks as introduced in \eqref{2.4} and \eqref{2.11}.

\begin{corollary}   \label{C2.25}
Let $s>0$ and $\delta \in \real$. Then $f\in S'(\rn)$ belongs to $\Cc^s (\rn, \delta)$ if, and only if, it can be represented as
\begin{\eq}  \label{2.158}
f = \sum_{\substack{\beta \in \nat^n_0, j\in \no, \\ m \in \zn}} \lambda^\beta_{j,m} \, k^\beta_{j,m}, \qquad \lambda \in \text{\ccd},
\end{\eq}
unconditional convergence being in $S'(\rn)$. Furthermore,
\begin{\eq}   \label{2.159}
\| f \, | \Cc^s (\rn,\delta) \| \sim \inf \| \lambda \, | \text{\ccd} \|
\end{\eq}
where the infimum is taken over all admissible representations  \eqref{2.158}. Let
\begin{\eq}   \label{2.160}
\lambda^\beta_{j,m} (f) = 2^{jn} \big( f , \Phi^\beta_{j,m} \big), \qquad f \in S'(\rn),
\end{\eq}
and
\begin{\eq}   \label{2.161}
\lambda (f) = \big\{ \lambda^\beta_{j,m} (f): \ \beta \in \nat^n_0, \  j\in \no, \ m\in \zn \big\}.
\end{\eq}
Then $f\in \Cc^s (\rn,\delta)$ can be represented as
\begin{\eq}   \label{2.162}
f = \sum_{\substack{\beta \in \nat^n_0, j\in \no, \\ m \in \zn}} \lambda^\beta_{j,m} (f) \, k^\beta_{j,m}, 
\end{\eq}
unconditional convergence being in $S'(\rn)$, and
\begin{\eq}   \label{2.163}
\| f \, | \Cc^s (\rn,\delta) \| \sim \| \lambda (f) \, |  \text{\ccd} \|
\end{\eq}
$($equivalent norms$)$.
\end{corollary}

\begin{proof}
This follows from Theorem \ref{T2.18} with $\vk =0$ specified to \eqref{2.152} and \eqref{2.157}.
\end{proof}

\begin{corollary}   \label{C2.26}
Let $s<0$ and $\delta \in \real$. Then $f\in S'(\rn)$ belongs to $\Cc^s (\rn, \delta)$ if, and only if, it can be represented as
\begin{\eq}  \label{2.164}
f = \sum_{\substack{\beta \in \nat^n_0, j\in \no, \\ m \in \zn}} \lambda^\beta_{j,m} \, \Phi^\beta_{j,m}, \qquad \lambda \in \text{\ccd},
\end{\eq}
unconditional convergence being in $S'(\rn)$. Furthermore,
\begin{\eq}   \label{2.165}
\| f \, | \Cc^s (\rn,\delta) \| \sim \inf \| \lambda \, | \text{\ccd} \|
\end{\eq}
where the infimum is taken over all admissible representations  \eqref{2.164}. Let
\begin{\eq}   \label{2.166}
\lambda^\beta_{j,m} (f) = 2^{jn} \big( f , k^\beta_{j,m} \big), \qquad f \in S'(\rn),
\end{\eq}
and
\begin{\eq}   \label{2.167}
\lambda (f) = \big\{ \lambda^\beta_{j,m} (f): \ \beta \in \nat^n_0, \  j\in \no, \ m\in \zn \big\}.
\end{\eq}
Then $f\in \Cc^s (\rn,\delta)$ can be represented as
\begin{\eq}   \label{2.168}
f = \sum_{\substack{\beta \in \nat^n_0, j\in \no, \\ m \in \zn}} \lambda^\beta_{j,m} (f) \, \Phi^\beta_{j,m}, 
\end{\eq}
unconditional convergence being in $S'(\rn)$, and
\begin{\eq}   \label{2.169}
\| f \, | \Cc^s (\rn,\delta) \| \sim \| \lambda (f) \, |  \text{\ccd} \|
\end{\eq}
$($equivalent norms$)$.
\end{corollary}

\begin{proof}
This follows from Theorem \ref{T2.23} with $\vk =0$ specified to \eqref{2.152} and \eqref{2.157}.
\end{proof}

\begin{remark}    \label{R2.27}
We inserted the above corollaries mainly as a preparation of what follows. One has by \eqref{2.149} the duality
\begin{\eq}   \label{2.170}
B^s_{1,1} (\rn, \delta)' = \Cc^{-s} (\rn, \delta^{-1} ), \qquad s\in \real, \quad \delta \in \real.
\end{\eq}
Then it follows by the same arguments as in the proof of Proposition \ref{P2.9}, now based on the Theorems \ref{T2.18}, \ref{T2.23}
and the explanations in Remark \ref{R2.24} that the unconditional convergence in $S'(\rn)$ both of \eqref{2.162} and \eqref{2.168}
can be strengthened by the unconditional convergence in the weak*--topology of the respective spaces $\Cc^s (\rn, \delta)$.
\end{remark}

\begin{problem}   \label{P2.28}
What about $\Cc^0 (\rn, \delta)$ with $\delta \in \real$?
\end{problem}

\subsection{Universal representations}    \label{S2.6}
We wish to extend the quarkonial representations for the spaces $\As (\rn)$ as described in the Sections \ref{S2.2} and \ref{S2.3} to
$S(\rn)$ and, in particular, to $S'(\rn)$. For this purpose one can rely on the related expansions for the weighted spaces $\As (\rn,w)$
according to Section \ref{S2.4}. We concentrate on the simplest case. These are the Corollaries \ref{C2.25} and \ref{C2.26} dealing with
the weighted H\"{o}lder--Zygmund spaces $\Cc^s (\rn, w)$. But first we adopt a more general point of 
view collecting a few well--known basic properties.

Let $w_\delta (x) = (1+ |x|^2)^{\delta/2}$, $x\in \rn$, $\delta \in \real$, be as in \eqref{2.94}. Let
\begin{\eq}   \label{2.171}
\Bs (\rn, w_\delta), \qquad s\in \real, \quad \delta \in \real \quad \text{and} \quad 0<p,q \le \infty,
\end{\eq}
be the weighted spaces as introduced at the beginning of Section \ref{S2.4} with the remarkable property \eqref{2.93}. In particular, for fixed $p$ and $q$ with $0<p,q \le \infty$ these spaces are monotone with respect to $\delta$ and $s$. Furthermore,
\begin{\eq}   \label{2.172}
S(\rn) = \bigcap_{\delta \in \real, s \in \real} \Bs (\rn, w_\delta), \qquad S'(\rn) = \bigcup_{\delta \in \real, s \in \real} \Bs (\rn, w_\delta).
\end{\eq}
This may be found in \cite[(2.281), p.\,74]{T20} with a reference to \cite{Kab08}. But it has already been mentioned in \cite[(77),
(78), p.\,316]{Tri00} and \cite[Section 8.2, pp.\,102--104]{T01}.
Of interest for us is the case $p=q=\infty$,
\begin{\eq}   \label{2.173}
S(\rn) = \bigcap_{\delta \in \real, s\in \real} \Cc^s (\rn, \delta), \qquad S' (\rn) = \bigcup_{\delta \in \real, s\in \real} \Cc^s (\rn, \delta ),
\end{\eq}
where we used the notation \eqref{2.152}.  
At least the first assertion is rather obvious. Let $C^N (\rn, \delta )$, $N \in \no$, $\delta >0$, be the Banach spaces normed by
\begin{\eq}   \label{2.174}
\| f \, | C^N (\rn, \delta) \| = \sup_{|\beta| \le N, x\in \rn} w_\delta (x) |D^\beta f(x) | \sim \sup_{|\beta| \le N, x\in \rn}
| D^\beta (w_\delta f)(x)|.
\end{\eq}
These are the generating norms of $S(\rn)$. Then the first assertion in \eqref{2.173} follows from
\begin{\eq}   \label{2.195}
\begin{aligned}
\Cc^{N+1} (\rn, \delta) &\hra B^N_{\infty,1} (\rn, w_\delta) \\
&\hra C^N(\rn, \delta) \hra B^N_{\infty, \infty}(\rn, w_\delta) =\Cc^N (\rn, \delta),
\end{aligned}
\end{\eq}
using \eqref{2.152}, \eqref{2.93} and  \eqref{2.174} for the  $C$--spaces.

Let again  $k^\beta_{j,m}$ and $\Phi^\beta_{j,m}$ be the quarkonial building blocks as introduced in \eqref{2.4} and \eqref{2.11},
\begin{\eq}   \label{2.176}
\lambda^\beta_{j,m} (f) = 2^{jn} \, \big(f, \Phi^\beta_{j,m} \big), \qquad f \in S'(\rn),
\end{\eq}
and
\begin{\eq}   \label{2.177}
\lambda (f) = \big\{ \lambda^\beta_{j,m} (f): \ \beta \in \nat^n_0, \ j \in \no, \ m\in \zn \big\}.
\end{\eq}
Let $\text{\ccd}$ be the sequence spaces according to \eqref{2.156}, \eqref{2.157}.

\begin{theorem}   \label{T2.29}
Let $f\in S'(\rn)$. Then $f\in S(\rn)$ if, and only if, it can be represented for any $s>0$ and any $\delta >0$ as
\begin{\eq}   \label{2.178}
f = \sum_{\substack{\beta \in \nat^n_0, j\in \no, \\ m \in \zn}} \lambda^\beta_{j,m} (f) \, k^\beta_{j,m}, \qquad \lambda (f) \in
\text{\ccd}.
\end{\eq}
\end{theorem}

\begin{proof}
This follows immediately from the first assertion in \eqref{2.173}, the monotonicity of the spaces $\Cc^s (\rn, \delta)$ with respect
to $s$ and $\delta$, and Corollary \ref{C2.25}.
\end{proof}

Similarly one can employ the second assertion in \eqref{2.173}.

\begin{theorem}   \label{T2.30}
Any $f\in S'(\rn)$ can be represented as
\begin{\eq}   \label{2.179}
f = \sum_{\substack{\beta \in \nat^n_0, j\in \no, \\ m \in \zn}} 2^{jn} \, \big(f, k^\beta_{j,m} \big) \, \Phi^\beta_{j,m}, 
\end{\eq}
unconditional convergence being in $S'(\rn)$.
\end{theorem}

\begin{proof} If $f \in S'(\rn)$ then it follows from \eqref{2.173} and the monotonicity of the spaces $\Cc^s (\rn, \delta)$ with
respect to $s$ and $\delta$ that
\begin{\eq}  \label{2.180}
f \in \Cc^s (\rn, \delta) \qquad \text{for some $s<0$ and $\delta <0$.}
\end{\eq}
Then one can apply Corollary \ref{C2.26} and the comments about convergence in Remark \ref{R2.27}.
\end{proof}

\section{Applications}    \label{S3}
\subsection{Pointwise multipliers}  \label{S3.1}
One may ask for applications of the above quarkonial expansions. First we are interested to employ Theorem \ref{T2.10} dealing with the corresponding  representations
\begin{\eq}   \label{3.1}
f = \sum_{\substack{\beta \in \nat^n_0, j\in \no, \\ m \in \zn}} 2^{jn} \big(f, k^\beta_{j,m} \big) \, \Phi^\beta_{j,m}
\end{\eq}
for the spaces
\begin{\eq}   \label{3.2}
\As (\rn) \qquad \text{with $A\in \{B,F \}$, \ $s<0$ and $0<p,q \le \infty$}
\end{\eq}
($p<\infty$ for $F$--spaces). If $T$ is a linear operator then
\begin{\eq}   \label{3.3}
Tf = \sum_{\substack{\beta \in \nat^n_0, j\in \no, \\ m \in \zn}} 2^{jn} \big( f, k^\beta_{j,m} \big) T\Phi^\beta_{j,m}, 
\end{\eq}
may be considered first for finite sums, discussing convergence afterwards. This reduces mapping properties of $T$ in or between spaces
$\As (\rn)$ according to \eqref{3.2} to sufficiently strong properties of $T \Phi^\beta_{j,m}$. At the best $T \Phi^\beta_{j,m}$ are
closely related to $L_\infty$--normalized molecules as described Remark \ref{R1.9}. There one finds some comments and references. But
the request that the moment conditions \eqref{1.37} remain valid if one replaces $b_{j,m}$ by $Tb_{j,m}$ is very restrictive. 
We described in Proposition \ref{P2.21} and Remark \ref{R2.22} with a reference to
\cite{Skr98} and \cite{Scha13} how these conditions can be relaxed at least on an atomic level. This requires that one reduces
$T \Phi^\beta_{j,m}$ to related so--called $\vr$--atoms. This has already been done in Step 1 of the proof of Theorem \ref{T2.23} in
a slightly different but near--by context. But this type of arguments can also be used for other purposes, for example, pointwise
multipliers.

After this discussion we return to \eqref{3.3} where $T$ stands for the multiplication of $f$ according to \eqref{3.1} with $g \in
\Cc^{\vr} (\rn) = B^{\vr}_{\infty, \infty} (\rn)$, $\vr >0$. In other words, we ask for conditions ensuring that
\begin{\eq}   \label{3.4}
gf = \sum_{\substack{\beta \in \nat^n_0, j\in \no, \\ m \in \zn}} 2^{jn} \big(f, k^\beta_{j,m} \big) \, g \Phi^\beta_{j,m}
\end{\eq}
makes sense in a given space $\As (\rn)$ with \eqref{3.2} such that $g$ is a pointwise multiplier,
\begin{\eq}   \label{3.5}
\| gf \, | \As (\rn) \| \le c\, \| g \, | \Cc^{\vr} (\rn)\| \cdot \| f \, | \As (\rn) \|
\end{\eq}
for some $c>0$, all $g\in \Cc^{\vr} (\rn)$ and all $f\in \As (\rn)$. We rely on the expansion in Theorem \ref{T2.10} with $\vk =0$,
in particular
\begin{\eq}   \label{3.6}
f = \sum_{\substack{\beta \in \nat^n_0, j\in \no, \\ m \in \zn}} \lambda^\beta_{j,m}(f) \, \Phi^\beta_{j,m}
\end{\eq}
where
\begin{\eq}   \label{3.7}
\lambda^\beta_{j,m} (f) = 2^{jn} \, \big(f, k^\beta_{j,m} \big), \qquad f \in S' (\rn),
\end{\eq}
and
\begin{\eq}   \label{3.8}
\| f \, | \As (\rn) \| \sim \| \lambda (f) \, | a^s_{p,q} (\rn)^0 \|
\end{\eq}
with $a^s_{p,q} (\rn)^0$ as in Definition \ref{D2.2} and Remark \ref{R2.3}. 
First we say what is meant by \eqref{3.4}. Let
\begin{\eq}   \label{3.9}
(gf)_N = \sum_{|\beta| + |m| +j \le N} \lambda^\beta_{j,m}(f) \, g \,\Phi^\beta_{j,m}, \qquad N \in \nat.
\end{\eq}

\begin{definition}   \label{D3.1}
Let $\As (\rn)$ be the spaces according to \eqref{3.2} with $p<\infty$ for $F$--spaces. Let $g\in \Cc^{\vr} (\rn)$ with $\vr >0$. 
Then $g$ is said to be a pointwise multiplier for $\As (\rn)$ if $(gf)_N$ converges in the weak*--topology in $S'(\rn)$ to a limit 
denoted as $gf$,
\begin{\eq}   \label{3.10}
(gf)_N (\vp) \to (gf)(\vp) \qquad \text{for any} \quad \vp \in S(\rn),
\end{\eq}
and
\begin{\eq}   \label{3.11}
\sup_{N \in \nat}  \big\| (gf)_N \, | \As (\rn) \big\| \le C_g \, \| f \, | \As (\rn) \|
\end{\eq}
for some $C_g >0$ and all $f\in \As (\rn)$.
\end{definition}

\begin{remark}   \label{R3.2}
This is a modification of what usually is called {\em smooth pointwise multipliers} adapted to our situation and purposes. Of course, 
\eqref{3.9} makes sense pointwise. All spaces $\As (\rn)$ have the Fatou property as described in \cite[Section 1.3.4, p.\,18]{T20}, based on related references.  Then it follows from \eqref{3.10}, \eqref{3.11} that
\begin{\eq}   \label{3.12}
f \mapsto gf \quad \text{generates a linear and bounded map in $\As (\rn)$}
\end{\eq}
and
\begin{\eq}   \label{3.13}
\| gf \, | \As (\rn) \| \le c \, C_g \, \|f \, |\As(\rn) \|, \qquad f \in \As (\rn),
\end{\eq}
where $c$ is independent of $g$ (and  $f$). Otherwise we refer the reader to \cite[Section 2.4.1, pp.\,40--43]{T20}. There one finds
detailed references to the rich history of this problem and also a description of some distinguished results. This will not be repeated here. We wish
to recover some already known assertions ensuring that $g \Phi^\beta_{j,m}$ in \eqref{3.9} may serve as building blocks similar to
related molecules. Let again
\begin{\eq}   \label{3.14}
\sigma^n_p = n \Big( \max \big( \frac{1}{p}, 1 \big) -1 \Big), \qquad \sigma^n_{p,q} = n \Big( \max \big( \frac{1}{p}, \frac{1}{q}, 1
\big) - 1 \Big),
\end{\eq}
where $0<p,q \le \infty$ as in \eqref{1.30}. Recall that $\Cc^{\vr} (\rn) = B^{\vr}_{\infty, \infty} (\rn)$. 
\end{remark}

\begin{theorem}   \label{T3.3}
Let 
\begin{\eq}   \label{3.15}
\As (\rn) \quad \text{with} \quad  A \in \{B, F \}, \quad s<0 \quad \text{and} \quad 0<p,q \le \infty
\end{\eq} 
with $p<\infty$ for the $F$--spaces. Let
\begin{\eq}  \label{3.16}
\vr >
\begin{cases}
\sigma^{n}_p -s &\text{if $A=B$}, \\
\sigma^{n}_{p,q} -s &\text{if $A=F$}.
\end{cases}
\end{\eq}
Then any $g\in \Cc^{\vr} (\rn)$ is a pointwise multiplier in $\As (\rn)$ and 
\begin{\eq}   \label{3.17}
\| gf \, | \As (\rn) \| \le c\, \|g \, | \Cc^{\vr} (\rn) \| \cdot \| f \, | \As (\rn) \|
\end{\eq}
for some $c>0$, all $g\in \Cc^{\vr} (\rn)$ and all $f\in \As (\rn)$.
\end{theorem}

\begin{proof}
{\em Step 1.} Let $f$ be given by \eqref{3.1}. Roughly speaking  we wish to show that
\begin{\eq}   \label{3.18}
gf = \sum_{\substack{\beta \in \nat^n_0, j\in \no, \\ m \in \zn}} 2^{jn} \big(f, k^\beta_{j,m} \big) \, g\,\Phi^\beta_{j,m}
\end{\eq}
is closely related to expansions as described in Theorem \ref{T2.10} with $g \Phi^\beta_{j,m}$ as $L_\infty$--normalized building 
blocks. We rely again on Proposition \ref{P2.21} what requires to break down $g \Phi^\beta_{j,m}$ into related atoms having 
sufficiently many modified moment conditions according to \eqref{2.128}. For this purpose we assume first that $g$ has a compact 
support. Then we can argue as in Step 1 of the proof of Theorem \ref{T2.23} and
\begin{\eq}   \label{3.19}
\big(g \Phi^\beta_{j,m} \big)(x) = \sum_{k \in \zn} g(x) \, a^{L,\beta,k}_{j,m} (x)
\end{\eq}
with $a^{L,\beta,k}_{j,m}$ as in \eqref{2.40}, \eqref{2.41} makes sense. Let, in addition, $g\in C^L (\rn)$ where $C^L (\rn)$ is
normed according to \eqref{2.126} with $L= \vr \in \nat$ as in \eqref{3.16} = \eqref{2.130}. One has
\begin{\eq}  \label{3.20}
\Big| \int_{\rn} g(x) \, a^{L,\beta,k}_{j,m} (x) \, \di x \Big| \le c \, 2^{-j(L+n)} \langle k \rangle^{-D} 2^{-\vk |\beta|} \, \| g \,
| C^L (\rn) \|
\end{\eq}
by the same arguments resulting in \eqref{2.142} based on Taylor expansions of $g$ at $2^{-j} (m+k)$ with remainder terms of order
$L$. As there $D>0$ and $\vk >0$ are at our disposal.
If one uses the integral remainder term of the Taylor expansion for $g\in \Cc^{\vr} (\rn)$, $0<\vr \not= \nat$, as already done in
\cite[pp.\,205--206]{T92} in connection with assertions of type \eqref{3.17}, then one obtains
\begin{\eq}   \label{3.21}
\Big| \int_{\rn} g(x) \, a^{L,\beta,k}_{j,m} (x) \, \di x \Big| \le c \, 2^{-j(\vr +n)} \langle k \rangle^{-D} 2^{-\vk |\beta|}
\| g \, | \Cc^{\vr} (\rn) \|
\end{\eq}
where again $D>0$ and $\vk >0$ are at our disposal.
Using the well--known fact
that $\Cc^{\vr} (\rn)$ is a multiplication algebra (covered also by Theorem \ref{T3.6} below in the context of quarkonial 
arguments) one can extend \eqref{3.21} to
\begin{\eq}   \label{3.22}
\Big| \int_{\rn} \Psi (x) \, g(x) \, a^{L,\beta,k}_{j,m} (x) \, \di x \Big| \le c\, 2^{-j(\vr+n)} \langle k \rangle^{-D} 2^{-\vk 
|\beta|} \, \| g \, | \Cc^{\vr} (\rn) \| \cdot \| \Psi \, | \Cc^{\vr} (\rn) \|.
\end{\eq}
Now one can apply Proposition \ref{P2.21} to the $\vr$--atoms $g(x) \, a^{L,\beta,k}_{j,m} (x)$ with the indicated additional factors.
Recall again that $\As (\rn)$ is a $u$--Banach space with $u= \min (1,p,q)$. 
Summation over $k \in \zn$ based on \eqref{3.19} gives the $\beta$--term in \eqref{3.18}. Summation over $\beta \in \nat^n_0$ 
and application of Theorem \ref{T2.10} prove
\begin{\eq}   \label{3.23}
\| gf \, | \As (\rn) \| \le c \, \| g \, | \Cc^{\vr} (\rn) \| \cdot \| f \,|\As (\rn) \|
\end{\eq}
for compactly supported $g \in \Cc^{\vr} (\rn)$.
\cm
{\em Step 2.} Let now $g\in \Cc^{\vr} (\rn)$ with $\vr$ as in \eqref{3.16} be an arbitrary function. We argue now similarly as in Step
2 of the proof of Theorem \ref{T2.18}.
Let $\psi_M (x) = \psi (x-M)$,
$x\in \rn$, $M\in \zn$, be a suitable resolution of unity based on a compactly supported $C^\infty$--function $\psi$ as in 
\eqref{2.111}. By  \cite[Section 2.4.2, pp.\,43--44]{T20} (and the references given there) one has
\begin{\eq}   \label{3.24}
\| gf \, | \Fs (\rn) \| \sim \Big( \sum_{M\in \zn} \| \psi_M g \cdot \psi_M f \, | \Fs (\rn) \|^p \Big)^{1/p}
\end{\eq}
for all $s\in \real$ and $0<p,q \le \infty$. This covers in particular the spaces $\Fs (\rn)$ with $s<0$, $0<p<\infty$, $0<q\le \infty
$ and $F^s_{\infty, \infty} (\rn) = B^s_{\infty, \infty} (\rn)$ with $s<0$. 
Now one can apply Step 1 to $\psi_M g \in \Cc^{\vr} (\rn)$. Together 
with  a pointwise multiplier assertion $\psi_M g \to g$ one obtains
\begin{\eq}   \label{3.25}
\begin{aligned}
\| gf \, | \Fs (\rn) \| &\le c \, \|g \, | \Cc^{\vr} (\rn) \| \cdot \Big(\sum_{M\in \zn} \| \psi_M f \, | \Fs (\rn) \|^p \Big)^{1/p} \\
&\le c' \| g \, | \Cc^{\vr} (\rn) \| \cdot \| f \, | \Fs (\rn)\|.
\end{aligned}
\end{\eq}
But we have to add a technical comment. Formally \eqref{3.24} requires that one already knows  that $gf \in \Fs (\rn)$. This is justified
so far only for compactly supported $g\in \Cc^{\vr} (\rn)$ for which we have the uniform estimate \eqref{3.25}. The extension to
arbitrary $g\in \Cc^{\vr} (\rn)$ is now a matter of the Fatou property, \cite[Section 1.3.4, p.\,18]{T20}. The $B$--spaces can be
incorporated by real interpolation,
\begin{\eq}   \label{3.26}
\Bs (\rn) = \big( F^{s_0}_{p,p} (\rn), F^{s_1}_{p,p} (\rn) \big)_{\theta,q}
\end{\eq}
with $0<q\le \infty$, $s_0 <s<s_1 <0$ suitably chosen and $s= (1- \theta) s_0 + \theta s_1$.
\end{proof}

\begin{remark}   \label{R3.4}
We relied again on Proposition \ref{P2.21}. This requires to break down the molecules $\Phi^\beta_{j,m}$ into atoms resulting in 
\eqref{3.19}. 
A molecular version of Proposition \ref{P2.21} would simplify the above arguments. We discussed this point already in Remark 
\ref{R2.22}.  Furthermore instead
of \eqref{3.1} and \eqref{3.19} one could rely on corresponding atomic representations (with sufficiently many moment conditions) or
related wavelet expansions. This again would simplify the above arguments somewhat. Furthermore  it applies to all spaces $\As (\rn)
$ with $s\in \real$ and $0<p,q \le \infty$ (maybe $p<\infty$ for $F$--spaces). This has already been done in \cite{Scha13}. But the
above seemingly more complicated approach may serve as a prototype for more general operators $T$ in \eqref{3.3} than pointwise multiplication.
The proof of Theorem \ref{T2.23} may be considered as an example.
\end{remark}

\begin{remark}   \label{R3.5}
We have to add a second comment. The above theorem is already covered by \cite[Section 2.8.2, Theorem, Corollary, pp.\,140--144]{T83} based on \cite{Tr78}, where we
proved corresponding assertions for all $s\in \real$ using the heavy machinery of paramultiplication with the restriction $\vr > \max
(s, \sigma^{n}_p -s)$ for $B$--spaces and $\vr > \max (s, \sigma^{n}_{p,q} -s)$ for  $F$--spaces as in \eqref{3.16} (restricted
there to $s<0$)  and in
\cite{Scha13}. It had been observed in
\cite{Fra86}, relying again on paramultiplication, that the $q$--dependence in case of the $F$--spaces can be removed: One has 
\eqref{3.17} for all spaces $\As (\rn)$ with $s\in \real$ and $0<p,q \le \infty$ ($p<\infty$ for the $F$--spaces) if $\vr > \max (s,
\sigma^{n}_p -s)$. A shorter new proof of this assertion may be found in \cite[Section 4.2.2, Theorem, Corollary, pp.\,202--206]{T92}
based on local means. The main aim of \cite{Scha13} in this context
is not to prove new results, but to present the above--described method based on
Proposition \ref{P2.21}. Finally we refer the reader to \cite[Theorem 2.30, p.\,41]{T20}
where we incorporated the spaces $F^s_{\infty,q} (\rn)$, $0<q \le \infty$. In other words, one has \eqref{3.17} for all spaces $\As
(\rn)$, $A \in \{B,F \}$, $s\in \real$ and $0<p,q \le \infty$ if $\vr > \max (s, \sigma^{n}_p -s)$. In \cite{T20}
 one finds also further
discussions. In particular, Theorem \ref{T3.3} is not new and better results are known. But there is good reason to believe that the
above approach can be applied to more sophisticated operators $T$ in \eqref{3.3} using the distinguished properties of 
$\Phi^\beta_{j,m}$ as described in \eqref{2.9}--\eqref{2.11}.
\end{remark}

\subsection{Multiplication algebras}    \label{S3.2}
We ask for further applications of quarkonial representations. So far we used Theorem \ref{T2.10} in the preceding 
Section \ref{S3.1} to have a new
look at pointwise multipliers. Now we wish to use Theorem \ref{T2.4} to deal with multiplication algebras. Nothing new can be 
expected because one has already final answers under which conditions  a space $\As (\rn)$ is a multiplication algebra. Definitions,
related assertions and references may be found in \cite[Section 2.4.3, pp.\,44--46]{T20}. In particular, a quasi--Banach space $A(\rn)$
on $\rn$ with
\begin{\eq}  \label{3.48}
S(\rn) \hra A(\rn) \hra S'(\rn), \qquad A(\rn) \subset L^{\loc}_1 (\rn)
\end{\eq}
is said to be a multiplication algebra if $f_1 f_2 \in A(\rn)$ whenever $f_1 \in A(\rn)$, $f_2 \in A(\rn)$ and if there is a constant
$c >0$ such that
\begin{\eq}   \label{3.49}
\| f_1 f_2 \, | A(\rn) \| \le c\,\| f_1 \, | A(\rn) \| \cdot \| f_2 \, | A(\rn) \|
\end{\eq}
for $f_1 \in A(\rn)$, $f_2 \in A(\rn)$. Note that we assume that $A(\rn)$ consists entirely of regular distributions. Then the product
of two elements makes sense at least pointwise almost everywhere.
We are now interested to indicate how 
Theorem \ref{T2.4} can be used in this context. It  may well be possible that one can recover the already known final
assertions in this way. But this is not our aim. We wish to describe the basic ideas and choose for this purpose the simplest case. 
These are the spaces
\begin{\eq}   \label{3.27}
\Bs (\rn) \qquad \text{with $0<p,q \le \infty$ and $s>n/p$}.
\end{\eq}

\begin{theorem}   \label{T3.6}
The spaces $\Bs (\rn)$ in \eqref{3.27} are multiplication algebras.
\end{theorem}

\begin{proof}
First we remark that
\begin{\eq}   \label{3.28}
\Bs (\rn) \hra L_\infty (\rn) \qquad \text{where $0<p,q \le \infty$ and $s>n/p$}.
\end{\eq} 
Let $f_1, f_2 \in \Bs (\rn)$ be represented according to Theorem \ref{T2.4} as
\begin{\eq}   \label{3.29}
f_l  = \sum_{\substack{\beta^l \in \nat^n_0, j^l \in \no, \\ m^l \in \zn}} 2^{j^l n} \big( f_l, \Phi^{\beta^l}_{j^l, m^l} \big)
k^{\beta^l}_{j^l, m^l}
\end{\eq}
with 
\begin{\eq}   \label{3.30}
\| f_l\, | \Bs (\rn) \| \sim \| \lambda (f_l) \, | b^s_{p,q} (\rn)^0 \|
\end{\eq}
in the notation used there.
We ask for a corresponding representation of $f_1 f_2$,
\begin{\eq}   \label{3.31}
f_1 f_2  = \sum_{\substack{\beta \in \nat^n_0, j \in \no, \\ m \in \zn}} 2^{j n} \big( f_1 f_2, \Phi^{\beta}_{j, m} \big)
k^{\beta}_{j, m}
\end{\eq}
such that
\begin{\eq}   \label{3.32}
\begin{aligned}
& \| f_1 f_2 \, | \Bs (\rn) \| \\
&\sim \sup_{\beta \in \nat^n_0} \bigg( \sum^\infty_{j=0} 2^{j(s - \frac{n}{p})q} \Big( \sum_{m\in \zn} \big| 2^{jn} \big( f_1 f_2,
\Phi^\beta_{j,m} \big)\big|^p \Big)^{q/p} \bigg)^{1/q}
\end{aligned}
\end{\eq}
is finite. By \eqref{3.29} one has
\begin{\eq}   \label{3.33}
\begin{aligned}
& 2^{jn} \big( f_1 f_2, \Phi^\beta_{j,m} \big) \\
&= 2^{jn} \sum_{\substack{\beta^1, \beta^2 \in \nat^n_0; j^1, j^2 \in \no;\\ m^1, m^2 \in \zn}} 2^{(j^1 + j^2)n}
\big(f_1, \Phi^{\beta^1}_{j^1, m^1} \big) \big( f_2, \Phi^{\beta^2}_{j^2, m^2} \big) \cdot \big( k^{\beta^1}_{j^1,m^1} 
k^{\beta^2}_{j^2,m^2}, \Phi^\beta_{j,m} \big).
\end{aligned}
\end{\eq}
By the exponential decay similarly as in \eqref{2.34} and
 \eqref{2.40}--\eqref{2.42} it is sufficient to deal with $\beta = \beta^1 = \beta^2 =0$. Let
$\Phi_{j,m}^0 = \Phi_{j,m}$ and $k^0_{j,m} = k_{j,m}$. There are two cases: $j\ge j^1$ with $j^1 \ge j^2$ and $j<j^1$ again
with $j^1 \ge j^2$. By  \eqref{2.11} and \eqref{3.28} one has always
\begin{\eq}   \label{3.34}
2^{j^2 n} \big| \big( f_2, \Phi_{j^2,m^2} \big) \big| \le c \, \|f_2 \, | L_\infty (\rn) \| \le c' \, \|f _2 \, | \Bs (\rn) \|.
\end{\eq}
Let $m\in \zn$ and $j\in \no$. Let $j^1 \le j$. We apply \eqref{3.19}, \eqref{3.20} with $g= k_{j^1,m^1} k_{j^2,m^2}$ where $D>0$ and
$L\in \nat$ are at our disposal. For fixed $j, j^1, m$ there are $\sim 1$ main terms in \eqref{3.33}
with $2^{-j} m \sim 2^{-j^1} m^1 \sim 2^{-j^2} m^2$ which can be estimated as
\begin{\eq}   \label{3.35}
2^{jn}  \Big| \int_{\rn} k_{j^1,m^1} (x) \, k_{j^2, m^2} (x) \, \Phi_{j,m} (x) \, \di x \Big| \le c \,2^{-(j-j^1)L}.
\end{\eq}
Any couple $(j^1, m^1)$ contributes to $\sim 2^{(j-j^1)n}$ terms $(j,m)$. Summation over $m\in \zn$
for fixed $j$ and $j^1$ with $j^1 \le j$ results in
\begin{\eq} \label{3.36}
\begin{aligned}
& 2^{jnp}
\sum_{2^{j-j^1} m^1 \sim m\in \zn} 2^{j^1np} \big| \big( f_1, \Phi_{j^1,m^1} \big) \big|^p \, \big|  \big( k_{j^1,m^1} k_{j^2,m^2},
\Phi_{j,m} \big) \big|^p \\
&\le c \, 2^{-(j-j^1)Lp} \, \sum_{m^1\in \zn} 2^{j^1 np} \big| \big( f_1, \Phi_{j^1, m^1} \big) \big|^p
\end{aligned}
\end{\eq}
where again $L\in \nat$ is at our disposal (compensating also the additional factor $2^{(j-j^1)n}$). Let $j<j^1$ (and again $j^2 \le j^1$). For fixed $m\in \zn$ there are now $\sim 2^{(j^1 - j)n}$ terms with
 $2^{-j}m \sim 2^{-j^1} m^1 \sim 2^{-j^2} m^2$ in \eqref{3.33}. For the related terms one has now
\begin{\eq}   \label{3.37}
\begin{aligned}
\sum_{m^1: 2^{-j^1} m^1 \sim 2^{-j} m} \Big| \int_{\rn} k_{j^1,m^1} (x) \, k_{j^2,m^2} \, \Phi_{j,m} (x) \, \di x \Big|
& \le c\, 2^{(j^1 -j)n} 2^{-j^1 n} \\
& = c \, 2^{-jn}
\end{aligned}
\end{\eq}
and
\begin{\eq}   \label{3.38}
\begin{aligned}
&\sum_{m^1: 2^{-j^1} m^1 \sim 2^{-j}m} 2^{j^1 n} \big| \big( f_1, \Phi_{j^1,m^1} \big) (k_{j^1,m^1} k_{j^2, m^2}, \Phi_{j,m} \big) \big|
\\
&\le c \, 2^{-jn} \sup_{m^1: 2^{-j^1} m^1 \sim 2^{-j} m} 2^{j^1 n} \big| \big(f_1, \Phi_{j^1, m^1} \big) \big| \\
&\le c \, 2^{-jn} \Big(\sum_{m^1: 2^{-j^1} m^1 \sim 2^{-j} m} \big| 2^{j^1 n} \big(f_1, \Phi_{j^1,m^1} \big) \big|^p \Big)^{1/p}.
\end{aligned}
\end{\eq}
For fixed $j$ and  $j^1$ with $j < j^1$  and summation over $m\in \zn$ one obtains now
\begin{\eq} \label{3.39}
\begin{aligned}
&\sum_{m \in \zn} \bigg( 
\sum_{2^{j-j^1} m^1 \sim m} 2^{j^1 n} \big| \big( f_1, \Phi_{j^1,m^1} \big)  \big( k_{j^1,m^1} k_{j^2,m^2},
\Phi_{j,m} \big) \big| \bigg)^p \\
& \le c \, 2^{-jnp} \sum_{m^1 \in \zn} 2^{j^1 np} \big| \big( f_1, \Phi_{j^1,m^1} \big) \big|^p.
\end{aligned}
\end{\eq}
Let $j\in \no$ be fixed. Inserting \eqref{3.36} for $j^1 \le j$ in \eqref{3.32}, \eqref{3.33} (where the power $p$ is compensated by
a suitably chosen $L$) and \eqref{3.39} for $j^1 >j$ in 
\eqref{3.32}, \eqref{3.33} one obtains
\begin{\eq}   \label{3.40}
\begin{aligned}
&2^{j (s- \frac{n}{p})p} \, \sum_{m\in \zn} \big| 2^{jn} (f_1 f_2, \Phi_{j,m}) \big|^p \\
&\le c\, \|f_2 \, | L_\infty (\rn) \|^p \, \Big( \sum_{j^1 \le j} 2^{-(j-j^1) (L-s +\frac{n}{p})p} + \sum_{j^1 >j} 2^{-(j^1 -j)(s - 
\frac{n}{p})p} \Big) \\
&\quad \cdot \sum_{m^1 \in \zn} 2^{j^1 (s- \frac{n}{p})p} \big| 2^{j^1 n} (f_1, \Phi_{j^1, m^1} ) \big|^p.
\end{aligned}
\end{\eq} 
We choose $s- \frac{n}{p} <L \in \nat$ and use $s > \frac{n}{p}$. We insert \eqref{3.40} in \eqref{3.32}. This covers the main term
$\beta = \beta_1 = \beta_2 =0$. As far as the remaining terms are concerned we use
\begin{\eq}  \label{3.41}
\sum_{j\in \nat} \Big( \sum_{l\in \nat} 2^{-\delta |j-l|} d_l \Big)^{\vk} \le c \, \sum_{j\in \nat} d_l^{\vk}, \qquad \delta >0, \quad
\vk >0,
\end{\eq}
$d_l \ge 0$. Then it follows from \eqref{3.32}, its counterpart with $f_1$ in place of $f_1 f_2$ and \eqref{3.34} that
\begin{\eq}   \label{3.42}
\begin{aligned}
&\| f_1 f_2 \, | \Bs (\rn) \| \\ &\le c \, \| f_1 \, | \Bs (\rn) \| \cdot \| f_2 \, | L_\infty (\rn) \|
 +c \, \| f_2 \, | \Bs (\rn) \| \cdot \| f_1 \, | L_\infty (\rn) \| \\
 &\le c' \, \|f_1 \, | \Bs (\rn) \| \cdot \| f_2 \, | \Bs (\rn) \|.
\end{aligned}
\end{\eq}
This proves the theorem.
\end{proof}

\begin{remark}   \label{R3.7}
Recall that $\As (\rn) \subset L^{\loc}_1 (\rn)$ is a multiplication algebra if, and only if,
\begin{\eq}   \label{3.43}
\begin{cases}
\text{either $s>n/p$} &\text{where $0<p,q \le \infty$}, \\
\text{or $s = n/p$}   &\text{where $0<p<\infty$, $0<q \le 1$},
\end{cases}
\end{\eq}
for the $B$--spaces and
\begin{\eq}   \label{3.44}
\begin{cases}
\text{either $s>n/p$} &\text{where $0<p,q \le \infty$}, \\
\text{or $s = n/p$}   &\text{where $0<p \le 1$, $0<q \le \infty$},
\end{cases}
\end{\eq}
for the $F$--spaces. One may consult \cite[Section 2.4.3, pp.\,44-46]{T20} and the  references to the long and substantial history
given there. In other words, the above theorem is a new proof of the if--part of \eqref{3.43} with exception of the limiting case with
$s= n/p$. It might well be possible  to incorporate this limiting case what might require some additional efforts. This may apply also
to the spaces $\Fs (\rn)$ with \eqref{3.44} based on the Fefferman--Stein vector--valued maximal inequality
\begin{\eq}   \label{3.45}
\Big\| \Big( \sum^\infty_{k=0} \big( M |g_k|^w \big) (\cdot)^{q/w} \Big)^{1/q} | L_p (\rn) \Big\| \le c \, \Big\| \Big( \sum^\infty_{k=0} |g_k  (\cdot) |^q \Big)^{1/q} | L_p (\rn) \Big\|
\end{\eq}
where
\begin{\eq}   \label{3.46}
0<p<\infty, \quad 0<q \le \infty, \quad 0<w< \min (p,q)
\end{\eq}
and
\begin{\eq}   \label{3.47}
(Mg)(x) = \sup_{x \in Q} |Q|^{-1} \int_Q |g(y)| \, \di  y, \qquad x \in \rn,
\end{\eq}
where the supremum is taken over all cubes $Q$ with $x\in Q$. But this will not be done here. It surely requires some additional efforts. We are interested here mainly in the method which may also 
be of some use for others (near--by) problems. Basically one reduces the desired assertions to mapping properties of related sequence
spaces with \eqref{3.35} as the crucial observation. Similar ideas had already been used in \cite[Section 1.4.2, pp.\,77--83]{T19}
where we dealt with multiplication algebras for the spaces $S^r_p B (\rn)$ with dominating mixed smoothness. Instead of the quarkonial
systems in \eqref{2.13} and \eqref{2.14} we relied there on products of one--dimensional wavelet system of type \eqref{1.43}. This 
method can also be used to study multiplication algebras for the Morrey smoothness spaces $\Lambda^{\vr} \As (\rn)$ and $\Lambda_{\vr}
\As (\rn)$. The related (almost) final results may be found in \cite[Theorem 5.13]{HaT22}. The underlying proofs of \cite[Theorem 2.43,
p.\,91]{T13} and \cite[Theorem 3.60, p.\,95]{T14} are based on the same type of arguments as above with the wavelets $\psi^{j^1}_{G^1,
m^1}$, $\psi^{j^2}_{G^2, m^2}$ and $\psi^j_{G,m}$ according to \eqref{1.43} in place of $k^{\beta_1}_{j^1, m^1}$, $k^{\beta_2}_{j^2,
m^2}$ and $\Phi^\beta_{j,m}$ in \eqref{3.33} and related estimates of type \eqref{3.35}. The original proofs ensuring that 
\eqref{3.43}, \eqref{3.44} are necessary and sufficient to ensure that $\As (\rn)$ is a multiplication algebra are based on the heavy
machinery of paramultiplication. The use of building blocks  (wavelets or quarks) might be simpler and more transparent. But it
requires again a lot of efforts and complicated estimates. Multiplication algebras seem to be in any case a tricky task.
\end{remark}

\subsection{Positivity}    \label{S3.3}
As indicated in the Abstract and in the Introduction, Section \ref{S1.1}, the recent theory of function spaces is governed by building
blocks, preferably atoms and wavelets, complemented now by quarks. They have their advantages and disadvantages, as well as some
remarkable interrelations. There are numerous striking applications of atoms and wavelets.  Quite recently we dealt in \cite{Tri23},
based on \cite{Tri22}, with mapping properties of so--called Fourier operators $F_\tau \in \Phi^\sigma_{1,\delta} (\rn)$, based on the
decomposition
\begin{\eq}   \label{3.50}
F_\tau = T_\tau \circ F, \qquad T_\tau \in \Psi^\sigma_{1,\delta} (\rn),
\end{\eq}
where $F$ is the Fourier transform and $\Psi^\sigma_{1,\delta} (\rn)$ are the H\"{o}rmander classes of pseudodifferential operators
with $\sigma \in \real$ and $0 \le \delta \le 1$. Mapping properties of $T_\tau \in \Psi^0_{1,\delta} (\rn)$, $0\le \delta \le 1$, in
the spaces
\begin{\eq}   \label{3.51}
\Bs (\rn), \ s> \sigma^n_p \qquad \text{and} \qquad \Fs (\rn), \ s > \sigma^n_{p,q}
\end{\eq}
with $\sigma^n_p$ and $\sigma^n_{p,q}$ as in \eqref{3.14} (or \eqref{3.60} below) are reduced in \cite{Tri23} to the wavelet
counterpart of \eqref{3.3}, this means
\begin{\eq}   \label{3.52}
T_\tau f =  \sum_{\substack{j\in \no, G\in G^j, \\ m\in \zn}}
2^{jn} \big( f, \psi^j_{G,m} \big) T_\tau \, \psi^j_{G,m}
\end{\eq}
based on Theorem \ref{T1.12}. The related proofs in \cite{Tri23} rely on the crucial property
\begin{\eq}   \label{3.53}
\big| D^\alpha_\xi \big(\psi^0_{G,0} \big)^\wedge (\xi) \big| \le c \, \frac{|\xi|^v}{(1 + |\xi|)^w}, \qquad \xi \in \rn, \quad
|\alpha| \le L, \quad G \in G^*,
\end{\eq}
where $L\in \no$, $v\in \nat$ and $w\in \nat$ are (independently) at disposal. The attempt to use Theorem \ref{T2.4} instead of
Theorem \ref{T1.12} for the above mapping properties requires a substitute of \eqref{3.53} with, say, $k^0_{0,0}$ in place of
$\psi^0_{G,0}$ (in the context of the reasoning in \cite{Tri23}). But this is not available and also not true. But the missing
quarkonial counterpart of \eqref{3.53} is just the main advantage of what follows. In this Section \ref{S3.3} we deal with a peculiar
property of function spaces, called positivity, demonstrating the advantages of quarks, compared with other building blocks in this
context. 

A distribution $f\in S'(\rn)$ is said to be real if $f(\vp) \in \real$ for any real $\vp \in S(\rn)$. A distribution $f\in S'(\rn)$
is said to be positive, $f \ge 0$, if $f(\vp) \ge 0$ for any $\vp \in S(\rn)$ such that
$\vp \ge 0$ (which means $\vp (x) \ge 0$ for all $x\in \rn$). For any
$\vp \in S(\rn)$ there are  positive numbers $c_j$ with $c_{j+1} \ge c_j +1$, $j\in \nat$, such that
\begin{\eq}  \label{3.54}
|\vp (x) | \le |x|^{-j} \qquad \text{if} \quad |x| >c_j.
\end{\eq}
By elementary reasoning one can construct a function $\vp_1 \in S(\rn)$ such that
\begin{\eq}   \label{3.55}
\vp_1 (x) \ge |\vp (x)|, \qquad \vp_1 (x) \sim |x|^{-j} \quad \text{if} \quad c_j \le |x| \le c_{j+1},
\end{\eq}
where the equivalence constants $\sim$ are independent of $j\in \nat$. Let, furthermore $\vp_1 (x) \ge |\vp (x)|$ if $|x| \le 
c_1$.  If, in addition, $\vp$ is real, then \eqref{3.55} can be complemented by
\begin{\eq}   \label{3.56}
\vp (x) = \vp_1 (x) - \vp_2 (x), \qquad \vp_1 \ge 0, \quad \vp_2 = \vp_1 - \vp \ge 0,
\end{\eq}
$\vp_1 \in S(\rn)$, $\vp_2 \in S(\rn)$. If $f\in S'(\rn)$ is positive and if $\vp$ is real then
\begin{\eq}   \label{3.57}
f(\vp) = f(\vp_1) - f(\vp_2)
\end{\eq}
shows that $f$ is real (as it should be).

\begin{definition}   \label{D3.8}
The space $\As(\rn)$ with $A \in \{B,F \}$, $s\in \real$ and $0<p,q \le \infty$ is said to have the positivity property if any real
$f\in \As (\rn)$ can be decomposed as
\begin{\eq}  \label{3.58}
f= f^1 - f^2 \quad \text{with} \quad f^l \ge 0, \quad f^l \in \As (\rn) \quad \text{for} \quad l= 1,2,
\end{\eq}
and
\begin{\eq}   \label{3.59}
\| f \, | \As (\rn) \| \sim \|f^1 \, | \As (\rn) \| + \| f^2 \, | \As (\rn) \|
\end{\eq}
with equivalence constants which are independent of $f$.
\end{definition}

\begin{remark}   \label{R3.9}
We dealt several times with the positivity property of the spaces $\As (\rn)$. The above version is a modification of \cite[Definition
3.4, p.\,88]{T20} which, in turn, is based on \cite[Section 3.3.2, pp.\,190--192]{T06} and \cite{Tri03}. We do not repeat the 
discussions and  results obtained there. We concentrate on a few points which might be considered as tiny improvements (or 
complements) of what is already known. 
\end{remark}

Let again
\begin{\eq}   \label{3.60}
\sigma^n_p = n \Big( \max \big( \frac{1}{p}, 1 \big) -1 \Big), \qquad \sigma^n_{p,q} = n \Big( \max \big( \frac{1}{p}, \frac{1}{q}, 1
\big) - 1 \Big),
\end{\eq}
where $0<p,q \le \infty$ as in \eqref{1.30} and \eqref{3.14}.

\begin{proposition}  \label{P3.10}
Let $0<p,q \le \infty$. Then the spaces
\begin{\eq}   \label{3.61}
\Bs (\rn) \qquad \text{with} \quad s > \sigma^{n}_p
\end{\eq}
and
\begin{\eq}  \label{3.62}
\Fs (\rn) \qquad \text{with} \quad s > \sigma^{n}_{p,q}
\end{\eq}
have the positivity property.
\end{proposition}

\begin{proof}
Let $p<\infty$ for the $F$--spaces. We rely on Theorem \ref{T2.4}.
If $f$ is real then it follows from \eqref{2.12} that also the coefficients 
\begin{\eq}   \label{3.63}
\lambda^\beta_{j,m} (f) = 2^{jn} \big(f, \Phi^\beta_{j,m} \big)
\end{\eq}
in \eqref{2.30}, \eqref{2.31} are real. Let
\begin{\eq}   \label{3.64}
f^1 = \sum_{\substack{\beta \in \nat^n_0, j\in \no, \\ m \in \zn}}  \max \big( \lambda^\beta_{j,m} (f), 0 \big) \, k^\beta_{j,m}.
\end{\eq}
Then the desired decomposition \eqref{3.58}, \eqref{3.59} follows from Theorem \ref{T2.4} and \eqref{2.3}, \eqref{2.4}.
As far as
the spaces $F^s_{\infty,q} (\rn)$ with $s>0$ and $0<q \le \infty$ are concerned we refer the reader to \cite[p.\,89]{T20} for rather
direct arguments.
\end{proof}

\begin{remark}   \label{R3.11}
According to \cite[Theorem 3.6, p.\,88]{T20} there are also some other spaces $\Fs (\rn)$ having the positivity  property. On the other
hand, spaces $\As (\rn)$ with $0<p,q \le \infty$ and $s<\sigma^{n}_p$ do not have the positivity property. This follows from the
remarkable observation that any positive  distribution $f\in S'(\rn)$ can be represented as
\begin{\eq}   \label{3.65}
f(\vp) = \int_{\rn} \vp (x) \, \mu (\di x), \qquad \vp \in S(\rn),
\end{\eq}
where $\mu$ is a uniquely determined (positive) Radon measure. Then $f$ belongs at least locally to $B^0_{1,\infty} (\rn)$. This
requires $s \ge \sigma^{n}_p$. Details and references in particular to \cite{T06} and \cite{Mal95}
may be found in \cite[Section 3.2, pp.\,87--90]{T20}.
\end{remark}

\begin{remark}   \label{R3.12}
In contrast to \cite{Tri03}, \cite{T06} and \cite{T20} the positivity property of $\As (\rn)$ is based on the decomposition 
\eqref{3.58} for real elements as a refinement of the related general assertion in Theorem \ref{T2.4}. 
This natural simplification relies on the
(simple but not totally obvious) observation that the crucial building blocks $\Phi^\beta_{j,m}$ in \eqref{2.12} are real. This has
been overlooked when we dealt with the positivity property in \cite{Tri03} and \cite{T06}. In \cite{T20} we complemented what was
already known by corresponding assertions for the spaces $F^s_{\infty,q} (\rn)$. The above tiny improvement may justify that we returned again to this topic as a supplement of \cite[Section 3.2, pp.\,87--90]{T20}. If $f$ belongs to a space $\As (\rn)$ covered
by the above proposition then it can be decomposed by positive $f^l \in \As (\rn)$, $1 \le l \le 4$, such that
\begin{\eq}   \label{3.64a}
f= f^1 - f^2 + i f^3 - if^4
\end{\eq}
and
\begin{\eq}   \label{3.65a}
\|f \, | \As (\rn) \| \sim \sum^4_{l=1} \| f^l \, |\As (\rn) \|.
\end{\eq}
This follows from an obvious modification of \eqref{3.64}, where one does not need that $\Phi^\beta_{j,m}$ in \eqref{3.63} is real.
This is the version considered so far in the above--mentioned literature.
\end{remark}

\begin{remark}   \label{R3.13}
Any $f\in S'(\rn)$ can be decomposed as
\begin{\eq}   \label{3.66}
f = f^1 + i f^2 \qquad \text{where both $f^1$ and $f^2$ are real}.
\end{\eq}
This can be justified by elementary reasoning. But it follows also from the universal representation Theorem \ref{T2.30} using that
$\Phi^\beta_{j,m}$ is real. One can say more. According to \eqref{2.172} any $f\in S'(\rn)$ belongs for fixed $p$ and $q$ with
$0<p,q \le \infty$ to some space $\Bs (\rn, w_\delta)$, say, with $s<0$. Then one has by Theorem \ref{T2.23} that both $f^1$ and $f^2$
are elements of the same space and
\begin{\eq}   \label{3.67}
\|f \, |\Bs (\rn, w_\delta)\| \sim \|f^1 \, |\Bs (\rn, w_\delta)\| + \|f^2 \, |\Bs (\rn, w_\delta)\|
\end{\eq}
where the equivalence constants are independent of $f$. The discussion in Remark \ref{R3.11} shows that a further decomposition into
positive distributions similarly as in \eqref{3.64a}, \eqref{3.65a} is not possible in general.
\end{remark}

\section{Quarks in domains}    \label{S4}
\subsection{Definitions and preliminaries}    \label{S4.1}
We extended in \cite{T08} the wavelet theory for the spaces $\As (\rn)$ with $A \in \{B,F \}$, $0<p,q \le \infty$ ($p<\infty$ for 
$F$--spaces) and $s\in \real$ from $\rn$ to corresponding spaces $\As (\Om)$ on domains (= open sets) $\Om$ on \rn, and some suitable 
subspaces. One may ask for a quarkonial counterpart of this theory. This is the main topic of the present
 Section \ref{S4}. Whereas the 
wavelet theory for the spaces $\As (\Om)$ is comparable with its \rn--counterpart we could not reach the same level for the related
quarkonial counterpart. But we obtained some (as we hope) substantial assertions, discussing also a few shortcomings. This may serve
as a basis for future research. But first we collect some definitions and useful properties following mainly \cite{T08}.

Open sets in \rn, $n\in \nat$, are called domains. We always assume that a domain $\Om$ in $\rn$ is not empty, $\Om \not= \emptyset$,
and usually that it does not coincide with \rn, $\Om \not= \rn$. For $2 \le n\in \nat$,
\begin{\eq}   \label{4.1}
\Rn \ni x' \mapsto h(x') \in \real
\end{\eq}
is called a {\em Lipschitz function} (on $\Rn$) if there is a number $c>0$ such that
\begin{\eq}   \label{4.2}
|h(x') - h(y')| \le c \, |x' - y'| \quad \text{for all $x'\in \Rn$, $y' \in \Rn$}.
\end{\eq}
A {\em special Lipschitz domain} in \rn, $2\le n \in \nat$, is the collection of all points $x= (x', x_n)$ with $x' \in \Rn$ such that
\begin{\eq}   \label{4.3}
h(x') <x_n <\infty,
\end{\eq}
where $h(x')$ is a Lipschitz function according to \eqref{4.1}, \eqref{4.2}. A {\em bounded Lipschitz domain} in \rn, $2\le n \in 
\nat$, is a bounded domain $\Om$ in $\rn$ where the boundary $\Gamma = \pa \Om$ can be covered by finitely many open balls $B_j$
with $j=1, \ldots, J$, centered at $\Gamma$ such that
\begin{\eq}   \label{4.4}
B_j \cap \Om = B_j \cap \Om_j \qquad \text{for} \quad j=1, \ldots, J,
\end{\eq}
where $\Om_j$ are rotations of suitable special Lipschitz domains in \rn. Let $l(Q)$ be the side--length of a finite cube $Q$ in $\rn$
with sides parallel to the axes of coordinates. A domain (= open set) $\Om$ in $\rn$ with $\Om \not= \rn$ and $\Gamma = \pa \Om$ is 
said to be $E$--thick (exterior thick) if one finds for any interior cube $Q^i \subset \Om$ with
\begin{\eq}   \label{4.5}
l (Q^i) \sim 2^{-j} \qquad \dist (Q^i, \Gamma) \sim 2^{-j}, \qquad j \ge j_0 \in \nat,
\end{\eq}
a complementing  exterior cube $Q^{e} \subset \Om^c = \rn \setminus \Om$ with
\begin{\eq}   \label{4.6}
l(Q^{e}) \sim 2^{-j}, \qquad \dist(Q^{e},\Gamma) \sim \dist (Q^i, Q^e) \sim 2^{-j}, \qquad j \ge j_0 \in \nat.
\end{\eq}
Here
\begin{\eq}  \label{4.7}
\dist (\Gamma^1, \Gamma^2) = \inf \big\{ |x^1 - x^2|: \ x^1 \in \Gamma^1, \ x^2 \in \Gamma^2 \big\}
\end{\eq}
for two sets $\Gamma^1$ and $\Gamma^2$ in \rn. One finds in \cite[Section 3.1, pp.\,69--77]{T08} a detailed discussion of diverse
types of sets and domains in $\rn$ and their interrelations. But the above cases are sufficient for our purposes. Special and bounded
Lipschitz domain $\Om$ in $\rn$ are $E$--thick with $|\Gamma| =0$, where $|\Gamma|$ is the Lebesgue measure of $\Gamma = \pa \Om$.
The classical snowflake domain in $\real^2$ is also $E$--thick and $|\Gamma| =0$. For arbitrary $E$--thick domains in $\rn$ one has not
necessarily $|\Gamma| =0$, counter--examples may be found in \cite[p.\,75]{T08}. But this is not needed in what follows.

Let $\Om$ be an arbitrary domain in \rn. Then $D(\Om) = C^\infty_0 (\Om)$ stands for the collection of all complex--valued infinitely
differentiable functions in $\rn$ with compact support in \Om. Let $D'(\Om)$ be the dual space of all distributions on \Om. Let $g \in
S'(\rn)$. Then we denote by $g|\Om$ its restriction to \Om,
\begin{\eq}   \label{4.8}
g|\Om \in D'(\Om), \qquad (g|\Om)(\vp) = g(\vp) \qquad \text{for} \quad \vp \in D(\Om).
\end{\eq}

\begin{definition}   \label{D4.1}
Let $\As (\rn)$ with $A \in \{B,F \}$, $s\in \real$ and $0<p,q \le \infty$ be the spaces as introduced in Definition \ref{D1.1}. Let 
$\Om$ be an arbitrary domain in $\rn$ with $\Om \not= \rn$.
\cm
{\em (i)} Then
\begin{\eq}   \label{4.9}
\As (\Om) = \{ f \in D'(\Om): \ f = g|\Om \ \text{for some $g\in \As (\rn)$} \},
\end{\eq}
\begin{\eq}   \label{4.10}
\| f \, | \As (\rn) \| = \inf \| g \, | \As (\rn) \|
\end{\eq}
where the infimum is taken over all $g  \in \As (\rn)$ with $g|\Om =f$.
\cm
{\em (ii)} Let
\begin{\eq}  \label{4.11}
\wt{A}^s_{p,q} (\ol{\Om}) = \{ f \in \As (\rn): \ \supp f \subset \ol{\Om} \}.
\end{\eq}
Then
\begin{\eq}   \label{4.12}
\wt{A}^s_{p,q} (\Om) = \{ f \in D'(\Om): \ f = g|\Om \ \text{for some $g\in \wt{A}^s_{p,q} (\ol{\Om})$} \},
\end{\eq}
\begin{\eq}   \label{4.13}
\| f \, |\wt{A}^s_{p,q} (\Om) \| = \inf \| g \, | \As (\rn) \|,
\end{\eq}
where the infimum is taken over all $g\in \wt{A}^s_{p,q} (\ol{\Om})$ with $g|\Om =f$.
\end{definition}   

\begin{remark}   \label{R4.2}
This coincides with \cite[Definition 2.1, p.\,28]{T08}. There one finds also some discussions, including the question under which 
circumstances one can identify $\wt{A}^s_{p,q} (\Om)$ with $\wt{A}^s_{p,q} (\ol{\Om})$. But this does not play any role in what
follows. We prefer the $F$--spaces for reasons which will be clear soon.
\end{remark}

Let again $\psi$ be a compactly supported $C^\infty$ function in $\rn$ such that
\begin{\eq}   \label{4.14}
\sum_{M \in \zn} \psi_M (x) =1, \qquad \psi_M (x) = \psi (x-M), \quad x \in \rn,
\end{\eq}
as already used in \eqref{2.36}. Let $s\in \real$ and $0<p,q \le \infty$. Then
\begin{\eq}   \label{4.15}
\| f \, | \Fs (\rn) \| \sim \Big( \sum_{M \in \zn} \| \psi_M f \, | \Fs (\rn) \|^p \Big)^{1/p}
\end{\eq}
are equivalent quasi--norms, \cite[Theorem 2.35, p.\,43]{T20} and  the references given there (this includes $F^s_{\infty,q} (\rn)$).
We used in \eqref{2.111}, \eqref{2.112} the weighted extension.
Secondly we recall what is meant by {\em local homogeneity}, restricted to the cases of interest later on. Let again
\begin{\eq}   \label{4.16}
\sigma^{n}_p = n \Big( \max \big( \frac{1}{p}, 1 \big) -1 \Big), \qquad \sigma_{p,q}^{n} = n \Big( \max \big( \frac{1}{p},
\frac{1}{q}, 1 \big) - 1 \Big),
\end{\eq}
$0< p,q \le \infty$. Let
\begin{\eq}   \label{4.17}
0<p \le \infty, \quad 0<q \le \infty, \qquad s> \sigma^n_p
\end{\eq}
(with $q>1$ if $p=\infty$) and $0<\lambda \le 1$. Then
\begin{\eq}   \label{4.18}
\| f(\lambda \cdot) \, | \Fs (\rn) \| \sim \lambda^{s- \frac{n}{p}} \, \| f \, | \Fs (\rn) \|
\end{\eq}
for
\begin{\eq}   \label{4.19}
f \in \Fs (\rn) \qquad \text{with} \quad \supp f \subset \{ x \in \rn : \ |x| \le \lambda \}
\end{\eq}
where the equivalence constants in \eqref{4.18} are independent of $\lambda$. This is covered by \cite[Section 3.3, pp.\,90--93]{T20}
and the references given there. It applies in particular to $\Cc^s (\rn) = B^s_{\infty, \infty} (\rn) = F^s_{\infty, \infty} (\rn)$,
$s>0$.

Let $\Om$ be an arbitrary domain (= open set) in $\rn$ with $\Om \not= \rn$, what means that $\Gamma = \pa \Om \not= \emptyset$. Let
again
$Q_{J,M} = 2^{-J} M + 2^{-J} (0,1)^n$ with  $J \in \no$ and $M \in \zn$ be the above cubes
as already used several times, beginning with \eqref{1.7}. Let $ C Q_{J,M}$, $C >0$, be the
cube concentric with $Q_{J,M}$ having side--length $C 2^{-J}$. Let 
\begin{\eq}   \label{4.20}
\Om = \bigcup_{J,M} \ol{Q_{J,M}}, \qquad \dist \big( C Q_{J,M}, \Gamma \big) \sim 2^{-J}, \quad J \in \nat,
\end{\eq}
complemented by $\dist (C Q_{0,M}, \Gamma ) \ge c >0$, be a Whitney decomposition of $\Om$ by pairwise disjoint cubes $Q_{J,M}$. This
is an immaterial modification of related Whitney decompositions as described in \cite[pp.\,30--31]{T08} and \cite[p.\,95]{T20}. All 
what follows is justified if $C>2$ is chosen sufficiently large. Let $\vr = \{ \vr_{J,M} \}$ be a related resolution of unity with
\begin{\eq}   \label{4.21}
\supp \vr_{J,M} \subset 2 Q_{J,M}, \qquad \big| D^\gamma \vr_{J,M} (x) \big| \le c_\gamma 2^{J|\gamma|}, \quad x\in \Om, \quad \gamma
\in \nat^n_0,
\end{\eq}
for some $c_\gamma >0$ and
\begin{\eq}   \label{4.22}
\sum^\infty_{J=0} \sum_M \vr_{J,M} (x) =1 \qquad \text{if} \quad x \in \Om.
\end{\eq}
To avoid awkward formulations  we assume tacitly that there are cubes $Q_{J,M}$ with $J=0$ in \eqref{4.20}.

\begin{definition}   \label{D4.3}
Let $\Om$ be an arbitrary domain in $\rn$ with $\Om \not= \rn$. Let
\begin{\eq}   \label{4.23}
0<p,q \le \infty \qquad \text{and} \quad s> \sigma^{n}_p
\end{\eq}
$(q>1$ if $p= \infty)$. Let $\vr = \{ \vr_{J,M} \}$ be the above resolution of unity. Then $F^{s, \rloc}_{p,q} (\Om)$ collects all
$f\in D'(\Om)$ such that
\begin{\eq}   \label{4.24}
\| f \, | F^{s, \rloc}_{p,q} (\Om) \|_{\vr} = \Big( \sum^\infty_{J=0} \sum_M \| \vr_{J,M} f \, | \Fs (\rn) \|^p \Big)^{1/p}
\end{\eq}
is finite $($with the usual modification if $p=\infty)$.
\end{definition}

\begin{remark}   \label{R4.4}
This coincides  with the corresponding part of \cite[Definition 3.13, pp.\,95--96]{T20}. There one finds also the history of these
so--called {\em refined localization spaces} and related references. This covers in  particular the H\"{o}lder--Zygmund spaces
\begin{\eq}   \label{4.25}
\Cc^{s,\rloc} (\Om) = F^{s, \rloc}_{\infty, \infty} (\Om) = B^{s, \rloc}_{\infty, \infty} (\Om), \qquad s>0.
\end{\eq}
In \eqref{4.24} one has to understand $\vr_{J,M} f \in \Fs (\rn)$ as the extension of $\vr_{J,M} f \in D'(\Om)$ by zero outside of \Om.
The above spaces $F^{s, \rloc}_{p,q} (\Om)$ are quasi--Banach spaces (Banach spaces if $p \ge 1$, $q\ge 1$). They are independent of
$\vr = \{\vr_{J,M} \}$ (equivalent quasi--norms), \cite[Theorem 3.15, p.\,96]{T20}.
\end{remark}

One may ask how the spaces introduced in the Definitions \ref{D4.1} and \ref{D4.3} are related to each other. We collect what is known
so far. Let $\Om$ be an arbitrary domain in $\rn$ with $\Om \not= \rn$. Let $d(x) = \dist (x,\Om)$, $\Gamma = \pa \Om$, and $\delta 
(x) = \min \big( d(x), 1 \big)$, $x\in \Om$, be the modified distance of $x$ to $\Gamma$ based on \eqref{4.7}. Let $\sigma^n_{p,q}$ be
as in \eqref{4.16}. 

\begin{theorem}   \label{T4.5}
{\em (i)} Let $\Om$ be an arbitrary domain in $\rn$ with $\Om \not= \rn$. Let
\begin{\eq}   \label{4.26}
0<p,q \le \infty \quad \text{and} \quad s > \sigma^n_{p,q}
\end{\eq}
$($with $q=\infty$ if $p= \infty)$. Then
\begin{\eq}      \label{4.27}
F^{s, \rloc}_{p,q} (\Om) \hra \wt{F}^s_{p,q} (\Om).
\end{\eq}
Furthermore, $f \in F^{s, \rloc}_{p,q} (\Om)$ if, and only if,
\begin{\eq}   \label{4.28}
\| f\, | \Fs (\Om) \| + \| \delta^{-s} f \, | L_p (\Om) \| \sim \| f\, | F^{s, \rloc}_{p,q} (\Om) \|
\end{\eq}
is finite $($equivalent quasi--norms$)$. 
\cm
{\em (ii)} Let $\Om$ be an $E$--thick domain in $\rn$ and let $p,q$ and $s$ be as in \eqref{4.26}. Then
\begin{\eq}    \label{4.29}
F^{s, \rloc}_{p,q} (\Om) = \wt{F}^s_{p,q} (\Om).
\end{\eq}
\end{theorem}

\begin{remark}   \label{R4.6}
The embedding \eqref{4.27} is covered by \cite[(2.110), p.\,45]{T08}. It relies on atomic arguments without moment conditions as in
Theorem \ref{T1.7}(ii). This explains why $s> \sigma^n_p$ in Definition \ref{D4.3} is strengthened by $s> \sigma^n_{p,q}$. The 
coincidence \eqref{4.29} for $E$--thick domains goes back to  \cite[Proposition 3.10, pp.\,77--78]{T08} justified by wavelet arguments.
Let $L_p (\Om, \delta^{-s})$, $0<p \le \infty$, $s>0$, be the related weighted $L_p$--space quasi--normed by
\begin{\eq}   \label{4.30}
\| f \, | L_p (\Om, \delta^{-s}) \| = \Big( \int_{\Om} |f(x)|^p \, \delta^{-sp} (x) \, \di x \Big)^{1/p},
\end{\eq}
obviously modified if $p= \infty$. Then one has according to \cite[Theorem 5.10, p.\,54]{T01} for bounded $C^\infty$ domains $\Om$
in $\rn$ and $p,q,s$ as above that
\begin{\eq}   \label{4.31}
\wt{F}^s_{p,q} (\Om) = \Fs (\Om) \cap L_p (\Om, \delta^{-s} )
\end{\eq}
in the understanding that the space on the right--hand side of \eqref{4.31} is a subspace  of $\Fs (\Om)$ (already continuously 
embedded into some space $L_r (\Om)$ with $1<r \le \infty$), quasi--normed by the left--hand side of \eqref{4.28}. The equivalence 
of the quasi--norms in \eqref{4.28} follows now from \eqref{4.29}. The complicated proof in \cite{T01} relies on characterizations of
the related $F$--spaces in terms of differences. The extension of this assertion to arbitrary domains $\Om$ in $\rn$ in 
\cite[Proposition 24, p.\,54]{Scha14} is based on wavelet arguments. The restriction to $p<\infty$, $q<\infty$ in \cite{Scha14} is
not really needed: The crucial arguments used there are based on corresponding  assertions in \cite{T08} which apply to $p,q,s$ as
in \eqref{4.26} (with $q=\infty$ if $p=\infty$) and can be taken over resulting in
\begin{\eq}   \label{4.32}
F^{s, \rloc}_{p,q} (\Om) = \Fs (\Om) \cap L_p (\Om, \delta^{-s} )
\end{\eq}
for arbitrary domains $\Om$  in $\rn$ and all admitted parameters, again in the understanding as explained in connection with 
\eqref{4.31}.
\end{remark}

\subsection{Main assertions}   \label{S4.2}
Let again $\Om$ be an arbitrary domain (= open set) in $\rn$ with $\Om \not= \rn$. We constructed in \cite[Theorem 2.33, p.\,49]{T08}
an orthonormal wavelet basis of Daubechies type in $L_2 (\Om)$. This representation had been  extended in \cite[Theorem 2.36, 
p.\,53--54]{T08} to $L_p (\Om)$ with $1<p<\infty$ and in \cite[Theorem 2.38, pp.\,54--55]{T08} to $F^{s,\rloc}_{p,q} (\Om)$ with $p,q,s$
as in \eqref{4.26} ($q=\infty$ if $p=\infty$). Also the proof of \eqref{4.29} for $E$--thick domains in \cite[pp.\,77--79]{T08} depends
on wavelet arguments. Basically (in a somewhat hidden way) the arguments rely on the local homogeneity \eqref{4.18} applied to 
\eqref{4.24}. The question arises whether there is a quarkonial counterpart of these observations for the spaces $F^{s,\rloc}_{p,q}
(\Om)$  (where one does not need moment conditions for related atoms and quarks). This is the case. We fix the outcome and outline the
proofs.

Let $\zom$ be the collection of all $(J,M) \in \no \times \zn$ needed in \eqref{4.20}. Let
\begin{\eq}   \label{4.33}
\zOm = \big\{ (j,m) \in \no \times \zn: \ Q_{j,m} \subset Q_{J,M} \ \text{for some $(J,M) \in \zom$} \big\}.
\end{\eq}
In particular $j \ge J$. Let again $\chi_{j,m}$ be the characteristic function of $Q_{j,m} = 2^{-j}m + 2^{-j} (0,1)^n$. Then
$f^s_{p,q} (\Om)^{\vk}$ with $s\in \real$, $\vk \in \real$ and $0<p,q \le \infty$ ($q=\infty$ if $p=\infty$)
is the collection of all sequences
\begin{\eq}    \label{4.34}
\lambda = \big\{ \lambda^\beta: \ \beta \in \nat^n_0 \big\}, \qquad \lambda^\beta = \big\{ \lambda^\beta_{j,m} \in \comp: \ (j,m) \in
\zOm \big\}
\end{\eq}
such that
\begin{\eq}   \label{4.35}
\| \lambda \, | f^s_{p,q} (\Om)^{\vk} \| = \sup_{\beta \in \nat^n_0} 2^{\vk |\beta|} \Big\| \Big( \sum_{(j,m) \in \zOm} 2^{jsq}
\big| \lambda^\beta_{j,m} \chi_{j,m} (\cdot) \big|^q \Big)^{1/q} \, \big| L_p (\Om) \Big\|
\end{\eq}
is finite. This is the $\Om$--version of Definition \ref{D2.2} with $f^s_{\infty, \infty} (\rn)^{\vk} = b^s_{\infty, \infty}(\rn)^{\vk}
$, now incorporated in the $F$--scale.  If $C>2$ in \eqref{4.20} with $\Gamma = \pa \Om$ is chosen 
sufficiently large then all what follows is justified. If $0<p<\infty$, $0<q\le \infty$ then $\chi_{j,m}$ in \eqref{4.35} can be 
replaced by smaller characteristic functions as formulated in \cite[Proposition 1.33(i), p.\,19]{T06} with a reference to \cite{FrJ90}.
Let again $d(x) = \dist (x, \Gamma)$, $x\in \Om$, $\Gamma = \pa \Om$, and
\begin{\eq}   \label{4.36}
d_{j,m} = \min \big( 1, d(2^{-j}m) \big), \qquad (j,m) \in \zOm.
\end{\eq}
Let again $k^\beta_{j,m}$ and $\Phi^\beta_{j,m}$ be the quarkonial building blocks as introduced in \eqref{2.4}, \eqref{2.11}. Let 
$\ve >0$ be as in \eqref{2.15} and Proposition \ref{P2.1}. As usual, weak$^*$--convergence, or simply convergence, of $f_l \to f$ in
$D'(\Om)$  if $l \to \infty$ means
\begin{\eq}   \label{4.37}
f_l (\vp) \to f(\vp) \qquad \text{for any} \quad \vp \in D(\Om).
\end{\eq}
We ask for an $\Om$--version of Theorem \ref{T2.4} for the related spaces $F^{s, \rloc}_{p,q} (\Om)$ according to Definition 
\ref{D4.3}.

\begin{theorem} \label{T4.7}
Let $\Om$ be an arbitrary domain in $\rn$ with $\Om \not= \rn$. Let
\begin{\eq}   \label{4.38}
0<p,q \le \infty \qquad \text{and} \quad s > \sigma^{n}_{p,q}
\end{\eq}
$(q=\infty$ if $p=\infty)$. Let $\vk > -\ve$. Then $f\in D'(\Om)$  belongs to $\Fsr (\Om)$ if, and only if, it can be represented as
\begin{\eq}   \label{4.39}
f = \sum_{\beta \in \nat^n_0, (j,m) \in \zOm} \lambda^\beta_{j,m} \, k^\beta_{j,m}, \qquad \lambda \in f^s_{p,q} (\Om)^{\vk},
\end{\eq}
unconditional convergence being in $D'(\Om)$. Furthermore,
\begin{\eq}   \label{4.40}
\| f \, | \Fsr (\Om) \| \sim \inf \| \lambda \, | f^s_{p,q} (\Om)^{\vk} \|
\end{\eq}
where the infimum is taken over all admissible representations \eqref{4.39}. In addition there are two numbers $c_1 >0$, $c_2 >0$ and
real functions $\Psi^\beta_{j,m} \in D(\Om)$ with
\begin{\eq}   \label{4.41}
\supp \Psi^\beta_{j,m} \subset \big\{ y \in \Om: \ |2^{-j}m -y | \le c_1 d_{j,m}, \ d(y) \ge c_2 d_{j,m} \big\}
\end{\eq}
such that
\begin{\eq}   \label{4.42}
f = \sum_{\beta \in \nat^n_0, (j,m) \in \zOm} \lambda^\beta_{j,m} (f) \, k^\beta_{j,m},
\end{\eq}
unconditional convergence being in $D'(\Om)$, with
\begin{\eq}   \label{4.43}
\| f \, | \Fsr (\Om) \| \sim \| \lambda (f) \, | f^s_{p,q} (\Om)^{\vk} \|
\end{\eq}
where
\begin{\eq}   \label{4.44}
\lambda^\beta_{j,m} (f) = 2^{jn} \, \big( f, \Psi^\beta_{j,m} \big), \qquad f \in D'(\Om).
\end{\eq}
\end{theorem}

\begin{proof}
{\em Step 1.} Compared with the related arguments in \cite[Theorem 2.38, p.\,54--55]{T08} 
as far as the above indicated wavelet expansions  are concerned there is essentially one new aspect. Let
\begin{\eq}   \label{4.45}
f \in \Fs (\rn) \qquad \text{with} \quad \supp f \subset K = \{ x \in \rn: \ |x| \le c \}
\end{\eq}
for some $c>0$ and with $p,q,s$ as above. Then one can expand $f$ according to \eqref{2.32} as
\begin{\eq}   \label{4.46}
f = \sum_{\substack{\beta \in \nat^n_0, j\in \no, \\ m \in \zn}} 2^{jn} \big(f, \Phi^\beta_{j,m} \big) \, k^\beta_{j,m}
\end{\eq}
where $\Phi^\beta_{j,m}$ are the real analytic functions in \eqref{2.11}, \eqref{2.12}. They have no compact support (in contrast to
$k^\beta_{j,m}$) and it seems to be necessary to deal with the full sum. Multiplication of $f$ with a cut--off function modifies
$k^\beta_{j,m}$ what we wish to avoid. But as already mentioned in Remark \ref{R2.8} the two systems in \eqref{2.13} and \eqref{2.14}
are totally independent. In particular it follows from the related proof in \cite{T06}, especially \cite[(3.110), (3.111), 
p.\,167]{T06}, and \eqref{2.1}--\eqref{2.4} that one needs in \eqref{4.46} only terms with
\begin{\eq}   \label{4.47}
\supp k^\beta_{j,m} \cap \supp f \not= \emptyset.
\end{\eq}
Then one has a local version of Theorem \ref{T2.4} for compactly supported $f$ with \eqref{4.45} (as in the case of compactly supported
wavelets).
\cm
{\em Step 2.} After the indicated localization one can now argue as in the case of wavelets based on \eqref{4.24}: reduction to the
standard situation using \eqref{4.18} and clipping together the pieces adapted to \eqref{4.35}. As before it is sufficient
to deal with $\beta =0$, having in mind that $\Fs (\rn)$ is a $u$-Banach space with $u = \min (1,p,q)$, and that terms with $\beta 
\not= 0$ decay exponentially with respect to $|\beta|$. For this purpose one applies \eqref{4.18} to $\vr_{J,M} (2^{-J} x) f(2^{-J}x)$
combined with suitable translations such that Step 1 can be used. This means, neglecting translations, 
\begin{\eq}   \label{4.48}
\vr_{J,M} (2^{-J} x) f(2^{-J}x) = \sum_{\substack{\beta \in \nat^n_0, j\in \no, \\ m \in \zn}} 2^{jn} \Big( f(2^{-J} \cdot), 
\vr_{J,M} (2^{-J} \cdot) \Phi^\beta_{j,m}(\cdot) \Big) \, k^\beta_{j,m} (x)
\end{\eq}
where the summation can be restricted to terms satisfying the counterpart of \eqref{4.47} with the left--hand side of \eqref{4.48}
in place of $f$.
By \eqref{2.4} one has $k^\beta_{j,m} (2^J x) = k^\beta_{j+J,m} (x)$ and by \eqref{2.11} a related counterpart for $\Phi^\beta_{j,m}
(2^Jx)$.
Re--transformation produces the corresponding terms in 
\eqref{4.42} with $j \ge J \in \no$. 
Everything is local. A fixed $k^\beta_{j,m} (x)$ in \eqref{4.42} may occur several times (finitely many, independently
of $j$, $m$ and $\beta$). This results in the functions $\Psi^\beta_{j,m}$ in \eqref{4.44} with the indicated properties in 
\eqref{4.41}. The rest is now a matter of scaling based on \eqref{4.18}, \eqref{4.24}, and \eqref{4.35}. But these technicalities are
the same as in \cite{T08}.
\end{proof}

\begin{remark}   \label{R4.8}
The functions $\Psi^\beta_{j,m}$ in \eqref{4.41}--\eqref{4.44} originate from the functions $\vr_{J,M} (2^{-J} \cdot)\, \Phi^\beta_{j,m} (\cdot)$
in \eqref{4.48} and their re--transformations. This shows that, in addition to \eqref{4.41}, the crucial observation \eqref{2.17} in
Proposition \ref{P2.1}(ii) has the  following  local counterpart:
\begin{itemize}
\item[] {\em Let $\vk \in \real$ and $K\in \nat$. Then there is a number $C>0$ such that}
\begin{\eq}   \label{4.49}
\big| 2^{\vk |\beta|} D^\alpha \Psi^\beta_{j,m} (x) \big| \le C \, 2^{j |\alpha|}, \qquad x\in \Om,
\end{\eq}
{\em for all $\beta \in \nat^n_0$, all $(j,m) \in \zOm$ and all $\alpha \in \nat^n_0$ with $|\alpha| \le K$.}
\end{itemize}
This makes again clear that $\beta =0$ is the dominant term, especially in the context of estimates. There is a second observation
which makes clear that the kernels $\Psi^\beta_{j,m}$ of the local means in \eqref{4.44} fit in the scheme of what can be expected.
As has been observed  in \cite{Skr98} one can relax the moment conditions in the atomic representation Theorem \ref{T1.7}, which
coincides with \cite[Theorem 1.7, p.\,5]{T08}, by \eqref{2.125} if $L\in \no$ is chosen sufficiently large in dependence  of 
$\As (\rn)$. We described in Proposition  \ref{P2.21} a further refinement going back to \cite{Scha13}. This has been used several
times, especially in connection with weighted spaces, where we relied on \eqref{2.142}. On the other hand, the optimal quarkonial
representations  of the function spaces $\As (\rn)$ in Theorem \ref{T2.4} are based on the local means \eqref{2.30} where the kernels
$\Phi^\beta_{j,m}$ have sufficiently many moment conditions  such that \cite[Theorem 1.15, pp.\,7--8]{T08} can be applied. But 
according to \cite[Corollary 3.15, p.\,293]{Scha13} one relax also the moment condition for kernels of local means in \cite[Theorem
1.15, pp.\,7--8]{T08} in a similar way as in \eqref{2.125} for atomic representations. This applies in particular to the kernels
$\Psi^\beta_{j,m}$ in \eqref{4.44}, which originate from \eqref{4.48}, which means, roughly speaking, from $\vr_{J,M} (x) 
\Phi^\beta_{j,m} (x)$ with $j \ge J$ (finite linear combinations) and what can be justified  in the same was as in 
\eqref{3.20}--\eqref{3.22}. In other words, the kernels $\Psi^\beta_{j,m}$ in \eqref{4.44} fit in the scheme of local means for the
spaces $\As (\rn)$ based, now, on modified moment conditions of type \eqref{2.125}. This is also in good agreement with corresponding
wavelet representations for the spaces $\Fsr (\Om)$ according to  \cite[Theorem 2.23, p.\,44]{T08}  including related peculiarities 
for the starting terms with $j=J$ near the boundary, where no (modified) moment conditions are requested.
\end{remark}

With $\wt{F}^s_{p,q} (\Om)$ as in Definition \ref{D4.1}(ii) one has \eqref{4.29} if $\Om$ is an $E$--thick domain in $\rn$ with $p,q,s$
as in \eqref{4.26}. This covers, as already
said, bounded Lipschitz domains in \rn, but also snowflake domains in $\real^2$. Then one can apply Theorem \ref{T4.7}. We fix the outcome.

\begin{corollary}   \label{C4.9}
Let $\Om$ be an $E$--thick domain in $\rn$. Let
\begin{\eq}   \label{4.50}
0<p,q \le \infty \quad \text{and} \quad s  > \sigma^{n}_{p,q}
\end{\eq}
with $q=\infty$ if $p=\infty$. Then Theorem \ref{T4.7} with $\wt{F}^s_{p,q} (\Om)$ in place of $\Fsr (\Om)$ remains valid.
\end{corollary}

\begin{proof} 
This follows from the above comments.
\end{proof}

\subsection{Universal representations}    \label{S4.3}
Corollary \ref{C4.9}, based on Theorem \ref{T4.7}, is the natural counterpart of the Theorems \ref{T2.4} and \ref{T2.18} dealing with
quarkonial expansions of the spaces $\Fs (\rn)$, where
\begin{\eq}   \label{4.51}
0<p,q \le \infty, \qquad s > \sigma^{n}_{p,q},
\end{\eq}
with $q= \infty$ if $p= \infty$, and their weighted generalizations. Duality paved the way to quarkonial expansions for spaces with $s<0$ in Section \ref{S2.3} and finally to universal representations as described in Section \ref{S2.6}. One may ask to which extent this
theory has a natural counterpart for spaces on domains. We outline how one can find corresponding assertions. But we do not go into
the details. Just on the contrary. We collect some cornerstones of a related theory, but we do not glue them together.

\subsubsection{Basic spaces}    \label{S4.3.1}
The representations in Section \ref{S2.6} in the context of the dual pairing $\big( S(\rn), S'(\rn) \big)
$ rely on \eqref{2.172}, \eqref{2.173}. It is quite clear that $\big( D(\Om), D' (\Om) \big)$ is not a substitute for this purpose. Let $\Om$ be a bounded
domain (= open set) in \rn. Let again $d(x) = \dist (x,\Gamma)$ be the distance of $x\in \Om$ and $\Gamma = \pa \Om$ based on 
\eqref{4.7}. Let $\vr \in C^\infty (\Om)$ with $\vr^{-1}
(x) \sim d(x)$ and
\begin{\eq}   \label{4.52}
| D^\gamma \vr (x) | \le c_\gamma \, \vr^{1+|\gamma|} (x) \qquad \text{for} \quad x \in \Om, \quad \text{and} \quad
\gamma \in \nat^n_0,
\end{\eq}
(smoothed distance to the boundary). Then the locally convex space $S(\Om)$ consisting of all complex--valued infinitely differentiable
functions $f$ on $\Om$ such that
\begin{\eq}   \label{4.53}
\sup_{x\in \Om} \vr^l (x) \, |D^\alpha f(x) | <\infty, \qquad l\in \no, \quad \alpha \in \nat^n_0,
\end{\eq}
is a better adapted counterpart of $S(\rn)$ for these purposes. According to \cite[Lemma 6.2.3, p.\,409]{T78} the space $S(\Om)
$ coincides with the locally convex subspace
\begin{\eq}   \label{4.54}
S(\Om) = \big\{ f \in S(\rn): \ \supp f \subset \ol{\Om} \big\}
\end{\eq}
of $S(\rn)$. Then it is quite clear what is meant by $S'(\Om)$. Assuming that one has already a satisfactory representation theory for
the unweighted spaces $\As (\Om)$ then one could try to extend this theory to their weighted counterparts as in Section \ref{S2.4}
and to $\big( S(\Om), S'(\Om) \big)$ as in Section \ref{S2.6}. But this will not be done here. We only mention that the (unweighted)
refined localization spaces $\Fsr (\Om)$ as introduced in Definition \ref{D4.3} have a natural weighted counterpart $\Fsr (\Om,
\vr^{\vk})$, $\vk \in \real$, quasi--normed by
\begin{\eq}   \label{4.55}
\Big( \sum^\infty_{J=0} \sum_M \| \vr^{\vk} \, \vr_{J,M} f \, | \Fs (\rn) \|^p \Big)^{1/p} \sim
\Big( \sum^\infty_{J=0} \sum_M  2^{J \vk p}\, \| \vr_{J,M} f \, | \Fs (\rn) \|^p \Big)^{1/p}.
\end{\eq}
If $\Om$ is a bounded $E$--thick domain in $\rn$ (including bounded Lipschitz domains) and if $p,q,s$ are as in \eqref{4.26} ($q=\infty
$ if $p=\infty$)
 then it follows from \eqref{4.29} that
\begin{\eq}   \label{4.56}
\| f \, | \Fsr (\Om, \vr^{\vk} ) \| \sim \| \vr^{\vk} f \, | \wt{F}^s_{p,q} (\Om) \|,
\qquad \vr^{\vk} f \in \wt{F}^s_{p,q} (\Om).
\end{\eq}
Similarly one obtains from \eqref{4.32} for arbitrary bounded domains $\Om$ and with $p,q,s$ as in \eqref{4.26} ($q=\infty
$ if $p=\infty$) that
\begin{\eq}   \label{4.57}
\| f \, | \Fsr (\Om, \vr^{\vk} ) \| \sim \| \vr^{\vk} f \, | \Fs (\Om) \| + \| \vr^{s+\vk} f \, | L_p (\Om) \|.
\end{\eq}
But we do not try to extend what will be outlined below for the unweighted spaces in $\Om$ to these indicated weighted spaces.

We assumed for sake of simplicity that the domain $\Om$ is bounded. If $\Om$ with $\Om \not= \rn$ is unbounded, with the half--space
\begin{\eq}   \label{4.58}
\real^n_+ = \{ x \in \rn: \ x= (x', x_n), \ x' \in \Rn, \ x_n >0 \},
\end{\eq}
$2 \le n \in \nat$, as the most prominent example, one has to combine the above constructions with the related arguments in the 
Sections \ref{S2.4}--\ref{S2.6}. But this is a technical matter which does not cause additional complications.

\subsubsection{Duality}   \label{S4.3.2} 
Representations for spaces $\As (\rn)$ with $s<0$ as considered in Section \ref{S2.3} rely on corresponding
assertions for spaces with positive smoothness and duality in the framework of $\big( S(\rn), S' (\rn) \big)$. Let $\Om$ be an
$E$--thick domain in \rn. Let $\As (\Om)$ and $\wt{A}^s_{p,q} (\Om)$ be the spaces as introduced in Definition \ref{D4.1} assuming
now
\begin{\eq}   \label{4.59}
s> \sigma^{n}_p \ \text{for $B$--spaces}, \qquad s>\sigma^{n}_{p,q} \ \text{for $F$--spaces}
\end{\eq}
with $\sigma^{n}_p$ and $\sigma^{n}_{p,q}$ as in \eqref{4.16}. If, in addition, $p<\infty$, $q<\infty$ then it follows from the
wavelet characterizations according to \cite[Theorem 3.13, pp.\,80--81]{T08} that $D(\Om)$ is dense in the spaces $\wt{A}^s_{p,q}
(\Om)$ with \eqref{4.59} and also dense
$\As (\Om)$ with $0<p,q <\infty$ and $s<0$. In particular it makes sense to ask for dual spaces in the framework of the dual pairing
$\big( D(\Om), D'(\Om) \big)$. Let
\begin{\eq}   \label{4.60}
s>0 \quad \text{and} \quad 1 \le p,q <\infty, \qquad \frac{1}{p} + \frac{1}{p'} = \frac{1}{q} + \frac{1}{q'} =1.
\end{\eq}
Then 
\begin{\eq}   \label{4.61}
\wt{B}^s_{p,q} (\Om)' = B^{-s}_{p', q'} (\Om)
\end{\eq}
is covered by \cite[Theorem 3.30, pp.\,97--98]{T08} where we gave a detailed wavelet proof of this assertion. Although not stated
explicitly there one may assume that
\begin{\eq}   \label{4.62}
\wt{F}^s_{p,q} (\Om)' = F^{-s}_{p',q'} (\Om)
\end{\eq}
for
\begin{\eq}   \label{4.62a}
s>0, \quad 1<p<\infty, \quad 1\le q <\infty \quad \text{and} \quad \frac{1}{p} + \frac{1}{p'} = \frac{1}{q} + \frac{1}{q'} =1.
\end{\eq}
These are the $\Om$--counterparts of the well--known duality assertions for the related spaces $\As (\rn)$ according to \cite[Theorem
2.11.2, p.\,178]{T83}. As far as the duality \eqref{4.60} and \eqref{4.61} with $H^s_p = F^s_{p,2}$ is concerned we refer also to
\cite[Theorem 2.10.5/1, p.\,235, Theorem 4.8.1, p.\,332]{T78}.

\subsubsection{Representability}  \label{S4.3.3} 
The quarkonial representations of the spaces $\As (\rn)$
in Section \ref{S2.3} for all $s <0$ and $0<p,q \le
\infty$ ($p<\infty$ for $F$--spaces) rely on corresponding assertions for spaces with positive smoothness and duality. As for an 
$\Om$--counterpart one may ask for a substitute of Proposition \ref{P2.9}, which means the representability of $f\in \Cc^s 
(\Om)$ where again
\begin{\eq}    \label{4.63}
\Cc^s (\Om) = B^s_{\infty, \infty} (\Om), \qquad s<0.
\end{\eq}
Let $\Om$ be an $E$--thick domain in $\rn$ as described in Section \ref{S4.1}, including special and bounded Lipschitz domains in \rn,
snowflake domains in $\real^2$ and also the half--spaces $\rnp$ according to \eqref{4.58}. Then one has by \eqref{4.60}, \eqref{4.61},
\begin{\eq}   \label{4.64}
\Cc^s (\Om) = \wt{B}^{-s}_{1,1} (\Om)', \qquad s<0.
\end{\eq}
This is the counterpart of \eqref{2.53}. Now one can argue as in Proposition \ref{P2.9} and its proof. By Theorem \ref{T4.7} and
Corollary \ref{C4.9} one has
\begin{\eq}   \label{4.65}
g = \sum_{\beta \in \nat^n_0, (j,m) \in \zOm} 2^{jn} \, \big( f, \Psi^\beta_{j,m} \big) \, k^\beta_{j,m}, \qquad g \in 
\wt{B}^{-s}_{1,1} (\Om) = \wt{F}^{-s}_{1,1} (\Om)
\end{\eq}
as the counterpart of \eqref{2.57}. One can apply the duality argument \eqref{2.58}.  Then one obtains the following assertion:
\begin{itemize}
\item[] {\em Let $\Om$ be an $E$--thick domain in \rn, $\Om \not= \rn$. Then any $f\in \Cc^s (\Om)$, $s<0$, can be represented in the weak$^*$--topology of $\Cc^s (\Om)$ as}
\begin{\eq}   \label{4.66}
f = \sum_{\beta \in \nat^n_0, (j,m) \in \zOm} 2^{jn} \, \big( f, k^\beta_{j,m} \big) \, \Psi^\beta_{j,m}.
\end{\eq}
\end{itemize}
Recall that both $k^\beta_{j,m} \in D(\Om)$ and $\Psi^\beta_{j,m} \in D(\Om)$. It follows by embedding that any $f\in \As (\Om)$
with $A \in \{B,F \}$, $s\in \real$ and $0<p,q \le \infty$ ($p<\infty$ for $F$--spaces) can be represented according to \eqref{4.66}.
Furthermore as discussed in Remark \ref{R4.8} the function $\Psi^\beta_{j,m}$ may not have the desired moment conditions as requested
for atoms in spaces $\As (\rn)$ with $s<0$, but they have modified moment conditions of type \eqref{2.125}. This shows that expansions 
of type  \eqref{4.66} fit in the scheme of atomic representations of spaces $\As$ with $s<0$. We return to this point in connection
with \eqref{4.67} below. This applies even more to the non--negative kernels $k^\beta_{j,m}$ according to \eqref{2.3}, \eqref{2.4} of
the local means $2^{jn} \big( f, k^\beta_{j,m} \big)$ in the spaces $\As (\rn)$ with $s<0$ as considered in \cite[Theorem 1.15,
p.\,7]{T08}.

\subsubsection{Universal representations, I}    \label{S4.3.4}
The question arises to which extent the (so far unweighted) spaces $\As (\Om)$ can be
characterized in terms of the expansion \eqref{4.66}. If $s<0$ then one asks for the $\Om$--counterpart of Theorem \ref{T2.10}. Now
one has to rely on $f^s_{p,q} (\Om)^{\vk}$ as introduced in \eqref{4.33}--\eqref{4.35} instead of $f^s_{p,q} (\rn)^{\vk}$ in 
\eqref{2.25}. This must
be complemented by corresponding spaces $b^s_{p,q} (\Om)^{\vk}$. The building blocks $\Phi^\beta_{j,m}$ in 
Theorem \ref{T2.10} have moment conditions of all orders. This is not guaranteed for the building blocks $\Psi^\beta_{j,m}$ in
\eqref{4.66} and in its  related generalizations
\begin{\eq}  \label{4.67}
f = \sum_{\beta \in \nat^n_0, (j,m) \in \zOm} \lambda^\beta_{j,m} \, \Psi^\beta_{j,m}.
\end{\eq}
However as discussed after \eqref{4.66} the building blocks $\Psi^\beta_{j,m}$ have modified moment conditions, at least if $j>J$ in
$(j,m) \in \zOm$ according to \eqref{4.33}. This might be not the case for the related starting terms $(j,m) \in \zOm$ with $j=J$ in
\eqref{4.33}. But this is also not necessary. 
One does not need moment conditions for atoms near the boundary of $E$--thick domains. The corresponding arguments may be
found in \cite[p.\,78]{T08} with a reference to \cite{TrW96}. It is just the basis to prove \eqref{4.29} for $E$--thick domains in
\cite[Proposition 3.10,  pp.\,77--78]{T08}. But it remains to be seen if these instruments can be used to establish the desired
quarkonial representations.

\subsubsection{Universal representations, II}  \label{S4.3.5}
The representations \eqref{4.66}, \eqref{4.67} are universal for all spaces $\As (\Om)$
with $A \in \{B,F \}$, $s\in \real$ and $0<p,q \le \infty$ ($p<\infty$ for $F$--spaces). But characterizations \eqref{4.67} with
$\lambda \in a^s_{p,q} (\Om)^{\vk}$ can only be expected if $s<0$. However there is the following $\Om$--counterpart of \eqref{2.77}
underlying Corollary \ref{C2.13}. Let $\Om$ be a bounded Lipschitz domain in \rn. Then it is also a bounded $E$--thick domain as has
already been mentioned in Section \ref{S4.1}. Let 
\begin{\eq}   \label{4.68}
A \in \{B,F \}, \quad s \in \real \quad \text{and} \quad 0<p,q \le \infty
\end{\eq}
($p<\infty$ for $F$--spaces). Let $N \in \no$. Then
\begin{\eq}   \label{4.69}
\| f \, | \As(\Om) \| \sim \sup_{0 \le |\alpha| \le N} \| D^\alpha f \, | A^{s-N}_{p,q} (\Om) \|,
\end{\eq}
(equivalent quasi--norms). This is covered by \cite[Proposition 4.21, p.\,113]{T08}.
Now one is in the same situation as in Corollary \ref{C2.13}: Possible quarkonial characterizations of
spaces $\As (\Om)$ with $s<0$, based on \eqref{4.66}, and described in the above Section \ref{S4.3.4}
can be extended to all $s\in \real$ (at
the expense of constrained conditions of type \eqref{2.80} with $\Om$ in place of $\rn$).

\subsubsection{Universal representations, III}  \label{S4.3.6}
So far we described possible quarkonial  representations for the unweighted spaces 
$\As (\Om)$. One may ask for a weighted counterpart as indicated in Section \ref{S4.3.1} above resulting in the $\Om$--version of Section \ref{S2.6}.

\subsection{Positivity}   \label{S4.4}
We dealt in Section \ref{S3.3} with the so--called positivity property of the spaces $\As (\rn)$. Some assertions obtained there can be transferred to related spaces $\As (\Om)$ as introduced in Definition \ref{D4.1} where $\Om$ is an arbitrary domain (=
open set) in \rn. Similarly as in Section \ref{S3.3} we say that $f\in D'(\Om)$ is real if $f(\vp) \in \real$ for any real $\vp \in
D(\Om)$. Furthermore $f\in D'(\Om)$ is said to be positive, $f \ge 0$, if $f(\vp) \ge 0$ for any $\vp \in D(\Om)$ with $\vp(x) \ge 0$
for all $x\in \Om$. An obvious modification of \eqref{3.56}, \eqref{3.57} shows that $f$ is real if $f$ is positive.

\begin{definition}   \label{D4.10}
Let $\Om$ be an arbitrary domain $($= open set$)$ in \rn. The space $\As (\Om)$ with $A \in \{B,F \}$, $s\in \real$ and $0<p,q \le
\infty$ is said to have the positivity property if any real $f\in \As (\Om)$ can be decomposed in
\begin{\eq}   \label{4.70}
f= f^1 - f^2 \quad \text{with} \quad f^l \ge 0, \quad f^l \in \As (\Om) \quad \text{for} \quad l =1,2,
\end{\eq}
and
\begin{\eq}   \label{4.71}
\| f \, |\As (\Om) \| \sim \| f^1 \, |\As (\Om) \| + \| f^2 \, |\As (\Om) \|
\end{\eq}
with equivalence constants which are independent of $f$.
\end{definition}

\begin{remark}  \label{R4.11}
This is the counterpart of Definition \ref{D3.8} (including $F^s_{\infty,q} (\Om)$).
\end{remark}

The following assertion extends Proposition \ref{P3.10} from $\rn$ to \Om.

\begin{proposition}    \label{P4.12}
Let $\Om$ be an arbitrary domain $($=open set$)$ in \rn. Then $\As (\Om)$ has the positivity property if $\As (\rn)$ has the positivity
property. 
\end{proposition}

\begin{proof}  
Let $f\in \As (\Om)$ be real and let $f= g|\Om$ with $g\in \As (\rn)$ as in \eqref{4.9}, \eqref{4.10}. According to \eqref{3.66} one can decompose $g\in \As (\rn)$ as
\begin{\eq}   \label{4.72}
g= g^1 + i g^2 \qquad \text{where both $g^1$ and $g^2$ are real}.
\end{\eq}
Then it follows from
\begin{\eq}     \label{4.73}
f(\vp) = g^1 (\vp) + i g^2 (\vp), \qquad \vp \in D(\Om),
\end{\eq}
that $g^2 (\vp) =0$ and $g^2 | \Om =0$. A suitable counterpart of \eqref{3.67} shows that the infimum in \eqref{4.10} can be restricted
to real $g$. The assumed decomposition of $g$ according to Definition \ref{D3.8} and its restriction to $\Om$ prove the above 
proposition.
\end{proof}

There is a converse of the above proposition, at least for the spaces $\Fs (\Om)$, including again $F^s_{\infty, q} (\Om)$.

\begin{corollary}   \label{C4.12}
If $\Fs (\omega)$ has the positivity property for some domain $\omega$ in $\rn$, $s\in \real$ and $0<p,q \le \infty$. Then $\Fs (\Om)$
has the positivity property for all domains $\Om$ in \rn, $\Om \not= \rn$, and for $\Fs (\rn)$.
\end{corollary}

\begin{proof}
By the above proposition and its proof  it is sufficient to show that $\Fs (\rn)$ has the positivity property if $\Fs (Q)$ with $Q=
(0,1)^n$ has the positivity property. Let
\begin{\eq}    \label{4.74}
\Om = \bigcup^\infty_{j=1} Q_j \quad \text{with} \quad Q_j = M_j + (0,1)^n, \quad M_j \in \zn, \quad j\in \nat,
\end{\eq}
such that
\begin{\eq}   \label{4.75}
\dist (Q_j, Q_k) \ge 1 \qquad \text{if} \quad j\not= k.
\end{\eq}
Then it follows from Proposition \ref{P4.12}, its proof, elementary pointwise multiplier assertions and the localization property
\eqref{4.15} that any
\begin{\eq}   \label{4.76}
f\in \Fs (\rn) \quad \text{real} \quad \text{with} \quad \supp f \subset \Om
\end{\eq}
admits the desired decomposition. One can represent $\rn$ as the finite union of domains of type \eqref{4.74}. This can be combined
with a suitable finite resolution of unity by non--negative smooth functions. This proves the corollary.
\end{proof}

\begin{remark}   \label{R4.13}
This argument cannot be extended immediately to related spaces $\Bs (\Om)$. There is no counterpart of \eqref{4.15} if $p \not= q$. 
But as discussed in Remark \ref{R3.11} the restriction $s> \sigma^n_p$ in \eqref{3.61} is quite natural. This is not so clear for its
counterpart $s> \sigma^n_{p,q}$ for the $F$--spaces.
\end{remark}

One can ask whether the positivity property of $\As (\rn)$ or of $\As (\Om)$ can be passed to suitable subspaces. This question
applies in particular to the spaces $\wt{F}^s_{p,q} (\Om)$ according to Definition \ref{D4.1}. This is the case if $\Om$ is an
$E$--thick domain in \rn. It follows from a corresponding assertion  for $\Fsr (\Om)$ and
\begin{\eq}   \label{4.77}
\Fsr (\Om) = \wt{F}^s_{p,q} (\Om) \qquad \text{for} \quad 0<p,q \le \infty \quad \text{and} \quad s> \sigma^n_{p,q}
\end{\eq}
(with $q=\infty$ if $p= \infty$) according to Theorem \ref{T4.5}(ii). Let again
\begin{\eq}   \label{4.78}
\sigma^n_{p,q} = n \Big( \max \big( \frac{1}{p}, \frac{1}{q}, 1 \big) - 1 \Big),
\end{\eq}

\begin{corollary}   \label{C4.15}
Let $\Om$ be an arbitrary domain in \rn, $\Om \not= \rn$. Let $\Fsr (\Om)$ be the spaces as introduced in Definition \ref{D4.3}. Let
\begin{\eq}   \label{4.79}
0<p,q \le \infty \quad \text{and} \quad s> \sigma^n_{p,q}
\end{\eq}
$($with $q=\infty$ if $p= \infty)$. Then $\Fsr (\Om)$ has the positivity property $($in obvious modification of Definition \ref{D4.10}$)$.
\end{corollary}

\begin{proof}
We rely on Theorem \ref{T4.7}. According to \eqref{2.12} and Remark \ref{R4.8} the functions $\Psi^\beta_{j,m}$ in \eqref{4.44} are
real. But then one can argue as in Proposition \ref{P3.10} using that the functions $k^\beta_{j,m}$ are non--negative.
\end{proof}

\end{document}